\documentclass[11pt]{article}

\usepackage{latexsym}
\usepackage{amssymb}
\usepackage{amsthm}
\usepackage{amscd}
\usepackage{amsmath}
\usepackage{mathrsfs}
\usepackage{graphicx}
\usepackage{hyperref} 
\usepackage{tabmac}

\usepackage{shuffle}
\usepackage[all]{xy}
\input xy \xyoption{frame}

\usepackage[colorinlistoftodos]{todonotes}

\theoremstyle{definition}
\newtheorem* {theorem*}{Theorem}
\newtheorem* {conjecture*}{Conjecture}
\newtheorem{theorem}{Theorem}[section]

\theoremstyle{definition}
\newtheorem{observation}[theorem]{Observation}

\newtheorem* {example*}{Example}

\newtheorem{lemma}[theorem]{Lemma}
\theoremstyle{definition}
\newtheorem{definition}[theorem]{Definition}
\theoremstyle{definition}

\newtheorem{conjecture}[theorem]{Conjecture}
\newtheorem{proposition}[theorem]{Proposition}
\newtheorem{corollary}[theorem]{Corollary}

\newtheorem *{remark}{Remark}
\theoremstyle{definition}
\newtheorem {example}[theorem]{Example}
\theoremstyle{definition}

\theoremstyle{definition}

\theoremstyle{definition}
\newtheorem* {remarks}{Remarks}
\xyoption{dvips}

\usepackage{fullpage}
 
\numberwithin{equation}{section}

\def\({\left(}
\def\){\right)}
     \newcommand{\CC}{\mathbb{C}}     \newcommand{\cP}{\mathcal{P}} 
   
\newcommand{\cR}{\mathcal{R}}   \newcommand{\cI}{\mathcal{I}}

\newcommand{\cZ}{\mathcal{Z}}
\def\NN{\mathbb{N}}
    
\def\CC{\mathbb{C}}
\def\RR{\mathbb{R}}
\def\ZZ{\mathbb{Z}} \def\Aut{\mathrm{Aut}}  
            
  \def\wt{\widetilde}

   \newcommand{\supp}{\mathrm{supp}}

\def\barr{\begin{array}}
\def\earr{\end{array}}
\def\ba{\begin{aligned}}
\def\ea{\end{aligned}}
\def\be{\begin{equation}}
\def\ee{\end{equation}}

\def\cX{\mathcal{X}}
\def\cY{\mathcal{Y}}

\def\Cyc{\mathrm{Cyc}}
\def\Fix{\mathrm{Fix}}

\def\qquand{\qquad\text{and}\qquad}
\def\quand{\quad\text{and}\quad}
\def\qquord{\qquad\text{or}\qquad}

\def\cW{\mathcal{W}}

\def\inv{\mathrm{inv}}
\def\ainv{\mathrm{inv}_\cA}

\def\I{\mathcal{I}}

\def\DesR{\mathrm{Des}_R}
\def\DesL{\mathrm{Des}_L}

\def\omdef{\overset{\mathrm{def}}}

\def\hs{\hspace{0.5mm}}

\def\id{\mathrm{id}}
\def\PP{\mathbb{P}}

\def\ben{\begin{enumerate}}
\def\een{\end{enumerate}}

\def\cE{\mathcal E}

\def\op{\mathrm{op}}
\newcommand{\Mon}[1]{\mathscr{F}(#1)}

\def\hs{\hspace{0.5mm}}
\def\simstar{{\hs\hs\sim_{\I_*}\hs\hs}}

\def\fpf{{\tt {FPF}}}

\def\ellhat{\hat\ell}

\def\Ifpf{\I_\fpf}

\def\x{\textbf{x}}
\def\y{\textbf{y}}
\def\z{\textbf{z}}
\def\a{\textbf{a}}
\def\b{\textbf{b}}
\def\c{\textbf{c}}
\def\e{\textbf{e}}

\newcommand{\wfpf}[1]{w_\fpf}

\newcommand{\xRightarrow}[2][]{\ext@arrow 0359\Rightarrowfill@{#1}{#2}}

\newcommand{\Fl}{\operatorname{Fl}}
\newcommand{\std}{\operatorname{std}}

\renewcommand{\O}{\operatorname{O}}
\newcommand{\Sp}{\operatorname{Sp}}

\newcommand{\rank}{\operatorname{rank}}

\def\dbldots{: \hspace{.65mm} : \hspace{.65mm} :}

\def\L{\underline L}
\def\R{R}

\newcommand{\cA}{\mathcal{A}}
  \newcommand{\cB}{\mathcal{B}}
\def\DemA{\mathcal{B}}
\def\pB{\mathcal{B}'}

\def\cAfpf{\cA_\fpf}
\def\DemAfpf{\DemA_\fpf}
\def\cRfpf{\hat\cR_\fpf}

\def\act{\rtimes}
\newcommand\dact{\hat\circ}
\def\dem{\circ}

\def\invdemspace{\hspace{.85mm}}
\newcommand{\geninvdem}[3]{#2 \invdemspace \dact#1\invdemspace #3}
\newcommand{\idem}[2]{\geninvdem{_*}{#1}{#2}}
\newcommand{\sidem}[2]{\geninvdem{}{#1}{#2}}

\makeatletter
\renewcommand{\@makefnmark}{\mbox{\textsuperscript{}}}
\makeatother

\allowdisplaybreaks[1]

\UseCrayolaColors

\begin{document}
\title{Involution words II: braid relations and atomic structures}

\author{Zachary Hamaker 
\\
 Institute for Mathematics and its Applications \\ University of Minnesota \\ { \tt zachary.hamaker@gmail.com}
 \and
   Eric Marberg\footnote{This author was supported through a fellowship from the National Science Foundation.}
\\
 Department of Mathematics \\  Stanford University \\ {\tt eric.marberg@gmail.com}
\vspace{3mm}
\and
Brendan Pawlowski \footnote{This author was partially supported by NSF grant 1148634.}
\\ Department of Mathematics  \\ University of Michigan \\ {\tt br.pawlowski@gmail.com}
}

\date{}

\maketitle

\begin{abstract}
Involution words are  variations of reduced words for twisted involutions in Coxeter groups.
They arise naturally in  the study of the Bruhat order, of certain Iwahori-Hecke algebra modules, 
and
of    orbit closures in  flag varieties.
Specifically, 
to  any twisted involutions $x,y$ in a Coxeter group $W$ with automorphism $*$, we associate a set of involution words $\hat\cR_*(x,y)$. This set  is  the disjoint union of the reduced words of a set of group elements  $\cA_*(x,y)$, which we call the atoms of $y$ relative to $x$. The atoms, in  turn, are contained in a larger set $\cB_*(x,y) \subset W$ with a similar definition, whose elements we refer to as  Hecke atoms.
Our main results concern some interesting properties of the sets $\hat\cR_*(x,y)$ and $\cA_*(x,y) \subset \DemA_*(x,y)$. For finite Coxeter groups we prove that  $\cA_*(1,y)$ consists of exactly the minimal-length elements $w \in W$ such that $w^* y \leq w$ in Bruhat order, and conjecture a more general property for arbitrary Coxeter groups.
In type $A$, we describe a simple set of conditions characterizing the sets $\cA_*(x,y)$ for all involutions $x,y \in S_n$, giving a common generalization of three recent theorems of Can, Joyce, and Wyser. We   show that   the atoms of a fixed involution in the symmetric group (relative to $x=1$) naturally form a graded poset, while the Hecke atoms surprisingly form an equivalence class under the ``Chinese relation'' studied by Cassaigne, Espie, \emph{et al.}  These facts allow us to recover a recent theorem of Hu and Zhang describing a set of ``braid relations'' spanning  the involution words of any self-inverse permutation. We prove a generalization of this result giving an analogue of Matsumoto's theorem for involution words in arbitrary Coxeter groups.
\end{abstract} 

\tableofcontents
\setcounter{tocdepth}{2}

\section{Introduction}\label{intro-sect}

\subsection{Background}

Let $(W,S)$ be a Coxeter system with length function $\ell : W \to \NN$ 
and write  $\Aut(W,S)$ for the group of automorphisms of $W$ that preserve the set of simple generators $S$.
Since $\Aut(W,S)$ is in bijection with the   set of graph automorphisms of the Coxeter diagram 
of $(W,S)$,   we sometimes refer to elements of $\Aut(W,S)$ as \emph{diagram automorphisms}.

Fix an involution   $* \in \Aut(W,S)$, i.e., a diagram automorphism of order 1 or 2.
We denote the action of $*$ on  elements $w \in W$ by $w^*$,
and 
refer to triples $(W,S,*)$ as \emph{twisted Coxeter systems}.
Each twisted Coxeter system  $(W,S,*)$
determines a set of \emph{twisted involutions} in $W$ given by
\[ \I_*(W) = \{ w \in W : w^{-1} = w^* \}.\]
When $W$ is fixed we write $\I_*$ in place of $\I_*(W)$. 
This work 
is primarily about   certain analogues of reduced words for elements of $\I_*$, which we define  as follows.
Given $s \in S$ and $w \in W$, let 
\be\label{demprod}
s \dem w= \begin{cases} sw &\text{if $\ell(w) < \ell(sw)$} \\ w&\text{otherwise}\end{cases} \qquand w \dem s= \begin{cases} ws &\text{if $\ell(w) < \ell(ws)$} \\ w&\text{otherwise}.\end{cases}
\ee
The symbol
 $\dem$ uniquely extends to an associative binary operation $W \times W \to W$, known in the literature as the \emph{Demazure product} \cite[\S3]{KM}.
Observe that if $s_i \in S$ and $w=s_1s_2\cdots s_k$ is a reduced expression then $s_1\dem s_2\dem \cdots \dem s_k = w$.
%
%
%
In turn, let
\be\label{dact-def}
\idem{x}{w} = (w^*)^{-1} \dem x \dem w \qquad\text{for }x,w \in W.
\ee
A more explicit formula for $\dact_*$ is given in Corollary \ref{invdem-cor}.
Although  it is a simple exercise to show that
$\idem{(\idem{x}{y})}{z} = \idem{x}{( y\dem z)}$ for all $x,y,z \in W$, note that
 the operation $\dact_*$ is not associative.
 However,
with the convention that
$\idem{\idem{\idem{\idem{a}{b}}{c}}{\cdots}}{z} =\idem{(\cdots \idem{(\idem{(\idem{a}{b})}c)}{\cdots})}{z}$, we usually omit the parentheses in  expressions with iterated instances of $\dact_*$.

\begin{definition}\label{def0}
Let $x,y \in \I_*$.
An \emph{involution word} of $y$ relative to $x$ 
is    
a sequence $(s_1,s_2,\dots,s_k)$ with $s_i \in S$ 
of shortest possible length $k$ such that 
$\idem{\idem{\idem{\idem{x}{s_1}}{s_2}}{\cdots}}{s_k} = y.$
We denote the (possibly empty) set of such sequences by $\hat \cR_*(x,y)$. 
\end{definition}


We define  $\hat\cR_*(y) = \hat\cR_*(1,y)$, and 
when $*=\id$ we abbreviate by writing $ \hat\cR(x,y)$ in place of  $\hat \cR_\id(x,y)$.
It is a basic fact that the operation $\dact_*$ defines a map $\I_* \times W \to \I_*$ 
and 
that the set
$\hat \cR_*(y)$; see Corollaries \ref{invdem-cor} and \ref{all-nonempty-cor}.  We call the elements of $\hat\cR_*(y)$   the \emph{involution words} of $y$.

\begin{example}\label{0-ex}
Suppose $W = S_n$ is the symmetric group, viewed as a twisted Coxeter system relative to $*=\id$ and the generating set $S = \{ s_1,s_2,\dots,s_{n-1}\}$ where $s_i=(i,i+1)$.
Writing $\dact=\dact_\id$,
we have
$\sidem{\sidem{\sidem{(1,2)(5,6)}{s_3}}{s_2}}{ s_1}=s_1\cdot ( s_2\cdot ( (1,2)(5,6)\cdot s_3)\cdot s_2)\cdot s_1 = (1,4)(2,3)(5,6)$,
and in fact $(s_3,s_2,s_1) \in \hat\cR(x,y)$   for
 $x=(1,2)(5,6)$ and $y=(1,4)(2,3)(5,6)$.
\end{example}

 Involution words arise naturally in at least three different contexts: in the study of the Bruhat order of $W$ restricted to $\I_*$ \cite{H1,H2,H3,RichSpring,RichSpring2},
in the study of certain Iwahori-Hecke algebra modules related to unitary representations of complex reductive groups  \cite{HuZhang,LV2,LV1,RV},
and in the study
of the cohomology of   orbit closures in the flag variety \cite{CJ,CJW,HMP1,Wyser,WY}.
After possibly exchanging our ``right-handed'' definition with its left-handed version, involution words are the same as what are called
  ``admissible sequences'' in \cite{RichSpring}, 
 ``$\underline S$-expressions'' 
 in \cite{H2,H3},
 and
  ``$\textbf{I}_*$-expressions''
  in
\cite{HuZhang,EM1}. For permutations, involution words are identified in  \cite{CJ,CJW} with  maximal chains in the  weak order on the set of $B$-orbit closures in certain spherical varieties.
We introduced  the term ``involution words'' in our complementary work \cite{HMP1}
with the hope of providing a sensible replacement for
this diversity of names   in the literature.

This paper is an investigation of the properties of involution words for their own sake.
However, many of our results, rather than addressing the sets $\hat \cR_*(x,y)$ directly, will concern the following  related  collections of group elements.

\begin{definition}\label{def1} For each $x,y \in \I_*$ let
$\DemA_*(x,y)= \{ w \in W: \idem{x}{w} = y\}
$
and define
$\cA_*(x,y)$ as the subset of minimal-length elements in $\DemA_*(x,y)$.
We refer to the elements of $\cA_*(x,y)$ as the \emph{atoms} of $y$ relative to $x$, and
 to the elements of $\DemA_*(x,y)$ as the \emph{Hecke atoms} of  $y$ relative to $x$.
\end{definition}

Recall that a \emph{reduced word} for $w \in W$ is a sequence $(s_1,s_2\dots,s_k)$ with $s_i \in S$ of shortest possible length $k$ such that $w=s_1s_2\cdots s_k$. Let $\cR(w)$ denote the set of reduced words for $w \in W$.
One reason for viewing $\cA_*(x,y)$ as a set of ``atomic'' entities is  the following observation.

\begin{proposition}\label{atomexist-prop} For each $x,y \in \I_*$,
it holds  that $\hat \cR_*(x,y) = \bigcup_{w \in \cA_*(x,y)} \cR(w)$.
Thus each of the sets $\hat \cR_*(x,y)$ for $x,y \in \I_*$ is preserved by the braid relations of $(W,S)$.
\end{proposition}

\begin{proof}
Since $\idem{\idem{\idem{\idem{x}{s_1}}{s_2}}{\cdots}}{s_k} =\idem{x}{(s_1\dem s_2\dem\cdots \dem s_k)}$ for all $s_i \in S$,  a word $(s_1,s_2,\dots,s_k) $ belongs to $ \hat\cR_*(x,y)$ if and only if it is a reduced word for an element $w \in W$ of shortest possible length such that $\idem{x}{w} = y$, i.e., if and only if it belongs to $ \cR(w)$ for some $w \in \cA_*(x,y)$.
\end{proof}

Set $\cA_*(y) = \cA_*(1,y)$ and $\DemA_*(y) = \DemA_*(1,y)$,
and when $*=\id$ we drop the subscript and write $\cA(x,y) = \cA_\id(x,y)$ and $\DemA(x,y) = \DemA_\id(x,y)$. We  combine these conventions when possible.
It follows from  \cite[\S1]{Lu2015} that the sets  $\DemA_*(y)$ are the preimages of the surjective map $\pi : W \to \I_*$ described by \cite[Theorem 0.2(c)]{Lu2015} in a recent paper of Lusztig.

\begin{example}\label{1-ex}
Let $W = S_n$.
If $x = (1,3)(2,5) = [3,5,1,4,2]$ and $y = (1,4)(2,5) = [4,5,3,1,2],$
then \[
\ba
 \cA(x,x) &= \{1\} \subset \DemA(x,x) = 
\{ 1,\ s_2,\ s_4,\ s_2s_4\}
\\
\cA(x,y)& = \{s_3\} \subset \DemA(x,y) = 
\{ s_3,\ s_2s_3,\ s_3s_2,\ s_3s_4,\ s_2s_3s_2,\ s_2s_3s_4,\ s_2s_4s_3,\ s_2s_4s_3s_2\}
.
\ea
\]
In general, $\cB(x,x)$ is the subgroup generated by all right descents of $x$. 
\end{example}

Proposition \ref{atomexist-prop} shows that each set of involution words decomposes as a disjoint union of  the ordinary reduced words of some corresponding set of atoms.  The problem of studying involution words reduces in this sense to that of classifying the elements of  $\cA_*(x,y)$ and $\DemA_*(x,y)$. Our main theorems will highlight some noteworthy properties of these sets. 

\subsection{Results}

We now outline our main results.
Section  \ref{prelim-sect}  presents a few  general facts about the sets $\hat \cR_*(x,y)$
and  $\cA_*(x,y)\subset \DemA_*(x,y)$,
and derives a second definition of $\hat\cR_*(x,y)$, showing that our notion of involution words   is equivalent to analogous concepts in \cite{HMP1,HuZhang,H2,H3,RichSpring}; see Corollary \ref{altdef-cor}.
In Section \ref{duality-sect},
we prove some  results relating involution words defined with respect to distinct diagram automorphisms $*$ and $\diamond$ with  the same image in the outer automorphism group of $W$. 
Such facts allow us to  translate between most statistics of interest for involution words defined with respect to the two diagram automorphisms of the symmetric group, for example.

Let $\leq$ denote the Bruhat order on $W$.
In Section \ref{bruhat-sect}, we study a conjectural,   alternate definition of  $\cA_*(x,y)$, involving a simple condition depending only on this order.
Among other results, we derive the following theorem from a slightly more general statement;  see Theorem \ref{lastbr-thm}.

\begin{theorem*}
Let $(W,S,*)$ be a twisted Coxeter system. If $W$ is finite, then for each $y \in \I_*$ the set  $\cA_*(y)=\cA_*(1,y)$ consists of  precisely the  minimal-length elements $w \in W$ satisfying $w^*y \leq w$.
\end{theorem*}

We conjecture  that if $(W,S,*)$ is any  twisted Coxeter system and $x,y \in \I_*$ are such that $\hat\cR_*(x,y)\neq\varnothing$, then $\cA_*(x,y)$ is the set of 
 of minimal-length elements $w\in W$ satisfying  $w^* y \leq xw$; see Conjecture \ref{atoms-conj}.
While it is not difficult to show that every  $w \in \cA_*(x,y)$ satisfies $w^*y \leq xw$ (see Proposition \ref{bruhatchar-prop}), our general conjecture  seems much harder to prove.
Our proof of the weaker form of this statement given above consists of an inductive argument leveraging   techniques from Knutson and Miller's theory of \emph{subword complexes} \cite{KM}.

Much of the rest of the paper  concerns   the sets $\cA_*(x,y)$ and $\cB_*(x,y)$   exclusively  when $*=\id$ is the identity and $W=S_n$ is the symmetric group.
One is particularly interested in the cases of these sets  when  $x = 1$ or when $n$ is even and $x  = s_1s_3s_5\cdots s_{n-1}$.
In these \emph{geometric cases},  elements of $\hat\cR(x,y)$ correspond to maximal chains in the natural ``weak order'' on the set of closures of the  $\O_n(\CC)$- or $\Sp_n(\CC)$-orbits  in the flag variety $\Fl(n)$. The collections of atoms $\cA(x,y)$  for $y \in \I(S_n)$ in turn can be identified with the sets $W(Y)$ attached to these orbit closures by a more general  construction of Brion  (see \cite[\S1]{Brion98}).
Such connections between involution words and the geometry of the flag variety are the main subject of our complementary work \cite{HMP1}.

Here our focus is the  algebraic structure of the sets $\cA(x,y)$ and $\DemA(x,y)$ in type $A$.
Let  $\Ifpf(S_{n})$  denote the set of fixed-point-free involutions in $S_n$. When $n$ is even, set $\wfpf{n} = s_1s_3s_5\cdots s_{n-1}$ and for each $y \in \Ifpf(S_n)$   define  
\[\cRfpf(y) = \hat \cR(\wfpf{n},y),
\qquad 
\cAfpf(y) = \cA(\wfpf{n},y),
\qquand 
\DemAfpf(y) = \DemA(\wfpf{n},y).\]
The sets $\cA(y)$ and $\cAfpf(y)$ for $y \in \I(S_n)$  have been studied previously by
Can, Joyce, and Wyser \cite{CJ,CJW}  using  different notation. 
 Can, Joyce, and Wyser's results   \cite[Theorem 2.5 and Corollary 2.16]{CJW}, in particular,
 provide a relatively simple list of numeric inequalities,
depending on the cycle structure of   $y$ and  involving only the one-line representation  of a permutation $w \in S_n$,
which are necessary and sufficient  for  $w$ to belong to $\cA(y)$ or $\cAfpf(y)$. 

 The main result of Section \ref{atomSn-sect}  is a more general statement of this type which  classifies the sets $\cA(x,y)$ for all  $x,y \in \I(S_n)$. 
Our exact statement requires  several notational preliminaries.
We sketch the idea here while deferring the   details.
Let $[n] = \{1,2,\dots,n\}$, and for $y \in \I(S_n)$, define
 \[ 
 \Cyc(y) =  \{ (a,b) \in [n]\times [n]  :  a \leq b=y(a)\}.
\]
The elements of $\Cyc(y)$ are in bijection with the cycles of $y$. 
Let  $\Gamma(y)$ be the  set formed by adding to $\Cyc(y)$   all pairs  $(a,b)$  with $a=y(a) >b =y(b)$.
To any  $\gamma,\gamma' \in \Gamma(y)$, we associate via \eqref{sigma-eq} a certain involution $\sigma(\gamma,\gamma') \in \I(B_4)$  in the Weyl group of type $B_4$.
We then introduce  a certain natural partial order $\prec$ on the finite set $\I(B_4)$; see  \eqref{prec-eq}. With $S_n$ acting on $[n]\times [n]$ by simultaneously permuting coordinates,  our classification of $\cA(x,y)$ is equivalent to the following statement (see Theorem \ref{atomSn-thm}):

\begin{theorem*}
Let $x,y \in \I(S_n)$ and $w \in S_n$. Then $w \in \cA(x,y)$ if and only if
$w\gamma \in \Gamma(x)$  for all $\gamma \in\Cyc(y)$,
and
$\sigma(w\gamma,w\gamma') \preceq \sigma(\gamma,\gamma')$
for all disjoint elements $\gamma,\gamma' \in \Cyc(y)$.
\end{theorem*}

Unpacking the relevant definitions here translates this result into an explicit set of  numeric inequalities (see Theorem \ref{atomSn'-thm}) which matches more closely the  form of the related statements in \cite{CJ,CJW}. 
Our theorem subsumes these earlier results, and our methods are self-contained.
As a consequence of facts shown in Section \ref{duality-sect}, our theorem also provides a characterization of the sets $\cA_*(x,y)$ for all $x,y \in \I_*(S_n)$, when $*$ is the nontrivial diagram automorphism mapping $s_i\mapsto s_{n-i}$.

In Section \ref{poset-sect} we turn to  the larger  sets of Hecke atoms $\DemA(y)$ and $\DemAfpf(y)$ for involutions in  $ y \in S_n$, which do not appear to have been studied previously.  
We prove that these sets are spanned and preserved by two unexpectedly simple equivalence relations on $S_n$.
In detail, 
define $\sim_\DemA$ to be the weakest equivalence relation on integer sequences with
\begin{equation}
\label{eq:chinese}
[\ \cdots\ c,a,b\ \cdots\ ] \sim_\DemA [\ \cdots\ b,c,a\ \cdots\ ] \sim_\DemA [\ \cdots\ c,b,a\ \cdots\ ]
\qquad\text{when }a\leq b\leq c,
\end{equation}
where the ellipses mask arbitrary, but respectively matching subsequences. One applies this relation to $S_n$ by viewing permutations as $n$-element sequences in one-line notation. We prove the following as Theorem \ref{equiv1-thm}:
\begin{theorem*}
The sets $\DemA(y) $ for $y \in \I(S_n)$ are  precisely the images under the inversion map $w\mapsto w^{-1}$ of the distinct    equivalence classes in $S_n$ with respect to $\sim_\DemA$.
\end{theorem*}
Surprisingly, the relation $\sim_{\DemA}$ coincides with something already present in the literature with little apparent connection to involutions in Coxeter groups: namely, $\sim_{\DemA}$ is the so-called \emph{Chinese relation}  introduced by Duchamp and Krob~\cite{DuchampKrob} and further studied by Cassaigne, Espie, Krob, Novelli, and Hivert~\cite{CEHKN}. Thus, the sets $\cB(y)$ in type $A$ may be identified  (on taking inverses) with elements of the  \emph{Chinese monoid} considered in \cite{CEHKN,CM2}.
This connection seems to represent  the first appearance in the literature of the Chinese monoid ``in the wild.''
 
A similar result holds
which identifies the (inverse elements of the) sets $\DemAfpf (y)$ for $y \in \Ifpf(S_{2n})$ as  equivalence classes under a simple  relation $\sim_{\DemAfpf}$ on even-length words; see Theorem \ref{equiv2-thm}.
We prove that when restricted to (the inverse elements of) $\cA(y)$ and $\cAfpf(y)$, the relations $\sim_{\DemA}$ and $\sim_{\DemAfpf}$ become bounded partial orders (see Theorems \ref{prec1-thm} and \ref{prec2-thm}), which
make   $\cA(y)$ into a graded poset and  $\cAfpf(y)$ into a graded lattice (see Propositions \ref{graded-prop} and \ref{graded-prop2}). As one consequence of these results, we show in type $A$  that $|\cA(y)| = 1$ if and only if $y \in S_n$ is a 321-avoiding involution; see Corollaries \ref{prec1-cor} and \ref{prec2-cor}.
As a second application, these poset structures lead to a simple, alternate proof of the following analogue of Matsumoto's theorem for involution words in type $A$, which was recently found by Hu and Zhang~\cite{HuZhang}:

\begin{theorem*}[Hu and Zhang \cite{HuZhang}]
For each $y \in \I(S_n)$, the set $\hat\cR(y)$ is spanned and preserved by the braid relations for $S_n$ together with the relations
$(s_i,s_{i+1},\dots) \sim (s_{i+1},s_i,\dots)$  for $i \in [n-2]$.
\end{theorem*}

Our approach also leads to a novel ``fixed-point-free'' version of Hu and Zhang's theorem: 

\begin{theorem*}
For each $y \in \Ifpf(S_{2n})$, the set $\hat\cR(y)$ is spanned and preserved by the braid relations for $S_{2n}$ together with the relations
$(s_{2i},s_{2i+1},\dots) \sim (s_{2i},s_{2i-1},\dots)$  for $i \in [n-1]$.
\end{theorem*}

These results are restated as Theorems \ref{braid1-thm} and \ref{braid2-thm}.
By direct algebraic methods, we prove a common generalization of these theorems which describes a set of relations spanning the sets $\hat \cR_*(y)$ for any twisted Coxeter system; see Theorem \ref{matsumoto-twist}. 
The situation in type $A$ is  special;   our ``twisted Matsumoto's theorem'' requires us to add many relations to the braid relations in general, and reduces in type $A$ to a weaker statement than what can be shown using the combinatorial classification of atoms for involutions in the symmetric group.
It remains an open problem to find minimal sets of relations spanning the sets $\hat\cR_*(y)$ for any particular twisted Coxeter system.

%
%

\subsection*{Acknowledgements}

We thank
Dan Bump, Michael Joyce, Vic Reiner, Ben Wyser, Alex Yong, and Benjamin Young for many helpful conversations and suggestions in the course of the development of this paper.

\section{Preliminaries}
\label{prelim-sect}

In this section we review some basic facts about twisted involutions  to develop some   general properties of the sets  $\cA_*(x,y)$ and $\DemA_*(x,y)$.  Many  results mentioned here appear in some form in the papers of Richardson and Springer \cite{RichSpring,RichSpring2,Springer} or   in  more recent work of Hultman \cite{H1,H2,H3}.

Let  $(W,S,*)$  be an arbitrary twisted Coxeter system with length function
 $\ell : W \to \NN$, and  as usual let  $\I_* = \I_*(W) = \{ w \in W : w^{-1} = w^*\}$.
For $w \in W$ let  
$ \DesL(w) = \{ s \in S : \ell(sw) < \ell(w)\} $ 
and
$ \DesR(w) = \{ s \in S: \ell(ws) < \ell(w)\}$
 denote the corresponding \emph{left} and \emph{right descent sets}. 
We define $ \act_* : \I_* \times S \to \I_* $  
as the operator given by
\be\label{op-eq} x\act_*s =\begin{cases} 
s^*xs &\text{if }s^* x \neq xs \\ 
xs & \text{if }s^*x = xs\end{cases} 
\qquad\text{for $x \in \I_*$ and $s \in S$.}
\ee
Although $(x\act_*s )\act s = x$, this operation does not extend to a group action of $W$ on $\I_*$.
The following   fact is a straightforward consequence of the exchange principle for Coxeter systems, and is equivalent to \cite[Lemma 3.4]{H2}:
\begin{proposition}[See \cite{H2}] \label{des-lem} If $x \in \I_*$ and $s \in \DesR(x)$ then 
$
\ell(x\act_*s ) = \begin{cases} \ell(x) - 2 &\text{if }s^*x\neq xs \\ \ell(x) - 1 &\text{if }s^*x=xs.\end{cases}
$
\end{proposition}


The proposition implies that the operation $\dact$ given in the introduction satisfies $\idem{x}{s} = x$ for $s \in \DesR(x)$ and $\idem{x}{s} = x\act_*s $ for $s \in S\setminus \DesR(x)$. We obtain this corollary as a consequence:
\begin{corollary}\label{invdem-cor}
If $ x \in \I_*$ and $s \in S$ then $\idem{x}{s} \in \I_*$ as
 \[ \idem{x}{s} = s^* \dem x \dem s =
\begin{cases} 
s^* xs & \text{if $s \notin \DesR(x)$ and $s^*x \neq xs$}
\\
xs &\text{if $s \notin \DesR(x)$ and $s^*x=xs$} 
\\
x&\text{if $s \in \DesR(x)$}.
 \end{cases}\]
 
\end{corollary}

It follows by induction from these observations that for any $x \in \I_*$ there is some sequence $s_1,s_2,\dots,s_k \in S$ such that $x=\idem{\idem{\idem{\idem{1}{s_1}}{ s_2}}{\cdots}}{s_k}$.
Hence, we deduce the following:
\begin{corollary}\label{all-nonempty-cor} If $x \in \I_*$ then the sets $\hat \cR_*(x)$, $\cA_*(x)$, and $\DemA_*(x)$ are all  nonempty.
\end{corollary}


It is thus well-defined to introduce the following length function on $\I_*$.

\begin{definition}\label{ellhat-def} Let $\ellhat_*(x)$ for $x \in \I_*$ denote the common length of any word in $\hat \cR_*(x)$. 
\end{definition}

When $*$ is the identity automorphism, we omit the subscript and set $\ellhat(x) = \ellhat_\id(x)$.
In the case when
  $W=S_n$ is the   symmetric group, it holds that
$ \ellhat(x) = \tfrac{1}{2} \( \ell(x)+ \kappa(x)\)$
where $\kappa(x)$ is the number of 2-cycles in an involution  $x \in S_n$; see \cite{Incitti1}. Incitti's work \cite{Incitti2} derives similar formulas  for $\ellhat$ when $W$ is one of the other classical Weyl groups.

Let $\leq$ denote the Bruhat order on $W$ as defined, e.g., in \cite[\S.1]{CCG}. The poset $(\I_*,\leq)$ is graded with rank function $\ellhat_* : \I _*\to \NN$ \cite[Theorem 4.8]{H1}
and  inherits an  analogue of the subword characterization of $(W,\leq)$ \cite[Theorem 2.8]{H3}.
Just using
Proposition \ref{des-lem} and basic properties of the Bruhat order, we deduce this second corollary:


\begin{corollary}[See \cite{H2}] \label{h2-cor} If $x \in \I_*$ and $s \in S$ then the following properties are equivalent:
 $  x \act_* s < x$ $\Leftrightarrow$ $s \in \DesR(x)$ $\Leftrightarrow$ $s^* \in \DesL(x)$ $\Leftrightarrow$ $ \ell(x\act_* s) < \ell(x)$ $\Leftrightarrow$ $ \ellhat_*(x\act_* s) = \ellhat_*(x)-1.$
\end{corollary}


This corollary implies the following alternate definition of involution words.

\begin{corollary}\label{altdef-cor}
Let $x,y \in \I_*$. If $\hat\cR_*(x,y)$ is nonempty (as happens, for example, when $x=1$) then the involution words of $y$ relative to $x$ are exactly the sequences $(s_1,s_2,\dots,s_k)$ with $s _i \in S$ of shortest possible length $k$ such that $y = x \act_* s_1 \act_* s_2 \act_*\cdots \act_* s_k$.
\end{corollary}

We now describe some properties of the sets $\cA_*(x,y)$ and $\DemA_*(x,y)$ 
given in
Definition \ref{def1}.

\begin{lemma}\label{atomdes-lem}
 If $x,y \in \I_*$ and $w \in \cA_*(x,y)$ then
$\DesR(w) \subset \DesR(y) $ and $\DesL(w) \cap \DesR(x)=\varnothing$.
 \end{lemma}
 
\begin{proof}
It is straightforward to check this using  the definition of $\cA_*(x,y)$ and Corollary \ref{h2-cor}.
\end{proof}


\begin{proposition}\label{atomdes-prop}
 Let $x,y\in \I_*$ and $s \in S$. 
 \ben
 \item[(a)] If $s \notin \DesR(y)$ then $\cA_*(x,y ) = \{ ws : w \in \cA_*(x,y\act_*s) \text{ with } s \in \DesR(w) \}.$
  \item[(b)] If $s \in \DesR(x)$ then $\cA_*(x,y) = \{ sw : w \in \cA_*(x \act_* s,y) \text{ with } s \in \DesL(w) \}.$
  \een
\end{proposition}

\begin{proof}
Suppose  $s \notin \DesR(y)$.
If $w \in \cA_*(x,y\act_* s)$
and $s \in \DesR(w)$, then $w$ has a reduced expression ending with $s$ and $ws \in \cA_*(x,y)$.
Conversely, if $v \in \cA_*(x,y)$, then $s \notin \DesR(v)$ by Lemma \ref{atomdes-lem} whence  $vs \in \cA_*(x,y\act_*s)$,   so $v=ws$ for the element $w=vs \in \cA_*(x,y)$  which has $s \in \DesR(w)$ by construction.
This proves part (a), and
part (b) follows  similarly.
\end{proof}

The left and right \emph{weak orders} $<_L$ and $<_R$ on $W$ are the transitive closures of the relations $w<_L sw$ and $w <_R wt$ for   $w \in W$ and $s,t \in S$ such that  
$\ell(sw) > \ell(w)$ and $\ell(wt) > \ell(w)$.
The appropriate analogue of these orders for twisted involutions is the following:

\begin{definition} The \emph{(two-sided) weak order} $<_{T,*}$ on $\I_*$  is the transitive closure of the relation with $w <_{T,*} w \act_* s$ for each $w \in \I_*$ and $s \in S$ such that $\ellhat_*(w)  < \ellhat_*(w\act_* s) $. 
\end{definition}

\begin{remark}
This order was first defined in \cite{RichSpring} (see also \cite[\S5]{H2}), and is   the unique partial order on $\cI_*$ with $x \leq_{T,*} y$ if and only if all (equivalently, any) of the sets $\hat \cR_*(x,y)$, $\cA_*(x,y)$, and $\DemA_*(x,y)$ are  nonempty.
If $x \leq_{T,*} y$ then $x \leq y$, but the reverse implication does not hold in general.
\end{remark}

%

If  $J \subset S$  then we write $W_J$ for the standard parabolic subgroup of $W$ which $J$ generates.  It is well-known (see, e.g., \cite[Chapter 1]{GP})  that $W_J$ is a lower set relative to the Bruhat order on $W$ (i.e., $x \leq y \in W_J$ implies $x \in W_J$), and hence also relative to both weak orders. 
It likewise follows that $\I_*\cap W_J$ is  a lower set relative to $\leq_{T,*}$.
If $X$ is any alphabet, then we write $\Mon{X}$ for the \emph{free monoid} on $X$, that is, the set of all finite sequences of elements of $X$.

\begin{proposition}\label{parred-prop}
Let $x,y \in \I_*$.
If $J \subset S$ and $y \in W_J$, then $\DemA_*(x,y) \subset W_J$ and $\hat\cR_*(x,y) \subset \Mon{J}$.
\end{proposition}

\begin{proof}
Fix $y \in \I_* \cap W_J$. 
After possibly replacing $J$ by $J\cap J^*$, we may assume that $J=J^*$.
If $\DemA_*(x,y)$ is empty then the result holds trivially, so assume   $x \leq_{T,*} y $  so that $x \in W_J$.
Suppose $w \in \DemA_*(x,y)$ but $w\notin W_J$.
We may then write $w = utv$ for some $u \in W_J$ and $t \in S\setminus J$ and $v \in W$ such that $\ell(w) = \ell(u) + 1 + \ell(v)$.
Define $a = \idem{x}{u}$ and $b = \idem{a}{t}$ so that $y =\idem{b}{v}$.
Since  $b \leq_{T,*} y \in W_J$ by construction and since $W_J$ is a lower set under $\leq$, we must have $ b \in W_J$. 
However,
as $a \in W_J$, it follows that $a < at \leq b$, which implies   
 that  $at \in W_J$ (again, since $W_J$ is a lower set in Bruhat order) which is impossible since $t \notin J$.
Hence we must have $\cA_*(x,y)\subset \DemA_*(x,y) \subset W_J$ as desired, so $\hat\cR_*(x,y) \subset \Mon{J}$ as  $\cR(w) \subset \Mon{J}$ for all $w \in W_J$ (e.g., by \cite[Proposition 1.2.10]{GP}).
\end{proof}

We use this result to confirm the following, intuitively clear property.
For twisted Coxeter systems $(W,S,*)$ and $(W',S',\diamond)$, we define a \emph{morphism} $\phi : (W,S,*) \to (W',S',\diamond)$ to be a group homomorphism $W \to W'$ such that $\phi(S) \subset \{1\}\cup S'$ and $\phi(w^*) = \phi(w)^\diamond$ for all $w \in W$.

\begin{corollary}\label{morph-cor}
Let $\phi : (W,S,*) \to (W',S',\diamond)$ be an injective morphism of twisted Coxeter systems.
Then $\phi$ restricts to a map $\I_*(W) \to \I_\diamond(W')$ and the maps
\[ (s_1,s_2,\dots,s_k) \mapsto (\phi(s_1),\phi(s_2),\dots,\phi(s_k)) \qquand x \mapsto \phi(x) \]
are  bijections 
$
\hat\cR_*(x,y) \mapsto \hat\cR_\diamond(\phi(x),\phi(y))
$
and
$
\DemA_*(x,y) \to \DemA_\diamond\(\phi(x),\phi(y)\)
$ for each $x,y \in \I_*(W)$.
\end{corollary}

\begin{proof}
By Proposition \ref{parred-prop} we reduce to when $\phi$ is an isomorphism; then the result is clear.
\end{proof}

\section{Duality for twisted involutions}
\label{duality-sect}

In the symmetric group $S_n$ one has two kinds of twisted involutions: the ordinary involutions and those relative to
the diagram automorphism $s_i \mapsto s_{n-i}$ (which is also the inner automorphism induced by the group's longest element).  This section presents some general observations explaining a sort of ``duality'' relating  most properties  of interest  for these two   types of involution words.

%

Let $(W,S)$ be any Coxeter system. If  $J \subset S$ is a subset such that $W_J$ is finite, then we write $w_J$ for the longest element in $W_J$. When $W$ is itself finite, we define $w_0 = w_S$.
We say that two subsets $J, K \subset S$  \emph{commute} if $st =ts$ for all $s\in J$ and $t \in K$. 
The \emph{irreducible factors} of $(W,S)$ are    the  Coxeter systems  of the form $(W_J,J)$ where $\varnothing \subsetneq J \subset S$ is minimal among the nonempty subsets of simple generators which commute with their complements. 
For any subset $J\subset S$ commuting with $K = S\setminus J$, it holds that  $W = W_J \times W_K$ and $\ell(xy) = \ell(x) + \ell(y)$ for all $x \in W_J$ and $y \in W_K$. In this situation we write $w |_J$ and $w|_K$ for the images of $w \in W$ under the natural projections $W \to W_J$ and $W\to W_K$, so that $w = w|_J\cdot  w|_K = w|_K\cdot  w|_J$.
We will need the following fact:

\begin{proposition}[Franzsen and Howlett \cite{FH}] \label{disjdec-prop} If $v \in W$ then $v^{-1}   s   v \in S$ for all $s \in S$ if and only if 
 $v= w_J$ for a subset $J \subset S$ which commutes with  $S\setminus J$ and which is such that $W_J$ is finite.
\end{proposition}

\begin{proof}
Since $v^{-1} s v = (v|_J)^{-1} \cdot s\cdot  (v|_J) \in J$ if $s \in J$ and $(W_J,J)$ is an irreducible factor of $(W,S)$, it suffices to assume that $(W,S)$ is irreducible and show  that (i) $v=1$ if $W$ is infinite and (ii) $v \in \{1,w_0\}$ if $W$ is finite.
Assertion (i) is  \cite[Lemma 9]{FH}, whose proof also establishes (ii).
\end{proof}

Fix an involution $* \in \Aut(W,S)$. For subsets $X,Y \subset W$  let $XY = \{ xy : (x,y) \in X\times Y\}$.
Recall that $\Mon{S}$ denotes the free monoid on $S$. Given 
subsets $\cX,\cY\subset \Mon{S}$, let $\cX \shuffle \cY $ denote the set 
of words  whose letters may be partitioned into two disjoint subwords, one equal to a word $\a \in \cX$ and the other equal to some $\b \in \cY$.

\begin{lemma}\label{shuffle-lem} Suppose $J\subset S$ commutes with $K = S\setminus J$ and $J=J^*$. If $x,y \in \I_*(W)$ then
\ben
\item[(a)]
$
\hat\cR_*(x,y) = \hat\cR_*(x|_J,y|_J) \shuffle \hat\cR_*(x|_K,y|_K)
$.

\item[(b)] $ \cA_*(x,y) = \cA_*(x|_J,y|_J) \cA_*(x|_K,y|_K)
$
and 
$\DemA_*(x,y) =  \DemA_*(x|_J,y|_J) \DemA_*(x|_K,y|_K)
.
$
\een
\end{lemma}

\begin{proof}
It suffices to prove part (b), and this is an easy exercise  on noting that for $x \in \I_*(W_J)$ and $y\in \I_*(W_K)$,
it holds that $\idem{(xy)}{s} = (\idem{x}{s})y$ for $s \in J$ and $\idem{(xy)}{t} = x(\idem{y}{t})$ for $t \in K$.
\end{proof}

For the rest of this section we fix 
 the following notation. First, let $v_0 \in \I_*(W)$ be an element
with $ v_0^{-1} s v_0 \in S$ for all $s \in S$.
We write $J\subset S$ for the subset which, by Proposition \ref{disjdec-prop}, exists
such that $v_0 = w_J$, and define $K=S\setminus J$. Note that $J$ and $K$ commute and that $v_0=v_0^{-1}=v_0^*$.
Recall that $*\in \Aut(W,S)$ is a fixed involution, and define $\diamond  \in \Aut(W,S)$ by
\be\label{diamond-def}  w^\diamond = v_0 w^* v_0\qquad\text{for $w \in W$.}\ee
It necessarily holds that $J = J^* = J^\diamond$ and so also $K=K^*=K^\diamond$.
By construction, both $(W,S,*)$ and $(W,S,\diamond)$ are twisted Coxeter systems,
and
our goal is to compare  involution words for elements of  $\I_*=\I_*(W)$ and $\I_\diamond=\I_\diamond(W)$.
We begin  by noting three quick lemmas:

\begin{lemma}\label{quickly-lem}
The map
$x \mapsto v_0 x
$
is a
  bijection $\I_\diamond \to \I_*$.
\end{lemma}

\begin{lemma}\label{dualtech-lem} If $s \in S$ and $x \in \I_\diamond$ then
$ (v_0 x) \act_* s = v_0( x \act_\diamond s)$.
\end{lemma}

Observe that   $ (\I_*)^\diamond = \I_*$, where 
for a subset $X \subset W$ we write $X^\diamond = \{ x^\diamond : x \in X\}$.

\begin{lemma}\label{diamond-lem} If $x,y \in \I_*$ then $ \cA_*(x^\diamond,y^\diamond) =\cA_*(x,y)^\diamond$ and 
 $ \DemA_*(x^\diamond,y^\diamond) =\DemA_*(x,y)^\diamond$.
\end{lemma} 

We omit the proofs of these statements, which are easy to check from the definitions.
Given  $\e = (s_1,s_2,\dots,s_k ) \in \Mon{S}$, set  
$\e^* = (s_1^*,s_2^*,\dots,s_k^*)$
and
$
\e^{\op} = (s_k,\dots,s_2,s_1).
$
We extend these operations to subsets of $\Mon{S}$  element-wise, and for $X \subset W$  set $X^{-1} = \{ x^{-1} : x \in X\}$. 

\begin{proposition}\label{maindual-prop} Let $x,y \in \I_\diamond$ 
and suppose $x \leq_{T,\diamond} y$.
\ben
\item[(a)] If $x|_J = y|_J$ then  $   \hat\cR_\diamond(x,y)=\hat \cR_*(v_0x, v_0y)  $ and $  \cA_\diamond(x,y) = \cA_*(v_0x,v_0y)$.

\item[(b)] If $x|_K = y|_K$ then  $  \hat\cR_{\diamond}(x,y) = \hat \cR_*(v_0y,v_0x)^\op$ and $\cA_\diamond(x,y) = \cA_*(v_0y,v_0x)^{-1}$.

\item[(c)] If $x|_J \neq y|_J$ and $x|_K \neq y|_K$ then $v_0x$ and $v_0y$ are not comparable in $\leq_{T,*}$. 

\een
\end{proposition}

\begin{proof}
Suppose $x|_J = y|_J$.  Since $v_0 =w_J \in W_J$ it  follows that 
$(v_0x)|_J = (v_0y)|_J $ and $ (v_0x)|_K = x|_K$ and $ (v_0y)|_K = y|_K,$ so by Lemma \ref{shuffle-lem} we have 
$\cA_\diamond(x,y)  =  \cA_\diamond(x|_K,y|_K)$
and
$\cA_*(v_0x,v_0y) =  \cA_*(x|_K,y|_K)$. As the involutions $*$ and $\diamond$ restrict to the same map on $W_K$, we deduce that $ \cA_\diamond(x|_K,y|_K) =  \cA_*(x|_K,y|_K)$, so $  \cA_\diamond(x,y) = \cA_*(v_0x,v_0y)$ and in turn $   \hat\cR_\diamond(x,y)=\hat \cR_*(v_0x, v_0y)  $.

To prove part (b), suppose $x|_K = y|_K$.  It then follows similarly by  Lemma \ref{shuffle-lem} that $\cA_\diamond(x,y) = \cA_\diamond(x|_J,y|_J)$ 
and 
$\cA_*(v_0y,v_0x) = \cA_*(v_0(y|_J),v_0(x|_J))$, so without loss of generality we may assume that $W$ is finite and $v_0 = w_0$  so that $J=S$ and $K=\varnothing$.
Adopt these hypotheses and fix a word $\e = (s_1,s_2,\dots,s_k) \in \Mon{S} $.
It is straightforward from Lemma \ref{dualtech-lem} to show that 
\[ x \act_\diamond s_1 \act_\diamond s_2 \act_\diamond \cdots \act_\diamond  s_k = y
\qquad\Leftrightarrow
\qquad
(v_0 y) \act_* s_k \act_* \cdots \act_* s_2 \act_* s_1 = v_0 x.
\]
From this equivalence and the fact that  multiplication by $v_0=w_0$ is order-reversing with respect to  Bruhat order \cite[Proposition 2.3.4]{CCG}, it follows in view of Corollary \ref{altdef-cor} that
$\e \in \hat\cR_\diamond(x,y)$ if and only if $\e^\op \in \hat\cR_*(v_0y,v_0x)$. 

Finally, to prove part (c) suppose $x|_J \neq y|_J$ and $x|_K \neq y|_K$. Write $x' = v_0 x$ and $y' = v_0y$. Since $ x \leq_{T,\diamond} y$ and since $\leq_{T,\diamond}$ is weaker than the Bruhat order, it follows from Lemma \ref{shuffle-lem} that $x|_J < y|_J$ and $x|_K < y|_K$. This implies, however, that $y'|_J = w_J (y|_J) < w_J (x|_J) = x'|_J$ while $x'|_K = x|_K < y|_K = y'|_K$;
hence,  by Lemma \ref{shuffle-lem},
$x'$ and $y'$ cannot be comparable in $\leq_{T,*}$. 
\end{proof}

We note two corollaries of the preceding result. Observe that $v_0 \in \I_*\cap \I_\diamond$.

\begin{corollary}\label{dualcor1} The map  $\e \mapsto \e^\op$  is   a bijection $\hat\cR_*(v_0) \leftrightarrow \hat\cR_\diamond(v_0)$.
\end{corollary}

\begin{proof}
Since $v_0=v_0^{-1}=w_J$, this follows immediately from Proposition \ref{maindual-prop}(b).
\end{proof}

\begin{corollary}\label{dualcor2}
Suppose $w \in \I_*$ is in the center of a standard parabolic subgroup of $W$. Then  $\hat\cR_*(w)$ is closed under $\e \mapsto \e^\op$, and both $\cA_*(w)$ and $\DemA_*(w)$ are  closed under taking inverses.
\end{corollary}

\begin{proof}
By Proposition \ref{parred-prop} we reduce to the case when $w \in Z(W)$.
Set $v_0 = w$. Then $\diamond = *$ so $\hat\cR_\diamond(w) = \hat\cR_*(w)$ and the desired statement follows from Corollary \ref{dualcor1}.
\end{proof}

\begin{remark}
It follows from  \cite[Exercise 4.10]{CCG} and \cite[Exercise 1 in \S6.3]{Hu} that 
 an element $w \in W$ belongs to the center of a standard parabolic subgroup under the following limited circumstances: namely, if and only if $w=w_L$ for a subset $L \subset S$ which commutes with $S\setminus L$ and which is such that $(W_L,L)$ is finite with no irreducible factors of type
 $A_n$ or $D_{2n+1}$ ($n\geq 2$), $E_6$, or $I_2(2m+1)$.
 \end{remark}

We explicitly record the following special case of these results. 

\begin{corollary}\label{lastdia-prop}
Assume  $W$ is finite and $v_0 = w_0$. Then, for all $x,y \in \I_\diamond$, it holds that
\[ \hat \cR_\diamond(x,y)= \hat \cR_*(w_0y,w_0x)^\op \qquand  \cA_\diamond(x,y) = \cA_*(w_0y,w_0x)^{-1}\]
and consequently $\ellhat_\diamond(x) = \ellhat_*(w_0) - \ellhat_*(w_0x)$.

\end{corollary}

\begin{proof}
In this case $K = \varnothing$ so we always have $x|_K =  y|_K = 1$. Invoking Proposition \ref{maindual-prop}(b) twice with the roles of $*$ and $\diamond$ interchanged   shows that $x \leq_{T,\diamond} y$ if and only if $w_0 y \leq_{T,*} w_0x$ and that in this case 
$\hat \cR_\diamond(x,y)= \hat \cR_*(w_0y,w_0x)^\op$ and $\cA_\diamond(x,y) = \cA_*(w_0y,w_0x)^{-1}$ as desired.
\end{proof}

\begin{remark}
Except in the case when $x=1$ and $y=v_0$ (see Corollary \ref{dualcor1a}), there does not seem to be a simple relationship between the sets $\DemA_*(x,y) $ and  $\DemA_\diamond(v_0y,v_0x)$.
It is an open problem to give twisted analogues of our results
in Section \ref{poset-sect}
 about   $\DemA(x,y)$ in type $A$.
\end{remark}

\section{Atoms and Bruhat order}\label{bruhat-sect}

Everywhere in this section $(W,S,*)$ denotes an arbitrary twisted Coxeter system. Write $\I_* = \I_*(W)$  and let $x,y \in \I_*$.
With $\ell$ the usual length function on $W$ and $\ellhat_*$  as in Definition \ref{ellhat-def}, we set
\[ \ell(x,y) = \ell(y)- \ell(x)\qquand  \ellhat_*(x,y) = \ellhat_*(y) - \ellhat_*(x).\]
Recall the definitions of $\cA_*(x,y)$ and $\DemA_*(x,y)$ from the introduction.
In this section we consider the following formally similar, though \emph{a priori} distinct sets:
\begin{definition}\label{atoms'-def} For each $x,y \in \I_*$ let
$\pB_*(x,y)= \{ w \in W: w^*y \leq xw\}
$
and define
$\cA'_*(x,y)$ as the subset of minimal-length elements in $\pB_*(x,y)$.
\end{definition}

Note that these sets are completely determined by the structure of the Bruhat order on $W$. 
The focus of this section is the following conjecture: 
\begin{conjecture}\label{atoms-conj} 
 $\cA_*(x,y) = \cA'_*(x,y)$ for all twisted involutions $x,y \in \I_*$ with $x \leq_{T,*} y$.
\end{conjecture}

\begin{remark}
Note that if $x\not \leq_{T,*} y$ then  $\cA_*(x,y)= \varnothing$, but that   $\cA'_*(x,y)$ is always nonempty.
Using  a computer, 
we have checked this conjecture  for all finite Weyl groups of rank at most 6, as well as for the groups of the exceptional types $H_3$ and $H_4$.
Below, we will show that the desired identity holds in the case when $x=1$ for all finite Coxeter groups.
\end{remark}

As usual, we abbreviate by writing $\cA'_*(y) = \cA'_*(1,y)$ and $\cA'(x,y) = \cA'_{\id}(x,y)$ with analogous conventions for $\pB_*(x,y)$.
We begin by noting  the following ``duality'' similar to Corollary \ref{lastdia-prop}.
As in Section \ref{duality-sect}, when $W$ is finite we write $w_0 \in W$ for its   longest element.

\begin{proposition}\label{dual'-prop}
Assume $W$ is finite  and define $\diamond \in \Aut(W,S)$ by $ x^\diamond = w_0 x^* w_0$. Taking inverses is then a bijection $ \cA_\diamond'(x,y) \leftrightarrow \cA'_*(w_0y,w_0x) $ and $ \pB_\diamond(x,y) \leftrightarrow \pB_*(w_0y,w_0x)$
for all $x,y \in \I_\diamond(W)$.
\end{proposition}

\begin{proof}
We have
$w^\diamond y   \leq xw $ $ \Leftrightarrow$ $ (w_0 x)w \leq w^* (w_0y)$ $\Leftrightarrow$ $ (w^{-1})^*( w_0x) \leq (w_0y) w^{-1}$
for  $x,y \in \I_\diamond(W)$ and $w \in W$  by standard properties of the Bruhat order; see \cite[Chapter 2]{CCG}.
\end{proof}

We recall the \emph{lifting property} of the Bruhat order, stated for example as \cite[Proposition 2.10]{H2}: 
if $x,y \in W$ are such that  $x \leq y$ and $s \in \DesR(y)$, then (i) $xs\leq y$ and (ii) if $s \in \DesR(x)$ then $xs\leq ys$.
As $x\leq y $ if and only if $x^{-1} \leq y^{-1}$, there is  an analogous left-handed version of this statement.
We state a technical lemma.

\begin{lemma}\label{atomlem1} Let $w,z \in W$ and $y \in \I_*$ and  $s \in \DesR(w) \cap \DesR(y) \cap \DesR(z)$. 
If  $w^*s^*(y\act_* s) \leq zs$ then $w^*y \leq z.$
\end{lemma}

\begin{proof}
Let $u = w^*s^*(y\act_* s)$ and suppose  $u \leq zs$. Then $u \leq z$, and from this we deduce by the lifting property of the Bruhat order that   $us \leq z$. There are two cases to consider according to whether $y\act_* s $ is $s^*y=ys$ or $s^*ys$,
but in either case   $w^*y \in \{u,us\}$ so  $w^*y \leq z$.
\end{proof}

We can   prove the following weaker form of Conjecture \ref{atoms-conj}.   

  \begin{proposition}\label{bruhatchar-prop} If $(W,S,*)$ is any twisted Coxeter system and $x,y \in \I_*$, then $\cA_*(x,y) \subset \pB_*(x,y)$.
  \end{proposition}

  \begin{proof}
  This result holds trivially when $\ellhat_*(x,y) < 0$ as then  $\cA_*(x,y)= \emptyset$, so assume $\ellhat_*(x,y)\geq 0$.
 We fix  $w \in \cA_*(x,y)$ and proceed by induction. 
The desired property  is clear when $\ellhat_*(x,y) = 0$ since then $w=1$ and   $w^*y = y = x = xw$.
Assume $\ellhat_*(x,y)=\ell(w)>0$, choose $s \in \DesR(w)$, and set $z = xw$.  
By construction   $\ell(z) = \ell(x) + \ell(w)$ so $s \in \DesR(z)$,
and by Lemma \ref{atomdes-lem} we   have $s \in \DesR(y)$.
As $ws \in \cA_*(x,y\act_* s)$ by Proposition \ref{atomdes-prop}, we may  assume by induction that $(ws)^*(y\act_* s) \leq x(ws) = zs$; given this, it is immediate from Lemma \ref{atomlem1} that $w^*y \leq z=xw$.
%
  \end{proof}

We also have the following related statement.

  \begin{proposition}\label{leq-lem}
  Let $x,y,z \in \I_*$ and $w \in \cA_*(x,z)$.
  Then (a) if $w \in \pB_*(x,y)$ then $y \leq z$ and 
(b) if $w \in \pB_*(y,z)$ then $x \leq y$. 
  \end{proposition}

  \begin{proof}
We adopt the following convention in this proof: given    elements $a_1,a_2,\dots,a_k \in W$, we say that $b=a_1a_2\cdots a_k$ is a \emph{reduced product} if $\ell(b) = \sum_{i=1}^k \ell(a_i)$.

  Fix $(s_1,s_2,\dots,s_k) \in \cR(w)$.  Set $u_0 = x$ and for   $i \in [k]$ define $u_i = u_{i-1}\act_* s_i$ so that $z = u_k$.
For each $i \in [k]$ define $e_i$ to be $s_i$ if $u_i = s_i^*  u_{i-1}  s_i$ and 1 if $u_i = s_i^*  u_{i-1} = u_{i-1}  s_i$.
Observe  that 
\[
  u_j = e_i^*\cdots e_2^* e_1^* x  s_1 s_2 \cdots s_i
  = s_k^*  \cdots s_2^*  s_1^*  x e_1  e_2  \cdots e_k
 \] 
  and that both of these expressions are reduced products.
  Now, to prove part (a),  suppose $w^*y \leq xw$, define  $a_0 = w^*y$ and $b_0=xw$,  and for  $i \in [k]$ let
$
   a_i = s_i^*  \cdots s_2^* s_1^* w^* y
   $
   and
   $
   b_i 
   = e_i^*  \cdots e_2^* e_1^* x w
   .
$
Observe that  $a_k = y$ and $b_k = z$ and that $b_i > b_{i-1}$ whenever $e_i\neq 1$.
One checks that $s_i^* \in \DesL(b_{i})$ for all $i$;  this  clearly holds if $e_i=s_i$, and if $e_i=1$ then writing
\[
  b_i = b_{i-1}=  u_{i-1}  s_{i}  s_{i+1} s_{i+2} \cdots s_k = s_{i}^*  u_{i-1}   s_{i+1}  s_{i+2}\cdots s_k
\]
while noting that the right hand side is a reduced product
implies   that $s_i^* \in \DesL(b_i)$.
From this, it follows by induction that $a_i \leq b_i$ for all $i$: the case $i=0$ holds by assumption, while
if $i>0$ and $a_{i-1} \leq b_{i-1}$, then $a_i \leq b_i$ follows from the left-handed version of the lifting property of the Bruhat order.
Hence $y = a_k \leq b_k = z$.

To prove part (b), suppose  $w^*z \leq yw$, define  $c_0 = w^*z$ and $d_0=yw$,  and for each $i \in [k]$ let
 $
   c_i = x  e_1  e_2 \cdots e_{k-i}
   $
   and
   $
   d_i 
   = y  s_1  s_2 \cdots s_{k-i}
   .
   $
We have $c_k = x$ and $d_k = y$ and  $c_{i+1} < c_{i}$ whenever $e_{k-i}\neq 1$, and as in the previous case one checks that if $e_{k-i}=1$ then $s_{k-i} \notin \DesR(c_i)$. 
   Using these facts and the lifting property, one argues by induction as before that $c_i \leq d_i$ for all $i$, whence $x = c_k \leq d_k = y$. We omit the details which are similar to those in the preceding case.
  \end{proof}

For $x,z \in \I_*$, let
$\cA_*(x,-) = \bigcup_{y\in \I_*} \cA_*(x,y)
$
and
$
\cA_*(-,z) = \bigcup_{y \in \I_*} \cA_*(y,z).$

\begin{corollary}
Let $x,y \in \I_*$ with $x\leq_{T,*} y$. Then $\cA_*(x,y)$ is equal to the set of minimal-length elements in each of the intersections $\cA_*(x,-) \cap \pB_*(x,y)$ and $\cA_*(-,y)\cap \pB_*(x,y)$.
\end{corollary}

\begin{proof}
We know that $\cA_*(x,y) \subset \pB_*(x,y)$ by Proposition \ref{bruhatchar-prop}. The corollary follows since by Proposition \ref{leq-lem}
if $w \in \cA_*(x,-) \cap \pB_*(x,y)$ then $w \in \cA_*(x,y')$ for some twisted involution $y' \geq y$,
while if $w \in \cA_*(-,y)\cap \pB_*(x,y)$ then $w \in \cA_*(x',y)$ for some   $x'\leq x$, so in either case  $\ell(w) \geq \ellhat_*(x,y)$.
\end{proof}

Recall that $\Mon{S}$ denotes the free monoid on $S$.
A \emph{subword} of $\e \in \Mon{S}$ is any ordered subsequence of $\e$. 
For $\e=(s_1,s_2,\dots,s_k) \in \Mon{S}$  we define 
\be\label{delta*}
\delta(\e) = s_1 \dem s_2 \dem \cdots \dem s_k \in W
 \qquand
 \hat \delta_*(\e) 
 =  \idem{\idem{\idem{\idem{1}{s_1}}{ s_2}}{\cdots}}{s_k} \in \I_* 
 \ee
 where $\dact_*$ and $\dem$ are as in  \eqref{demprod} and \eqref{dact-def}. Note that $ \hat \delta_*(\e) =\idem{1}{\delta(\e)}$ and that 
 by definition 
$\hat\delta(s_1,s_2,\dots,s_k) = \delta(s_k^*,\dots,s_2^*,s_1^* , s_1,s_2,\dots,s_k)$.
The following result appears as  \cite[Lemma 3.4]{KM}: 

\begin{lemma}
[Knutson and Miller \cite{KM}]
\label{subword-lem}
If $\e \in \Mon{S}$ and $w \in W$
then $w \leq \delta(\e) $ if and only if  $\e$ has a subword in $\cR(w)$.
\end{lemma}

In general, the sets $\DemA_*(x,y)$ and $\pB_*(x,y)$ are not as clearly related as $\cA_*(x,y)$ and $\cA'_*(x,y)$. 
However,  the following   identity does hold:

\begin{theorem}\label{b'-thm} If $W$ is finite with longest element $w_0$, then $\DemA_*(w_0) = \pB_*(w_0)$.
\end{theorem}

\begin{proof}
Recall (e.g., from \cite[\S2.3]{CCG}) that $w\mapsto w^*w_0$ is an anti-involution of $(W,\leq)$ and that $\ell(w^*w_0) = \ell(w_0) - \ell(w)$ for all $w \in W$.    Therefore, if $v \geq w \in \pB_*(w_0)$, then $v^*w_0 \leq w^*w_0 \leq w \leq v$ so $v \in \pB_*(w_0)$. 
By construction   every $v \in \DemA_*(w_0)$ satisfies $v \geq w$ for some $w \in \cA_*(w_0)$, so as $\cA_*(w_0)  \subset \pB_*(w_0)$ by Proposition \ref{bruhatchar-prop}, it follows that $\DemA_*(w_0) \subset \pB_*(w_0)$.

To show the reverse inclusion, 
note that each $w \in W$ belongs to $\DemA_*(y)$ for a unique twisted involution $y \in \I_*$. Hence, it suffices to show that $\pB_*(w_0) \cap \DemA_*(y) = \varnothing$ unless $y=w_0$.
For this,
suppose $w \in   \pB_*(w_0) \cap \DemA_*(y)$ for some $y \in \I_*$.
Since $w^*w_0 \leq w$, the element $w^*w_0$ has a reduced word $(t_1,t_2,\dots,t_p)$ that is a subword of a reduced word 
$(s_1,s_2,\dots,s_q) \in \cR(w)$.
As the word $\e=(s_q^*,\dots,s_2^*,s_1^*,t_1,t_2,\dots,t_p) $ evidently belongs to $ \cR(w_0)$,
it follows by Lemma \ref{subword-lem}  that 
$ w_0 = \delta(\e) \leq \delta(s_q^*,\dots,s_2^*,s_1^*,s_1,s_2,\dots,s_q) = \hat\delta_*(s_1,s_2,\dots,s_q) = y$,
so  $y=w_0$ as desired.
\end{proof}

The following is immediate from the theorem and Proposition \ref{dual'-prop}

\begin{corollary}\label{dualcor1a} Assume $W$ is finite  and define $\diamond \in \Aut(W,S)$ by $w^\diamond = w_0 w^* w_0$.
The map  $w\mapsto w^{-1}$ is then a bijection $\DemA_*(w_0) \leftrightarrow \DemA_\diamond(w_0)$.
\end{corollary}


For the   next theorem,
we require one other    property of the Bruhat order \cite[Lemma 2.2.10]{CCG}. 

\begin{lemma}[See \cite{CCG}] \label{techbr-lem}
Let $(W,S)$ be any Coxeter system and set $T = \{ wsw^{-1} : (w,s) \in W\times S\}$. If $a,b \in W$ and $t \in T$ are such that $a<at$ and $b<tb$ then $ab<atb$.
\end{lemma}

\begin{theorem}\label{lastbr-thm}
Suppose $W$ is finite with longest element $w_0$.
Let
 $x,y \in \I_*$ with $x\leq_{T,*} y$.
 If either 
$ \ell(x) = \ellhat_*(x) $ or $ \ell(y,w_0) = \ellhat_*(y,w_0)$,
then $\cA_*(x,y) = \cA'_*(x,y).$
\end{theorem}


\begin{proof}[Proof of Theorem~\ref{lastbr-thm}]
Since we know that $\cA_*(x,y) \subset \pB_*(x,y)$ by Theorem \ref{bruhatchar-prop},
it suffices to prove that (a) $\ell(u) \geq \ellhat_*(x,y)$ for all $u \in \pB_*(x,y)$
and
(b) $\cA'_*(x,y) \subset \cA_*(x,y)$. It will be helpful to distinguish between these two properties, although we note that (b) actually implies (a).

First assume $x=1$ so that $y \in \I_*$ is arbitrary. We  argue that (a) and (b) hold in this special case by induction on $\ellhat_*(y,w_0)$. If $\ellhat_*(y,w_0)=0$ so that $y=w_0$, then what we want to show is immediate from Theorem \ref{b'-thm}.
Suppose $y<w_0$ and assume that (a) and (b) hold with $x=1$ when $y$ is replaced by any twisted involution of greater length.
Fix $u \in \pB_*(y)$. 
Since $y<w_0$ there exists $s \in \{s_1,s_2,\dots,s_{n-1}\}$ such that $y<ys$. Set $z=y\act_*s$ and note that $(us)^*z \in \{ u^*y, u^*ys\}$. First suppose $u<us$. Then $us \in \pB_*(z)$ by  the lifting property,
so  by our inductive hypothesis   $\ell(us) = \ell(u)+1 \geq \ellhat_*(z) = \ellhat_*(y) + 1$ and if $\ell(u)=\ellhat_*(y)$ so that $\ell(us) = \ellhat_*(y)+1$ then $us \in \cA_*(z)$.
This  implies that (a) and (b) hold
by Proposition \ref{atomdes-prop}.
Alternatively suppose $us < u$;  one of the following   cases then occurs:
\ben
\item[(i)] 
Suppose $ s^*y=ys$. Then $u^*z=u^*s^*y$ and $y<ys=s^*y $, so as $u^*s^* < u^*$, it follows by Lemma \ref{techbr-lem}  with $a=u^*s^*$ and $b=y$ and $t=s^*$ 
that $u^*z = ab < atb = u^*y \leq u$.

\item[(ii)] Suppose $s^*y\neq ys$. Then $u^*z = us^*ys$ and $y<ys < s^*ys$, so it follows by Lemma \ref{techbr-lem}  with $a=u^*s^*$ and $b=ys$ and $t=s^*$ 
 that $u^*z = ab < atb = u^*ys$. Since $us<u$ and $u^*y\leq u$,  the lifting property implies that $u^*ys \leq u$, so we  again have $u^*z < u$.

\een
In either case we have $u \in \pB_*(z)$, so  by induction  
 $\ell(u) \geq \ellhat_*(z) > \ellhat_*(y)$ and therefore  $u\notin \cA'_*(y)$.
   From this analysis, we conclude   that (a) and (b) hold when $x=1$.

Now  suppose 
 that $\ell(x)=\ellhat_*(x)$. 
If $x=1$ then properties (a) and (b) hold by what we have just shown. Assume $x>1$, choose $s \in \DesR(x)$,  set $x' = x\act_*s = s^*x = xs$,
and fix $u \in \pB_*(x,y)$. If
 $su>u$, then it follows by  Lemma \ref{techbr-lem} that $x'u=s^*xu<xu$ so by the lifting property  $s^*u^*y \leq xu=x'su$. By induction on $\ellhat_*(x)$, if $su>u$ then  
$\ell(su) = \ell(u)+1 \geq \ellhat_*(x',y) = \ellhat_*(x,y)+1
$,
and  if  $\ell(su) = \ellhat_*(x',y)$ then $ su \in \cA_*(x',y)$. We deduce that:
\ben
\item[(1)] If $su>u$ then  $\ell(u) \geq \ellhat_*(x,y)$, and if $\ell(u) = \ellhat_*(x,y)$ then  $u \in \cA_*(x,y)$ by Proposition \ref{atomdes-prop}.
\item[(2)] If $su<u$, then    Lemma \ref{techbr-lem} implies  $u^*y \leq xu < s^*xu = x(su)$, so by the lifting property $su \in \pB_*(x,y)$, and  by (1) it holds that  $\ell(u) = \ell(su)+1 > \ellhat_*(x,y)$ whence $u \notin\cA'_*(x,y)$. 
 \een
Thus (a) and (b) continue to hold if $\ell(x) = \ellhat_*(x)$.

Finally, suppose $\ell(y,w_0) = \ellhat_*(y,w_0)$. Define $w^\diamond = w_0 w^* w_0$ for $w \in W$. 
Then, by Corollary \ref{lastdia-prop} and Proposition \ref{dual'-prop} and the cases already considered, it holds that
$\ell(w_0y) = \ellhat_\diamond(w_0y)$,
so we have 
$
\pB_*(x,y) = \pB_\diamond(w_0y,w_0x)^{-1} \subset \{ u \in W : \ell(u) \geq \ellhat_\diamond(w_0x,w_0y)\}^{-1} = \{ u \in W : \ell(u) \geq \ellhat_*(x,y)\}$
and likewise
$ \cA'_*(x,y) = \cA'_\diamond(w_0y,w_0x)^{-1} 
\subset
 \cA_\diamond(w_0y,w_0x)^{-1} = \cA_*(x,y)
 $ as desired.
\end{proof}

We highlight a  few special   cases of Theorem \ref{lastbr-thm}. As $\ell(1) =\ellhat_*(1) = 0$, the following holds:

\begin{corollary}
If $W$ is finite then $\cA_*(y) = \cA'_*(y)$ for all $y \in \I_*$.
\end{corollary}

Write $\cAfpf(x) = \cA(\wfpf{n},x)$ and $\cAfpf'(x) = \cA'(\wfpf{n},x)$ for $x \in \Ifpf(S_{2n})$, where as in the introduction  $\wfpf{n}= s_1 s_3s_5\cdots s_{n-1} \in S_{2n}$.
Since $\ell(\wfpf{n}) = \ellhat(\wfpf{n})=n$, the following also holds.

\begin{corollary}\label{last-cor}
$\cA(x) =\cA'(x)$ for all $x \in \I(S_n)$ and $\cAfpf(y) = \cAfpf'(y)$ for all $y \in \Ifpf(S_{2n})$.
\end{corollary}

\section{Relative atoms in type $A$}\label{atomSn-sect}
\def\precB{\prec}
\def\preceqB{\preceq}

The atoms of involutions in symmetric groups have been studied previously by Can, Joyce, and Wyser \cite{CJ,CJW}. Their results 
\cite[Theorem 3.7]{CJ}, \cite[Theorem 2.5]{CJW}, and \cite[Corollary 2.16]{CJW}  
provide a complete description of the elements  $w \in \cA(x,y)$ in the  cases when $x=1$ or $x=s_1s_3s_5\cdots s_{n-1}$ (with $n$ even); this description consists of a   simple list of numeric inequalities involving the one-line representation of $w$ and the cycle structure of  $y$. In this section we prove a common generalization of these results which describes $\cA(x,y)$ for all involutions $x,y \in S_n$. Our arguments are self-contained; our main result, Theorem \ref{atomSn-thm}, will follow from several auxiliary lemmas and definitions.

Let $\PP = \{1,2,\dots\}$ denote the set of positive integers and define $ [n] = \{ i \in \PP : i\leq n\}$ for  nonnegative integers $n \in \NN$.
As  we are concerned solely with atoms of  involutions in  symmetric groups, throughout this  section
we let $s_i =(i,i+1) \in S_n$ for $i \in [n-1]$.
We define
\[\cI(S_n) = \{ x \in S_n : x=x^{-1}\}\] and write
$ \act = \act_\id $ and $ \dact = \dact_\id $ and $ \hat\cR = \hat \cR_\id$ and $ \cA = \cA_\id$ and
$
\DemA = \DemA_\id$
and 
$\ellhat=\ellhat_\id$
for the notations given in Sections \ref{intro-sect} and \ref{prelim-sect}. 
Recall  that   $\hat\cR(y) = \hat\cR(1,y)$ and $\cA(y) = \cA(1,y)$ and $\cB(y) = \cB(1,y)$. 

The involutions in $S_n$ are  the permutations of $[n]$ whose cycles each have length 1 or 2. 
We have $x \act s_i = s_i  x  s_i$ if and only if $i$ and $i+1$ belong to distinct cycles of $x$ and are not both fixed points; otherwise $x \act s_i = s_i  x = x s_i$. 
Recall that if $w \in S_n$ then 
\be\label{des-sn} \DesR(w) =\{ s_i : i \in [n-1]\text{ such that }w(i)>w(i+1)\}.
\ee
Hence, we have $\sidem{x}{s_i} = x \act s_i$ if and only if $x(i)>x(i+1)$, and otherwise $\sidem{x}{s_i} =x$.
In reasoning about the operations $\act $ and $\dact$, it is often useful to draw an involution $x \in \I(S_n)$ as the incomplete matching on the numbers $1,2,\dots,n$, ordered from left to right in a line, whose connected components are the cycles of $x$.
For example, for $x=(1,5)(2,6)$ we would draw  
\[
\ba \\[-10pt]
x \ =\  {\xy<0cm,-.05cm>\xymatrix@R=.3cm@C=.2cm{ 
*{\bullet}   \ar @/^1.pc/ @{-} [rrrr]   & *{\bullet}  \ar @/^1.pc/ @{-} [rrrr]  & *{\bullet} & *{\bullet} & *{\bullet} & *{\bullet}
}\endxy}
\qquand
x\act s_1  \ =\  {\xy<0cm,-.05cm>\xymatrix@R=.3cm@C=.2cm{ 
*{\bullet}   \ar @/^1.2pc/ @{-} [rrrrr]   & *{\bullet}  \ar @/^0.8pc/ @{-} [rrr]  & *{\bullet} & *{\bullet} & *{\bullet} & *{\bullet}
}\endxy}
\qquand
x\act s_3 \ =\  {\xy<0cm,-.05cm>\xymatrix@R=.3cm@C=.2cm{ 
*{\bullet}   \ar @/^1.pc/ @{-} [rrrr]   & *{\bullet}  \ar @/^1.pc/ @{-} [rrrr]  & *{\bullet}  \ar @/^.5pc/ @{-} [r]& *{\bullet} & *{\bullet} & *{\bullet}
}\endxy}.
\ea
\]
In relation to this model, $\act $ and $\dact $ correspond to simple ``graph theoretic'' operations like adding or deleting an edge between adjacent vertices, or switching the locations of adjacent vertices.
Moreover, $s_i$ is a descent of $x$ if and only if, informally, the transition $x \mapsto x\act s_i$ results in a matching whose edges are ``smaller'' or have fewer ``nestings.''


We   introduce some less standard notation. First, we define the \emph{cycle set} and \emph{fixed-point set} of an involution $y \in \I(S_n)$ by
\[\Cyc(y) =  \{ (a,b) \in [n]\times [n] :  a \leq b=y(a)\}
\qquand
\Fix(y) = \{ a \in [n] : y(a) = a\}.\]
Note that $(i,i) \in \Cyc(y)$ if $i \in \Fix(y)$. Next, we define the \emph{extended cycle set} of $y \in \I(S_n)$ 
by
\[\Gamma(y) = 
\Cyc(y) \cup \{ (a,b) \in \Fix(y) \times \Fix(y) : b<a\}.\]
Equivalently, $\Gamma(y)$ consists of all pairs 
$(a,b) \in [n]\times [n]$ with $ \{a,b\}=\{y(a),y(b)\}$ and $ y(b)  \leq  y(a)$.
We suppress the dependence on $n$ in these definitions for convenience.

\begin{example}
If $y =(1,6)(3,5) \in \I(S_6)$ then 
$\Gamma(y) = \{ (1,6),(2,2),(3,5),(4,4), (4,2)\}$.
\end{example}

\begin{remark}
Define $\Delta^+ = \{ \varepsilon_i - \varepsilon_j \in \RR^n : 1\leq i< j \leq n\}$ 
and $\Delta^- = - \Delta^+$, so that 
$\Delta = \Delta^+ \cup \Delta^-$
is the root system of type $A_{n-1}$ on which $S_n$ acts by permuting coordinates.
Then we can write
$ \Gamma(y) = \{ \alpha \in \Delta^- : \alpha = y\alpha \} \cup \{ (i,i) : i \in \Fix(y)\} \cup
\{ \alpha \in \Delta^+ : \alpha = -y\alpha \}
$
 for  $y \in \I(S_n)$,
where here we identify  pairs $(i,j) \in [n]\times [n]$ with  roots $\varepsilon_i -\varepsilon_j$.
\end{remark}

The symmetric group $S_n$ acts on $[n]\times [n]$ by simultaneously permuting coordinates, so that  if $w \in S_n$ and  
 $\gamma=(a,b) \in [n]\times [n]$ then  $w\gamma = w(a,b) = (w(a),w(b))$.
   
   \begin{lemma}\label{a-lem1} Let $y \in \I(S_n)$ and $s \in \DesR(y)$. Then $s\gamma \in \Gamma(y\act s)$ for all $\gamma \in \Gamma(y)$.
\end{lemma}

\begin{proof}
Proving this lemma is a straightforward exercise; we omit the details.
\end{proof}

In the next few pages we introduce a few more involved constructions, which will serve as bookkeeping devices to simplify the statement of  more complex conditions involving the cycles of an involution.

We  define a \emph{colored involution} on $[2n]$ to be a partial matching of $[2n]$ whose vertices are colored
by the numbers $1,2,\dots,n$ such that there are exactly two vertices of each color,
and such that any two connected vertices  have the same color.
One    identifies these objects with elements of the wreath product group
$G(n,2n) = S_{2n} \wr \ZZ/n\ZZ$ as follows.
Recall that elements of $G(n,2n)$ are pairs $(\theta, v)$ with $\theta \in S_{2n}$ and $v \in (\ZZ/n\ZZ)^{2n}$, and that multiplication in the group is given by 
$(\theta,v)(\theta',v') = (\theta\theta', v + \theta(v'))$  with $S_{2n}$ acting on $(\ZZ/n\ZZ)^{2n}$ by permuting coordinates.
A colored involution $\alpha$ on $[2n]$
corresponds to the element  $(\theta,v) \in G(n,2n)$ 
where $\theta$ is the involution in $S_{2n}$ whose 2-cycles are the  connected components of $\alpha$, and $v \in (\ZZ/n\ZZ)^{2n}$ is the vector whose $i$th coordinate is the color of vertex $i$ in $\alpha$, taken modulo $n$.
Observe that the   element $(\theta,v)$ constructed in this way belongs to $\I_*(G(n,2n))$
where  $* \in \Aut(G(n,2n))$ is the involution $(\theta,v) \mapsto (\theta,-v)$.

The number of colored involutions on $[2n]$ is $(2n)!$ and there is a simple bijection $\tau$ from permutations in $S_{2n}$ to colored involutions: given $w \in S_{2n}$, let $\tau(w)$ be the colored involution on $[2n]$ in which for each $i \in [n]$, vertices $w(2i-1)$ and $w(2i)$ are colored by $i$, and  connected by an edge if and only if $w(2i-1) > w(2i)$. We   view 
$\tau$ as an injective map
\[ \tau : S_{2n} \hookrightarrow \I_*\(G(n,2n)\).\]
We make no distinction between $\theta \in S_{2n}$ and $(\theta,0) \in G(n,2n)$ and so view $S_{2n}\subset G(n,2n)$.
With this convention, 
it holds that 
$ \tau(w) = w  (\theta_w,c) w^{-1}$ 
where $\theta_w$ is the product of the  odd descents of $w$ (that is, the elements $s_i \in \DesR(w)$ with $i$ odd)
and 
$
 c= (1,1,2,2,\dots,n,n) \in(\ZZ/n\ZZ)^{2n}.$
For $\alpha   \in \I_*(G(n,2n))$ and $i \in [2n-1]$,
we define $\alpha \act s_i$ as usual 
to be either $ \alpha s_i$ when $s_i \alpha = \alpha s_i$ or $s_i \alpha s_i$ when $s_i \alpha \neq \alpha s_i$.
Although $\tau$ is not a group homomorphism, the following property does hold:

\begin{proposition}
If $s \in \{s_1,s_2,\dots,s_{2n-1}\}$ and $w \in S_{2n}$ then $\tau(sw) = \tau(w) \act s$.
Hence, the operation $\act $   extends to a right action of $S_{2n}$ on colored involutions on $[2n]$.
\end{proposition}

\begin{proof}
We leave this simple exercise to the reader.
\end{proof}

From
 any sequence of pairs in $\ZZ^2$ we construct a colored involution in the following way.
Let $\std$ be the usual \emph{standardization map} on words, so that if $\e$ is a finite sequence of  integers in which the $i$th smallest letter appears exactly $m_i$ times,
then $\std(\e)$ is the permutation whose one-line representation is obtained by replacing the letters equal to the smallest letter in $\e$ with $1,2,\dots,m_1$ (in that order),  the letters equal to the second smallest in $\e$  with $m_1+1,m_1+2,\dots,m_1+m_2$ (in that order), and so on. For example, $\std[1,2,1,1,2] = [1,4,2,3,5]$.
Now, for $(a_1,b_1),(a_2,b_2),\dots,(a_n,b_n) \in \ZZ^2$  we define 
\be\label{sigma-eq}
\sigma\( (a_1,b_1),(a_2,b_2),\dots,(a_n,b_n) \)
=
\tau \circ \std[ b_1,a_1,b_2,a_2,\dots,b_n,a_n] \in \I_*\( G(n,2n)\).
\ee
Equivalently, $\sigma\( (a_1,b_1),(a_2,b_2),\dots,(a_n,b_n) \)$ is the colored involution on $[2n]$
in which, writing $ \std[ b_1,a_1,b_2,a_2,\dots,b_n,a_n] = [b_1',a_1', b_2', a_2',\dots,b_n', a_n'] \in S_{2n}$,
 it holds that vertices $a_i'$ and $b_i'$ have color $i$, and are connected if and only if $a_i < b_i$.

We are primarily interested in this construction when $n=2$.
In this special case,
we draw the colored matching $\sigma\( (a_1,b_1),(a_2,b_2)\)$ on $\{1,2,3,4\}$
such that the vertices $a_1'$ and $b_1'$ are labeled as $\bullet$
while the vertices $a_2'$ and $b_2'$ are labeled as $\circ$,
as in the following example.

\begin{example} If $a,b,c,d$ are integers with $a<b<c<d$ then
\[ 
\ba
\\[-8pt]
&\sigma((a,c),(b,d)) 
\ =\ 
   {\xy<0cm,-.05cm>\xymatrix@R=.3cm@C=.2cm{ 
*{\bullet}   \ar @/^.9pc/ @{-} [rr]   & *{\circ}  \ar @/^.9pc/ @{-} [rr]  & *{\bullet} & *{\circ}
}\endxy}
\\
  \\
&\sigma((c,a),(b,d)) 
\ =\ 
   {\xy<0cm,-.05cm>\xymatrix@R=.3cm@C=.2cm{ 
*{\bullet}      & *{\circ}  \ar @/^.9pc/ @{-} [rr]  & *{\bullet} & *{\circ}
}\endxy}
\\
\\
&\sigma((c,a),(d,b)) 
\ =\ 
   {\xy<0cm,-.05cm>\xymatrix@R=.3cm@C=.2cm{ 
*{\bullet}      & *{\circ}    & *{\bullet} & *{\circ}
}\endxy}
\\[-10pt]\\
\ea
\qquand
\ba
\\[-10pt]
&\sigma((b,d),(a,c)) 
\ =\ 
   {\xy<0cm,-.05cm>\xymatrix@R=.3cm@C=.2cm{ 
*{\circ}   \ar @/^.9pc/ @{-} [rr]   & *{\bullet}  \ar @/^.9pc/ @{-} [rr]  & *{\circ} & *{\bullet}
}\endxy}
\\
  \\
&\sigma((d,a),(b,c)) 
\ =\ 
   {\xy<0cm,-.05cm>\xymatrix@R=.3cm@C=.2cm{ 
*{\bullet}      & *{\circ} \ar @/^.6pc/ @{-} [r]   & *{\circ} & *{\bullet}
}\endxy}
\\
\\
&\sigma((a,b),(c,c)) 
\ =\ 
   {\xy<0cm,-.05cm>\xymatrix@R=.3cm@C=.2cm{ 
*{\bullet} \ar @/^.6pc/ @{-} [r]      & *{\bullet}    & *{\circ} & *{\circ}
}\endxy}.
\\[-10pt]\\
\ea
\]
\end{example}


Recall that $(\theta,v)^* = (\theta,-v) \in G(n,2n)$.
If $\alpha = (a,b)$ and $\beta = (c,d)$, then we define $\alpha \cap \beta= \{a,b\} \cap \{c,d\}$.
The following observation is clear from our definitions.

 \begin{observation}\label{a-obs}
 If $\alpha_1,\alpha_2,\dots,\alpha_n \in \ZZ^2$
are pairwise disjoint in the sense that $\alpha_i \cap \alpha_j = \varnothing$ for   $i\neq j$, 
then $\sigma(\alpha_1,\alpha_2,\dots,\alpha_n)^* = \sigma(\alpha_n,\dots,\alpha_2,\alpha_1)$.
 \end{observation}


Write $\pi$ for the  surjective group homomorphism
$ \pi : G(n,2n) \to S_{2n}$
given by $(\theta,v) \mapsto \theta$,
and define $\prec$ as the weakest partial order on $\I_*\(G(n,2n)\)$ such that 
\be\label{prec-eq} \alpha \act s \prec \alpha \qquad\text{whenever}\qquad \pi(\alpha\act s) <_T \pi(\alpha)\ee
for  $\alpha = (\theta,v) \in \I_*(G(n,2n))$ and $s \in \{s_1,s_2,\dots,s_{2n-1}\}$.
Note that it is not true that $\pi(\alpha \act s)  = \pi(\alpha) \act s$, and therefore 
 $\alpha \preceq \beta$ is not equivalent to $\pi(\alpha) \leq_T \pi(\beta)$.
Figure \ref{arr-fig} shows a lower interval in $\prec$  when $n=2$; in this case,
the figure completely determines $\prec$ in the following sense.

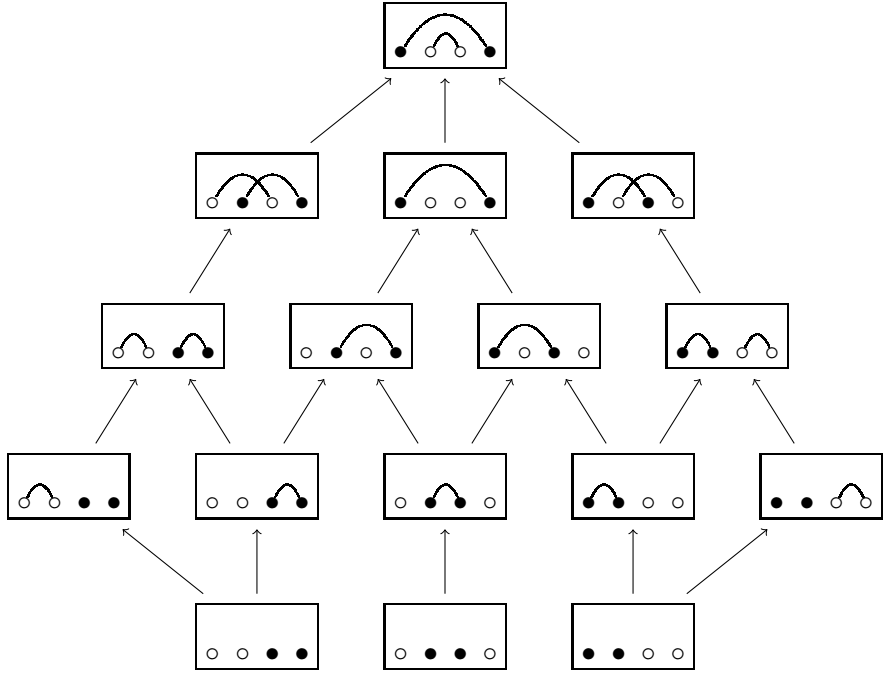
\begin{figure}[h]
\[
\begin{tikzpicture}
  \node (max) at (0,2) {
  \boxed{\xy\xymatrix@R=.3cm@C=.2cm{ 
  { }
  \\
*{\bullet}   \ar @/^1.2pc/ @{-} [rrr]   & *{\circ}  \ar @/^.6pc/ @{-} [r]  & *{\circ} & *{\bullet}
}\endxy}
};
  \node (a) at (-2.5,0) {
  \boxed{\xy\xymatrix@R=.3cm@C=.2cm{ 
  { }
  \\
*{\circ}   \ar @/^0.9pc/ @{-} [rr]   & *{\bullet}  \ar @/^.9pc/ @{-} [rr]  & *{\circ} & *{\bullet}
}\endxy}
};
  \node (b) at (0,0) {
  \boxed{\xy\xymatrix@R=.3cm@C=.2cm{ 
  { }
  \\
*{\bullet}   \ar @/^1.2pc/ @{-} [rrr]   & *{\circ} [r]  & *{\circ} & *{\bullet}
}\endxy}
};
  \node (c) at (2.5,0) {
  \boxed{\xy\xymatrix@R=.3cm@C=.2cm{ 
  { }
  \\
*{\bullet}   \ar @/^.9pc/ @{-} [rr]   & *{\circ}  \ar @/^.9pc/ @{-} [rr]  & *{\bullet} & *{\circ}
}\endxy}
};
  \node (d) at (-3.75,-2) {
  \boxed{\xy\xymatrix@R=.3cm@C=.2cm{ 
  { }
  \\
*{\circ}   \ar @/^.6pc/ @{-} [r]   & *{\circ}   & *{\bullet}  \ar @/^.6pc/ @{-} [r] & *{\bullet}
}\endxy}
};
  \node (e) at (-1.25,-2) {
  \boxed{\xy\xymatrix@R=.3cm@C=.2cm{ 
  { }
  \\
*{\circ}   & *{\bullet}  \ar @/^.9pc/ @{-} [rr]  & *{\circ}  & *{\bullet}
}\endxy}
};
  \node (f) at (1.25,-2) {
  \boxed{\xy\xymatrix@R=.3cm@C=.2cm{ 
  { }
  \\
*{\bullet}   \ar @/^.9pc/ @{-} [rr]   & *{\circ}    & *{\bullet}  & *{\circ}
}\endxy}
};
  \node (g) at (3.75,-2) {
  \boxed{\xy\xymatrix@R=.3cm@C=.2cm{ 
  { }
  \\
*{\bullet}   \ar @/^.6pc/ @{-} [r]   & *{\bullet}   & *{\circ}  \ar @/^.6pc/ @{-} [r] & *{\circ}
}\endxy}
};
  \node (h) at (-5,-4) {
  \boxed{\xy\xymatrix@R=.3cm@C=.2cm{ 
  { }
  \\
*{\circ}  \ar @/^.6pc/ @{-} [r]   & *{\circ}   & *{\bullet} & *{\bullet}
}\endxy}
};
  \node (i) at (-2.5,-4) {
  \boxed{\xy\xymatrix@R=.3cm@C=.2cm{ 
  { }
  \\
*{\circ}    & *{\circ}   & *{\bullet}  \ar @/^.6pc/ @{-} [r] & *{\bullet}
}\endxy}
};
  \node (j) at (0,-4) {
  \boxed{\xy\xymatrix@R=.3cm@C=.2cm{ 
  { }
  \\
*{\circ}     & *{\bullet}  \ar @/^.6pc/ @{-} [r]  & *{\bullet} & *{\circ}
}\endxy}
};
  \node (k) at (2.5,-4) {
  \boxed{\xy\xymatrix@R=.3cm@C=.2cm{ 
  { }
  \\
*{\bullet}   \ar @/^.6pc/ @{-} [r]   & *{\bullet}   & *{\circ}  & *{\circ}
}\endxy}
};
  \node (l) at (5,-4) {
  \boxed{\xy\xymatrix@R=.3cm@C=.2cm{ 
  { }
  \\
*{\bullet}     & *{\bullet}   & *{\circ}  \ar @/^.6pc/ @{-} [r] & *{\circ}
}\endxy}
};
%
  \node (m) at (-2.5,-6) {
  \boxed{\xy\xymatrix@R=.3cm@C=.2cm{ 
  \\
*{\circ}  & *{\circ}   & *{\bullet}  & *{\bullet}
}\endxy}
};
  \node (n) at (0,-6) {
  \boxed{\xy\xymatrix@R=.3cm@C=.2cm{ 
  \\
*{\circ}     & *{\bullet}   & *{\bullet} & *{\circ}
}\endxy}
};
  \node (o) at (2.5,-6) {
  \boxed{\xy\xymatrix@R=.3cm@C=.2cm{ 
  \\
*{\bullet}   & *{\bullet}   & *{\circ} & *{\circ}
}\endxy}
};
%
  \draw  [<-]
(max)   edge  (a)
(a)  edge  (d)
(d)   edge  (h)
(h)  edge  (m)
(max)   edge  (b)
(b)  edge  (e)
(e)  edge  (i)
(i)  edge  (m)
(max)   edge  (c)
(c)  edge (g)
(g)  edge (l)
(l)  edge  (o)
(d)  edge (i)
(e)  edge (j)
(j)  edge  (n)
(b)  edge  (f)
(f)  edge  (j)
(f)  edge  (k)
(k)  edge  (o)
(g)  edge  (k)
;
\end{tikzpicture}
\]
\caption{Hasse diagram of a lower interval of $\precB$ on colored involutions}\label{arr-fig}
\end{figure}

\begin{observation}\label{precB-obs}
Let  $\alpha,\alpha' ,\beta,\beta' \in \ZZ^2$ be   pairwise disjoint
and set  $u=\sigma(\alpha,\alpha') $ and $v= \sigma(\beta,\beta')$.
Then $u \precB v$  if and only if there exists a directed path
in Figure \ref{arr-fig} from  $u$ to $v$ or from $u^*$ to $v^*$.
In particular, $u \preceqB v$ if and only if $u^*\preceqB v^*$.
\end{observation}

\begin{proof}
Since $w^* \act   s_i = (w\act   s_i)^*$ and   $\pi(w^*) = \pi(w)$ for all $w \in \I_*\(G(n,2n)\)$, we have $u \preceqB v $ if and only if $u^* \preceqB v^*$. 
As Figure \ref{arr-fig} depicts a lower interval with respect to $\precB$, the observation is  equivalent to the assertion that when $v$ is not minimal with respect to $\precB$, either $v$ or $v^*$ appears   in the figure. This property holds by inspection. 
\end{proof}

We  state some more substantial observations as the  next three  lemmas.

\begin{lemma}\label{a-lem2} Let $y \in \I(S_n)$ and $s \in \DesR(y)$. Then $\sigma(s\gamma,s\gamma')\preceqB \sigma(\gamma,\gamma')$ for all
$\gamma,\gamma' \in \Gamma(y)$ with $\gamma\cap \gamma' = \varnothing$.
\end{lemma}

\begin{proof}
This assertion holds nearly by construction. For $w \in S_n$, let $\supp(w) = \{ i \in [n] : w(i) \neq i\}$.
Fix   $\gamma = (a,b)$ and $\gamma' = (a',b')$ in $ \Gamma(y)$ with $\gamma\cap \gamma' = \varnothing$, and define $u = \sigma(s\gamma,s\gamma')$ and $v = \sigma(\gamma,\gamma')$. 
If $ \supp(s) \not \subset \{ a,b,a',b'\}$ then $u =v$.
If $\supp(s) = \{a,b\}$ or $\supp(s) = \{a',b'\}$ then 
since $s$ is a descent of $y$ it follows that
$u = s_j  v = v  s_j =v \act  s_j \precB v$ for some $j \in \{1,2,3\}$.
Suppose, alternatively, that $\supp(s)$ contains exactly one element in $\{a,b\}$ and one element in $\{a',b'\}$. 
One checks that if $a\neq b$ and $a'\neq b'$ then  $u =  s_j  v   s_j = v \act  s_j \precB v$ for some $j \in \{1,2,3\}$. In turn, if $a'=b'$ 
then necessarily
\[
v \ =\    {\xy<0cm,-.05cm>\xymatrix@R=.3cm@C=.2cm{ 
*{\bullet}   \ar @/^0.9pc/ @{-} [rrr]   & *{\circ}   & *{\circ} & *{\bullet}
}\endxy}
\qquand
u \in \Bigl\{\    {\xy<0cm,-.05cm>\xymatrix@R=.3cm@C=.2cm{ 
*{\bullet}   \ar @/^0.5pc/ @{-} [r]   & *{\bullet}   & *{\circ} & *{\circ}
}\endxy}
\ , \
{\xy<0cm,-.05cm>\xymatrix@R=.3cm@C=.2cm{ 
*{\circ}     & *{\circ}   & *{\bullet}  \ar @/^0.5pc/ @{-} [r] & *{\bullet}
}\endxy}
\
\Bigr\}
\]
whence 
$u \preceqB v$ by Observation \ref{precB-obs}.
When $a=b$, finally, one deduces that $u \preceqB v$  either by a similar argument, or directly from the previous case via Observations \ref{a-obs} and \ref{precB-obs}.
\end{proof}

As usual, let $w_0=[n,n-1,\dots,2,1]$ denote the longest element of $S_n$.

\begin{lemma}\label{w0-lem}
Let $u   \in S_n$.
Suppose for all $i,j\in [n]$ 
it holds that
$u(n+1-i) < u(i)$ when $i<n+1-i$
and
that   $u(j) < u(n+1-i)$ or $u(i) < u(j)$
when $i<j<n+1-i$.
Then $u \in \cA(w_0)$. 
\end{lemma}

\begin{proof}
Let $\cX_n$ denote the set of permutations $u \in S_n$ satisfying the given conditions.
Clearly $\cX_n \subset \cA(w_0)$ when $n \in \{1,2\}$,
so assume   $n>2$ and that the lemma holds when $n$ is replaced by any smaller positive integer.
Write 
 $\phi : S_{n-2} \to  S_n$ for the injective  homomorphism with $s_j \mapsto s_{j+1}$
for $j \in [n-3]$
and 
define $x,y \in \I(S_n)$
by $x=(1,n)$ and $y=\phi([n-2,n-1,\dots,1]) = [1,n-1,n-2,\dots,2,n]$. 
Since by induction every element of $\cX_{n-2}$ is an atom of the longest element of $S_{n-2}$, 
it follows by  Corollary \ref{morph-cor}
 that $\phi(\cX_{n-2}) \subset \cA(y)$. 
 In turn, since $y$ belongs to the parabolic subgroup of $S_n$ generated by $J = \{ s_2,s_3,\dots,s_{n-2}\}$
and since
$wx=xw$ and $\ell(xw) = \ell(x) + \ell(w)$  for all $w \in \langle J \rangle$, it holds in view of Proposition \ref{parred-prop} that $\cA(x) \cA(y) \subset \cA(xy) = \cA(w_0)$.
Thus, to prove the lemma it suffices to show that $\cX(n) \subset \cA(x) \phi(\cX_{n-2})$.

To this end, 
 fix $u \in \cX_n$ and observe that necessarily  $u(n) = i < i+1 = u(1)$ for some $i \in [n-1]$.
Define 
$a = s_is_{i-1} s_{i-2}\cdots s_1s_{i+1} s_{i+2}\cdots s_{n-1}
$
and 
set
$b = a^{-1}u.$
One checks that the expression given for $a$ is reduced and that  $a \in \cA(x)$.
Since 
$a^{-1}$ 
is   precisely the permutation of $[n]$ that maps $i\mapsto n$ and $i+1\mapsto 1$ and restricts to an order-preserving bijection $[n]\setminus \{i,i+1\} \to \{ 2,3,\dots,n-1\}$,
it follows in turn that $b \in \phi(\cX_{n-2})$. As $u=ab$, we conclude that $\cX(n) \subset \cA(x) \phi(\cX_{n-2})$ as desired.
\end{proof}

The following is the key technical observation about the preceding constructions.

\begin{lemma}\label{a-lem3} Let $x,y \in \I(S_n)$. Suppose $w \in S_n$ is such that for all  $\gamma,\gamma' \in \Gamma(y)$ with $\gamma\cap \gamma' =\varnothing$, it holds that $w\gamma,w\gamma' \in \Gamma(x)$ and $\sigma(w\gamma,w\gamma') \preceqB \sigma(\gamma,\gamma')$. 
Then
\[ \DesL(w) \cap \DesR(x) =\varnothing \qquand \DesR(w) \subset \DesR(y).\]
\end{lemma}

\begin{proof}
Observe that if $ j \in \Fix(y)$ then $w(j) \in \Fix(x)$ since  $(w(j),w(j)) \in \Gamma(x)$
as 
$(j,j) \in \Cyc(y)$. 
Hence, if $\gamma=(a,b) \in \Cyc(x)$ such that $a<b$, then  $w^{-1}(a)$ cannot be a fixed point of $y$, so since either $(a,wyw^{-1}(a))$ or $(wyw^{-1}(a),a)$  belongs to $\Gamma(x)$ by hypothesis, it follows that $b=wyw^{-1}(a)$ and $w^{-1}\gamma \in \Cyc(y)$ and  therefore $w^{-1}(a) < w^{-1}(b)$.

Fix $i \in [n-1]$  and let $s = s_i = (i,i+1)$.
First suppose $s \in \DesR(x)$, so that $x(i) > x(i+1)$; we will show that $w^{-1}(i) < w^{-1}(i+1)$.
This holds if $(i,i+1) \in \Cyc(x)$ by the observation in the preceding paragraph. 
Let $j = x(i+1)$ and $k = x(i)$ and suppose $\{ j,k \} \neq \{i,i+1\}$. Then $j<k$ since $s \in \DesR(x)$, and there are five distinct cases for the relative order of $i$, $i+1$, $j$, and $k$:
\ben
\item[(1)] Suppose $j<k<i<i+1$. Then $\gamma = (j,i+1)$ and $\gamma' = (k,i)$ 
both  belong to $\Cyc(x)$, so $w^{-1}\gamma$ and $ w^{-1}\gamma'$ must both belong to $ \Cyc(y)$. 
Since we therefore have
\[ \sigma( \gamma,\gamma')   \ =\    {\xy<0cm,-.05cm>\xymatrix@R=.3cm@C=.2cm{ 
*{\bullet}   \ar @/^.9pc/ @{-} [rrr]   & *{\circ}  \ar @/^.45pc/ @{-} [r]   & *{\circ} & *{\bullet}
}\endxy}
\ \preceqB\ \sigma(w^{-1}\gamma,w^{-1}\gamma')
\]
it follows that $\sigma(\gamma,\gamma') = \sigma(w^{-1}\gamma,w^{-1}\gamma')$ 
and that $w^{-1}(i) < w^{-1}(i+1)$ as desired.

\item[(2)] If $i<i+1<j<k$, then our   argument is  similar to  case (1); we  omit the details.

\item[(3)] Suppose $j<i<i+1<k$. Then $\gamma = (j,i+1)$ and $\gamma' = (i,k)$ 
both  belong to $\Cyc(x)$, so  $w^{-1}\gamma$ and $ w^{-1}\gamma'$  both belong to $ \Cyc(y)$ and we have by hypothesis that
\[
\sigma( \gamma,\gamma')   \ =\    {\xy<0cm,-.05cm>\xymatrix@R=.3cm@C=.2cm{ 
*{\bullet}   \ar @/^.8pc/ @{-} [rr]   & *{\circ}  \ar @/^.8pc/ @{-} [rr]   & *{\bullet} & *{\circ}
}\endxy}
\ \preceqB \sigma(w^{-1}\gamma,w^{-1}\gamma').
\]
It follows from Observation \ref{precB-obs} that
\[
\sigma(w^{-1}\gamma,w^{-1}\gamma')
\in \Bigl\{\ 
 {\xy<0cm,-.05cm>\xymatrix@R=.3cm@C=.2cm{ 
*{\bullet}   \ar @/^.8pc/ @{-} [rr]   & *{\circ}  \ar @/^.8pc/ @{-} [rr]   & *{\bullet} & *{\circ}
}\endxy}
\ , \ 
 {\xy<0cm,-.05cm>\xymatrix@R=.3cm@C=.2cm{ 
*{\bullet}   \ar @/^.9pc/ @{-} [rrr]   & *{\circ}  \ar @/^.45pc/ @{-} [r]   & *{\circ} & *{\bullet}
}\endxy}
\
\Bigr\}
\]
from which we deduce that $w^{-1}(i) < w^{-1}(i+1)$ as desired.

\item[(4)] Suppose $j<k=i<i+1$. 
Define 
$l = wyw^{-1}(i)$ and
note that  $l \in \Fix(x)$ since $i \in \Fix(x)$ and $\{ w^{-1}(i), w^{-1}(l)\}$ is a cycle of $y$.
Let $\gamma = (j,i+1)$
and define $\gamma'$ to be  $(i,l)$ if $l\leq i$ or $(l,i)$ otherwise.
One checks that both $\{\gamma,\gamma' \}\subset \Gamma(x)$ and  $\{w^{-1}\gamma,w^{-1}\gamma'\} \subset  \Gamma(y)$,
so by 
hypothesis
$\sigma( \gamma,\gamma')  \preceqB \sigma(w^{-1}\gamma,w^{-1}\gamma')$.
In turn, we observe that 
\[ \sigma( \gamma,\gamma')   \ =\    {\xy<0cm,-.05cm>\xymatrix@R=.3cm@C=.2cm{ 
*{\bullet}   \ar @/^.9pc/ @{-} [rrr]   & *{\circ}    & *{\circ} & *{\bullet}
}\endxy}
\qquord
 \sigma( \gamma,\gamma')   \ =\    {\xy<0cm,-.05cm>\xymatrix@R=.3cm@C=.2cm{ 
*{\circ}     & *{\bullet}  \ar @/^.8pc/ @{-} [rr]   & *{\circ} & *{\bullet}
}\endxy}
\qquord
 \sigma( \gamma,\gamma')   \ =\    {\xy<0cm,-.05cm>\xymatrix@R=.3cm@C=.2cm{ 
 *{\bullet}  \ar @/^.8pc/ @{-} [rr]   & *{\circ} & *{\bullet} & *{\circ}
}\endxy}
\]
according to whether $j<l < i+1$ or $l<j$ or $i+1<l$, respectively.
In each case, by noting the limited choices for $\sigma(w^{-1}\gamma,w^{-1}\gamma')\succeq \sigma(\gamma,\gamma')$ in view of  Observation \ref{precB-obs}, it is straightforward to deduce that
$w^{-1}(i) < w^{-1}(i+1)$ as desired.



\item[(5)] If $i<i+1=j<k$, then our   argument is  similar to  case (4); we  omit the details.

\een
We conclude from this analysis that $s \notin \DesL(w)$ when $s \in \DesR(x)$.

Next suppose $s \notin \DesR(y)$, so that $y(i) < y(i+1)$;
we will  now show that $w(i) < w(i+1)$.
This holds if $\{i,i+1\} \subset \Fix(y)$ 
since
then, by the observation at the beginning of this proof, we have
$\{ w(i), w(i+1)\} \subset \Fix(x)$   while  $(w(i+1), w(i)) \in \Gamma(x)$ by hypothesis.
Let $j = y(i)$ and $k=y(i+1)$ so that $j<k$ and assume that   $i\neq j$ or $i+1\neq k$. There are    three distinct cases: 
 \ben
 \item[(1)] Suppose $j<k<i<i+1$. Then $\gamma = (j,i)$ and $\gamma' = (k,i+1)$ both belong to $\Gamma(y)$ so
 \[\sigma(w\gamma,w\gamma') \preceqB \sigma( \gamma,\gamma')   \ =\    {\xy<0cm,-.05cm>\xymatrix@R=.3cm@C=.2cm{ 
*{\bullet}   \ar @/^.8pc/ @{-} [rr]   & *{\circ}  \ar @/^.8pc/ @{-} [rr]   & *{\bullet} & *{\circ}
}\endxy}.\]
Using Observation \ref{precB-obs}, we deduce that $w(i) < w(i+1)$ as desired.

\item[(2)] If  $j< i<i+1= k$ or $j=i<i+1< k$ then $\gamma = (j,i)$ and $\gamma' = (i+1,k)$   belong to $\Gamma(y)$ and 
 \[\sigma(w\gamma,w\gamma') \preceqB \sigma( \gamma,\gamma')   
 \in
 \Bigl\{
   {\xy<0cm,-.05cm>\xymatrix@R=.3cm@C=.2cm{ 
*{\bullet}    & *{\bullet}     & *{\circ} \ar @/^.5pc/ @{-} [r] & *{\circ}
}\endxy}
\ , \
  {\xy<0cm,-.05cm>\xymatrix@R=.3cm@C=.2cm{ 
*{\bullet} \ar @/^.5pc/ @{-} [r]    & *{\bullet}     & *{\circ} & *{\circ}
}\endxy}
\Bigr\},\]
from which we deduce again via Observation \ref{precB-obs} that $w(i) < w(i+1)$.

\item[(3)] If $i<i+1<j<k$, then  our   argument is  similar to  case (1); we  omit the details.
\een
We conclude that $s \notin \DesR(w)$ when $s\notin \DesR(y)$.
\end{proof}

The following is the main result of this section.
Showing that the conditions given in this theorem are necessary for $w \in\cA(x,y)$ will   be relatively straightforward; what is notable about this statement is   that these conditions turn out to be sufficient.

\begin{theorem}\label{atomSn-thm}
Let $x,y \in \I(S_n)$ and $w \in S_n$. Then $w \in \cA(x,y)$ if and only if:
\ben
\item[(a)]  $w\gamma \in \Gamma(x)$  for all $\gamma \in\Cyc(y)$.
\item[(b)] $\sigma(w\gamma,w\gamma') \preceqB \sigma(\gamma,\gamma')$
for all $\gamma,\gamma' \in \Cyc(y)$ with $\gamma\cap \gamma' =\varnothing$.
\een
\end{theorem}


Before giving its proof,
it will be useful to state an equivalent form of this result in which our two conditions are unpacked into an explicit set of numeric inequalities.

\begin{theorem}[Restatement of Theorem \ref{atomSn-thm}]
\label{atomSn'-thm}
Let $x,y \in \I(S_n)$ and $w \in S_n$. Then $w \in \cA(x,y)$ if and only if whenever  
$ (a,b),(a',b') \in \Cyc(y)$, the following conditions hold:
\ben

\item If  $w(a)< w(b)$ then $(w(a),w(b))\in \Cyc(x)$ and otherwise $\{w(a),w(b)\} \subset \Fix(x)$.

\item If 
$a\leq  b < a'\leq b'$ then 
$
w(a) < w(a')
$
and
$
w(a) < w(b')
$
and
$
w(b)<w(b')
$
and
$
w(b) < w(a').
$

\item If 
$a < a' < b < b'$ then
$w(a) < w(a')$ and $w(a) < w(b')$ and $w(b) < w(b')$.

\item If 
$a<a' < b'<b$
 then 
neither $ w(b) < w(a') < w(a) $ nor $ w(b) < w(b') < w(a)$ occurs
and neither
$ w(a') < w(a) <w(b) < w(b') $ nor $ w(a')<w(b) \leq w(a) <w(b')
$
occurs.

\item If 
$a<a' = b'<b$
 then it does not occur that 
 $w(b) < w(a') = w(b') < w(a)$.

\een
\end{theorem}

\begin{proof}[Proof of equivalence of  Theorems \ref{atomSn-thm} and \ref{atomSn'-thm}]
Fix $x,y \in \I(S_n)$ and $w \in S_n$.
Condition 1 in Theorem \ref{atomSn'-thm} simply restates condition (a) in Theorem \ref{atomSn-thm}. Assume this condition holds, and let $\gamma = (a,b)$ and $\gamma' = (a',b')$ be distinct elements of $\Cyc(y)$ such that $a<a'$.
It  suffices to show that conditions 2-5 in Theorem \ref{atomSn'-thm} are equivalent to the assertion that  $\sigma(w\gamma,w\gamma') \preceqB \sigma(\gamma,\gamma')$,
since
the latter holds
 if and only if $\sigma(w\gamma',w\gamma) \preceqB\sigma(\gamma',\gamma)$
 by
 Observations \ref{a-obs} and \ref{precB-obs}.
Proving this is  straightforward  from Observation \ref{precB-obs}.
For example, suppose $a<a'< b' <b$. Then
\[ 
\sigma(\gamma,\gamma')  \ = \  \xy<0cm,-.05cm>\xymatrix@R=.3cm@C=.2cm{ 
*{\bullet}   \ar @/^1.0pc/ @{-} [rrr]   & *{\circ}  \ar @/^.5pc/ @{-} [r]  & *{\circ} & *{\bullet}
}\endxy 
\]
so we have $\sigma(w\gamma,w\gamma') \preceqB \sigma(\gamma,\gamma')$  if and only if $\sigma(w\gamma,w\gamma')$ appears in Figure \ref{arr-fig}.
Only 8 colored involutions of the form $\sigma(\lambda,\mu)$ do not appear in this figure; namely:
\[
\xy<0cm,-.05cm>\xymatrix@R=.3cm@C=.2cm{ 
*{\circ}   \ar @/^1.0pc/ @{-} [rrr]   & *{\bullet}  \ar @/^.5pc/ @{-} [r]  & *{\bullet} & *{\circ}
}\endxy
,\quad
\xy<0cm,-.05cm>\xymatrix@R=.3cm@C=.2cm{ 
*{\circ}   \ar @/^1.0pc/ @{-} [rrr]   & *{\bullet}   & *{\bullet} & *{\circ}
}\endxy
,\quad
\xy<0cm,-.05cm>\xymatrix@R=.3cm@C=.2cm{ 
*{\bullet}    & *{\circ}  \ar @/^.75pc/ @{-} [rr]  & *{\bullet} & *{\circ}
}\endxy
,\quad
\xy<0cm,-.05cm>\xymatrix@R=.3cm@C=.2cm{ 
*{\circ}  \ar @/^.75pc/ @{-} [rr]  & *{\bullet} & *{\circ} & *{\bullet}
}\endxy
,\quad
\xy<0cm,-.05cm>\xymatrix@R=.3cm@C=.2cm{ 
*{\bullet}    & *{\circ} \ar @/^0.5pc/ @{-} [r] & *{\circ} & *{\bullet}
}\endxy
,\quad
\xy<0cm,-.05cm>\xymatrix@R=.3cm@C=.2cm{ 
*{\bullet}    & *{\circ}  & *{\bullet} & *{\circ}
}\endxy
,\quad
\xy<0cm,-.05cm>\xymatrix@R=.3cm@C=.2cm{ 
*{\circ}    & *{\bullet}  & *{\circ} & *{\bullet}
}\endxy
,\quand
\xy<0cm,-.05cm>\xymatrix@R=.3cm@C=.2cm{ 
*{\bullet}    & *{\circ}  & *{\circ} & *{\bullet}
}\endxy
.
\]
The element $\sigma(w\gamma,w\gamma')$ 
is equal to the first of these involutions if and only if $w(a') < w(a) < w(b) < w(b')$,
to the second if and only if $w(a') < w(b) \leq w(a) < w(b')$,
and to one
of 
the
remaining six  precisely when $w(b) < w(a')<w(a)$ or $w(b) < w(b') < w(a)$.
Thus, when $a<a'<b'<b$
we have
$\sigma(w\gamma,w\gamma') \preceqB \sigma(\gamma,\gamma')$  if and only if 
  condition 4 in Theorem \ref{atomSn'-thm} holds.
%
It follows by similar arguments that 
 the relation  $\sigma(w\gamma,w\gamma') \preceqB \sigma(\gamma,\gamma')$
is equivalent to conditions 2, 3, and 5 in Theorem \ref{atomSn'-thm} when $a\leq b < a' \leq b'$
and  $a<a'<b<b'$ and $a<a'=b'<b$, respectively.
\end{proof}

We  now prove Theorem \ref{atomSn-thm}.

\def\cAA{\cE}
\begin{proof}[Proof of Theorem \ref{atomSn-thm}]
If $w \in \cA(x,y)$,
then $y > y \act s_{i_1} > y \act s_{i_1} \act s_{i_2} > \dots > y \act s_{s_1} \act s_{i_2} \cdots \act s_{i_k} = x$
for any reduced expression $(s_{i_1},s_{i_2},\dots,s_{i_k}) \in \cR(w^{-1})$,
so 
it follows    by successively applying Lemmas \ref{a-lem1} and \ref{a-lem2}
that $w$ satisfies conditions (a) and (b).
Conversely, suppose $w \in S_n$ satisfies both of these conditions. We  argue that
 $w$ then satisfies the \emph{a priori} stronger conditions: 
 \ben
 \item[(a$'$)] $w\gamma \in \Gamma(x)$ for all $\gamma \in \Gamma(y)$.
 \item[(b$'$)]
$\sigma(w\gamma,w\gamma') \preceqB \sigma(\gamma,\gamma')$ for all $\gamma,\gamma' \in \Gamma(y)$ with $\gamma\cap \gamma' = \varnothing$.
\een
To show this, let $\gamma = (a,b) \in \Gamma(y)\setminus \Cyc(y)$ so that $a,b \in \Fix(y)$ and $b<a$.
Since $(a,a),(b,b) \in \Cyc(y)$,
 we are  able to deduce that $w(a,b) \in \Gamma(x)$ by noting by hypothesis that  
\be\label{rel1}\{ w(a,a),w(b,b)\} \subset \Gamma(x)
\qquand \sigma(w(a,a),w(b,b))\preceqB \sigma((a,a),(b,b)).\ee
Next, let $\gamma'=(a',b') \in \Gamma(y)$ such that $\{a,b\}\cap\{a',b'\}=\varnothing$. If $\gamma' \in \Cyc(y)$, then it is straightforward to show that $\sigma(w\gamma,w\gamma') \preceqB \sigma(\gamma,\gamma')$ from 
\eqref{rel1}
together with the additional relations
$ \sigma(w(a,a),w\gamma')\preceqB \sigma((a,a),\gamma')
$
and
$
\sigma(w(b,b),w\gamma')\preceqB \sigma((b,b),\gamma')
$
which hold by hypothesis.
If instead $\gamma' \in \Gamma(y)\setminus \Cyc(y)$,
then 
the same deduction follows by a similar argument.
We conclude that 
(a$'$) and (b$'$)
hold for all $w \in S_n$ which satisfy (a) and (b), as desired.

Let $\cAA(x,y)$ be the set of permutations $w \in S_n$ satisfying conditions 
 (a$'$) and (b$'$), and note that each element of $\cAA(x,y)$ automatically satisfies the conditions in Theorem \ref{atomSn'-thm}.
By the preceding paragraph, we know that $\cA(x,y) \subset \cAA(x,y)$ and it suffices to show  the reverse inclusion.
To see that this holds when $x=1$ and $y=w_0$, note that  $\Gamma(1)$ contains no pairs $(i,j)$ with $i<j$ and that every pair in $\Cyc(w_0)$ has the form $(i,n+1-i)$ with $i<n+1-i$.
Therefore, if $w \in \cAA(1,w_0)$, then
condition 1 in Theorem \ref{atomSn'-thm} implies that   
 $w(n+1-i) < w(i)$ for all $i<n+1-i$,
while  conditions 4 and 5 in Theorem \ref{atomSn'-thm} imply that no $i,j \in [n]$   satisfy both $i<j<n+1-i$ and  $w(n+1-i) < w(j) <w(i)$.
In view of
Lemma \ref{w0-lem},
we conclude that
$\cAA(1,w_0) \subset \cA(1,w_0)$.
 
 To treat the general case,
assume $\ellhat(x,y) < \ellhat(w_0)$ and 
 suppose $\cAA(x',y') \subset \cA(x',y')$ whenever $x',y' \in \I(S_n)$ are such that $\ellhat(x,y) < \ellhat(x',y')$.
There   necessarily then exists $s\in \{s_1,s_2,\dots,s_{n-1}\}$ such that  either $s \in \DesR(x)$ or $s \notin \DesR(y)$.
First suppose $s \in \DesR(x)$; then 
Lemmas \ref{a-lem1}, \ref{a-lem2}, and \ref{a-lem3} together imply that $w<sw \in \cAA(x\act s,y)$  for all $w \in \cAA(x,y)$,
so   
$\cAA(x,y) \subset \{ su : u \in \cAA(x\act s, y) \text{ and }s \in \DesL(u)\}.$
Since $\cAA(x\act s,y) \subset \cA(x\act s,y)$ by hypothesis,
it follows by Proposition \ref{atomdes-prop}(b) that $\cAA(x,y) \subset \cA(x,y)$  as desired.
On the other hand, if $s \notin \DesR(y)$,
then a similar combination of Lemmas \ref{a-lem1}, \ref{a-lem2}, and \ref{a-lem3},
Proposition \ref{atomdes-prop}(a), and our inductive hypothesis shows likewise that
$
\cAA(x,y) \subset \{ us : u \in \cAA(x, y\act s) \text{ and }s \notin \DesR(u)\}
\subset \cA(x,y).
$
We conclude by induction that $\cAA(x,y) \subset \cA(x,y)$ always holds, which completes our proof.
\end{proof}

While we have so far only used the notation $\sigma(\gamma_1,\dots,\gamma_k)$ when $k=2$, we conjecture that the general version of this construction can be used to give another  characterization of the sets $\cA(x,y)$ for $x,y \in \I(S_n)$.

\begin{conjecture}
Let $x,y \in \I(S_n)$ and $w \in S_n$. 
Then 
 $w \in \cA(x,y)$ if and only if
\ben
\item[(a)]  $w\gamma \in \Gamma(x)$  for all $\gamma \in\Cyc(y)$.
\item[(b)] $\sigma(w\gamma_1,\dots,w\gamma_k) \preceqB \sigma(\gamma_1,\dots,\gamma_k)$
where
$\gamma_1,\dots,\gamma_k$ are the distinct  elements of $\Cyc(y)$.
\een
\end{conjecture} 

\begin{remark}
If $\gamma_i= (a_i,b_i)$ and $z = \std[b_1,a_1,\dots,b_k,a_k]$,
then  (b) is equivalent to $\tau(\wt w z) \preceq \tau(z)$, where $\wt w$ is the element of $S_{2k}$ formed by
 doubling $w(i)$ in the one-line representation of $w$ for each $i \in \Fix(y)$ and then standardizing.
For example, if $w = [3,2,1,6,5,4]$ and $y=(1,3)(4,6)$ then $\wt w = \std[3,2,2,1,6,5,5,4] = [4,2,3,1,8,6,7,5]$. Notably, if $y$ has no fixed points then $\wt w = w$.
\end{remark}

As mentioned at the outset of this section, Theorem \ref{atomSn-thm} gives a common generalization of \cite[Theorem 3.7]{CJ} and \cite[Theorem 2.5 and Corollary 2.16]{CJW}.
We will use these results  in the next section, so  briefly indicate here how to recover  them as special cases of our theorem.
The following statements differ slightly from their predecessors in \cite{CJW} since 
the sets which Can, Joyce, and Wyser denote as $\cW(\pi)$ in \cite{CJW} are composed of the inverses of what we call the atoms of $\pi \in \I(S_n)$.

\begin{corollary}[Can, Joyce, and Wyser \cite{CJW}] \label{atomSn-cor} Let $y \in \I(S_n)$ and $w \in S_n$. Then $w \in \cA(y)$ if and only if  
the following properties hold:
\ben
\item[(a)] If $(a,b) \in \Cyc(y)$ then $w(b) \leq w(a)$.

\item[(b)] If $(a,b) \in \Cyc(y)$ then no integer $t$ with $a<t<b$ is such that $w(b) < w(t) < w(a)$. 

\item[(c)] If $(a,b),(a',b') \in \Cyc(y)$ are such that $a<a'$ and $b<b'$ then $w(b) \leq w(a) < w(b') \leq w(a')$.
\een
\end{corollary}

\begin{proof}
Consider Theorem \ref{atomSn'-thm} with  $x=1$.
Since $\Cyc(1) = \{ (i,i) : i \in [n]\}$, the first condition in that result holds if and only if $w(b) \leq w(a)$ for all $(a,b) \in \Cyc(y)$.
Assume this   holds and let $(a,b),(a',b') \in \Cyc(y)$ with $a<a'$; then,  conditions 2 and 3 in Theorem \ref{atomSn'-thm} hold if and only $w(b) \leq w(a) < w(b') \leq w(a')$ when  $b<b'$,
while conditions 4 and 5
hold if and only if neither $w(b) < w(a') < w(a)$ nor $w(b) < w(b') < w(a)$ occurs when  $ b' \leq b$. It is straightforward to check that these properties are together equivalent to    (b) and (c).
\end{proof}

Recall that $w_0 $ denotes the longest element of $S_n$.

\begin{corollary}[Can and Joyce \cite{CJ}] \label{awn-ex}
Let $u   \in S_n$. Then $u \in \cA(w_0)$
if and only if for all $i,j \in [n]$ it holds that 
(a) if $i<n+1-i$  then $u(n+1-i) < u(i)$
and (b) if $i<j<n+1-i$ then either $u(j)<u(n+1-i)$ or $u(i) < u(j)$.
Consequently,  $|\cA(w_0)| = (n-1)!!$. 
\end{corollary}

\begin{proof}
By Lemma \ref{w0-lem}, it suffices to show that if $u$ satisfies the conditions in Corollary \ref{atomSn-cor} for $y=w_0$ then $u$ satisfies (a) and (b). This is straightforward; we omit the details.
\end{proof}

Let   $\Ifpf(S_{n})$ denote the set of involutions $y \in \I(S_{n})$ which are \emph{fixed-point-free}, that is, with $\Fix(y) =\{ i \in [n]: y(i) = i\}= \varnothing$.
Evidently the set $\Ifpf(S_n)$ is empty unless $n$ is even.
Define 
$\wfpf{n} = s_1s_3s_5\cdots s_{2n-1} \in S_{2n}
$
and
$ \cAfpf(y) = \cA(\wfpf{n},y)
$
for $y \in \Ifpf(S_{2n})$,
and in the following statement, redefine $\Cyc(y) = \{ (a,b) \in [2n] \times [2n] : a< b=y(a)\}$.

\begin{corollary}[Can, Joyce, and Wyser \cite{CJW}] \label{atomSn-fpf-cor} Let $y \in \Ifpf(S_{2n})$ and $w \in S_{2n}$. Then $w \in \cAfpf(y)$ if and only if  
 the following properties hold:
\ben
\item[(a)] If $(a,b) \in \Cyc(y)$ then $w(a) = 2i-1 $ and $w(b)=2i$ for some $i \in [n]$.

\item[(b)] If  $(a,b),(a',b') \in \Cyc(y)$ such that $a<a'$ and $b<b'$ then $w(a)<w(b)<w(a') <w(b')$.
\een
\end{corollary}


\begin{proof}
Consider Theorem \ref{atomSn'-thm} with  $x=\wfpf{n}$.
Since $\Cyc(\wfpf{n}) = \{ (2i-1,2i) : i \in [n]\}$ and $\Fix(\wfpf{n}) = \varnothing$, the first condition in Theorem \ref{atomSn'-thm} is   equivalent to condition (a) in the present corollary. When this holds, moreover, conditions 4 and 5 in Theorem \ref{atomSn'-thm} hold vacuously, while conditions 2 and 3 are equivalent to condition (b) here.
\end{proof}

\section{Orders and equivalence relations in type $A$}
\label{poset-sect}

In this section we write finite sequences of integers as lists bounded by square brackets, e.g.,  $[x_1,x_2,\dots,x_n]$, with the aim of distinguishing sequences  from cycles of permutations. 

\subsection{Atoms and  the Chinese monoid}

A \emph{consecutive subsequence} of a sequence   $ [x_1,x_2,\dots,x_n]$ is one of the form $[x_{i+1},x_{i+2},\dots,x_{i+t}]$ for some $i \in [n]$ and $t\geq 0$.
A \emph{consecutive pair, triple}, \emph{quadruple}, etc., is a consecutive subsequence of length 2, 3, 4, and so on.

\begin{definition}\label{sim-def}
Define $\sim_\DemA$ as the transitive and symmetric closure of the  relation on $n$-element integer sequences
given by setting $x \sim_\DemA y$ whenever $x$ contains a consecutive triple of the form $[c,a,b]$ with $a\leq b\leq c$ and $y$ is formed by replacing this triple by either $[b,c,a]$ or $[c,b,a]$. 
\end{definition}

\begin{remark}\label{sim-rmk}
More concisely,  $\sim_\DemA$ is the weakest equivalence relation on integer sequences with
\[
[\ \cdots\ c,a,b\ \cdots\ ] \sim_\DemA [\ \cdots\ b,c,a\ \cdots\ ] \sim_\DemA [\ \cdots\ c,b,a\ \cdots\ ]
\qquad\text{when }a\leq b\leq c,
\]
where  in this sort of statement, it is  assumed that the sequences  agree in all    entries within the corresponding ellipses (and so have the same length). 
\end{remark}

\begin{example}
The following elements constitute a single equivalence class under $\sim_\DemA$:
\[ [4,3,2,1] \sim_\DemA [4,3,1,2] \sim_\DemA [4,2,3,1]  \sim_\DemA [3,4,2,1]\sim_\DemA [4,1,3,2] \sim_\DemA [3, 4,1,2] \sim_\DemA [3,2,4,1].\]
\end{example}

As our notation suggests, there is a very natural connection between the equivalence relation $\sim_{\DemA}$ and the sets $\DemA(x)$ for involutions $x\in S_n$.
Somewhat remarkably,  this particular relation has been studied previously in a few places in the literature   \cite{CEHKN,DuchampKrob,CM2,LPRW} with     unrelated motivations.
Namely, $\sim_{\DemA}$ is what the authors of \cite{CEHKN,DuchampKrob} call the \emph{Chinese relation} (on the totally ordered alphabet $\ZZ$). The quotient of the free monoid by this relation, in turn,
 has been termed the \emph{Chinese monoid}.
 This monoid has many formal similarities with the better-known \emph{plactic monoid}; for example, both have the same Hilbert series \cite[Eq.\ (3)]{CEHKN}.
On account of this kind of comparison, the   authors of \cite{CEHKN} speculate  that
the Chinese monoid should be the quantized ``maximal torus'' of some unknown quantum group.
The results in this section may provide a useful step towards better understanding this heuristic,
in view of the 
connections between involution words and representations of reductive groups (see \cite{HMP1,LV1,WY}).

When it causes no ambiguity, we write $w_i$ in place of the usual notation $w(i)$ to the denote image of $i \in [n]$ under a permutation $w \in S_n$.
Throughout this section, we identify permutations $w \in S_n$ with the $n$-element sequences given by their one-line representations $[w_1,w_2,\dots,w_n]$.

\begin{lemma} \label{equiv1-lem}
Fix $u,v,w \in S_n$ and $i \in [n-2]$, and suppose $u_j = v_j = w_j$ for all $j \notin \{i,i+1,i+2\}$.
The following are then equivalent:
\ben
\item[(a)] 
We have $[u_i,u_{i+1},u_{i+2}] = [c,a,b]$
and
$[v_i,v_{i+1},v_{i+2}]  = [b,c,a]$
and
$[w_i,w_{i+1},w_{i+2}]  =[c,b,a]$
for some $a,b,c \in [n]$ with $a<b<c$.

\item[(b)] There exists $z \in S_n$ with $\ell(u) = \ell(v) = \ell(z)+2$ and $\ell(w) = \ell(z) + 3$ 
such that 
$ u =z s_{i+1}s_i $ and $ v = z s_i s_{i+1} $ and $ w = z s_is_{i+1} s_i = z s_{i+1}s_{i} s_{i+1}  .$

\een
\end{lemma}

\begin{proof}
The implication (a) $\Rightarrow$ (b) is a straightforward consequence of \eqref{des-sn}.
Conversely, if (b) holds, then 
$us_is_{i+1} < u s_i < u < us_{i+1}$, which implies  $ u_{i+1} < u_{i+2} < u_i$,
from which (a) is immediate as $w= us_{i+1}$ and $v = w s_i$.
%
\end{proof}

The connection between the Chinese relation  and involution words is now given as follows.
\begin{theorem}\label{equiv1-thm}
The sets $\DemA(x)^{-1} = \left\{ u^{-1} : u \in \DemA(x)\right\} $ for $x \in \I(S_n)$ are the distinct  equivalence classes in $S_n$ under $\sim_\DemA$.
In particular, each element of the Chinese monoid
contained in $S_n$  is
 the set of  inverses of the elements of $\DemA(x)$ for some $x \in \I(S_n)$.
\end{theorem}

\begin{proof}
We claim that each equivalence class in $S_n$ under $\sim_\DemA$ is contained in a set $  \DemA(\theta)^{-1} $ for some $\theta \in \I(S_n)$.
To show this, suppose $u,v,w \in S_n$ and $i \in [n-2]$ satisfy the hypotheses of Lemma \ref{equiv1-lem} and possess the equivalent properties described by that result. 
Let $\theta = \sidem{1}{u^{-1}}$ and $\theta' = \sidem{1}{ v^{-1} }$ and $\theta'' = \sidem{1}{w^{-1}}$. 
By construction  $u\sim_\DemA v \sim_\DemA w$, and to prove our claim it suffices to check that $\theta = \theta' =\theta''$.
This is straightforward, for if $z \in S_n$ is as in Lemma \ref{equiv1-lem}(b), then we have 
$ \theta = \sidem{(\sidem{\sidem{1}{ s_{i+1}}}{s_{i}})}{z^{-1}}
$
and
$
\theta' = \sidem{(\sidem{\sidem{1}{ s_{i}}}{s_{i+1}})}{z^{-1}} 
$
and
$
\theta'' =  \sidem{(\sidem{\sidem{\sidem{1}{s_{i}}}{ s_{i+1}}}{s_{i}})}{z^{-1}}
$,
while  $\sidem{\sidem{1}{ s_{i+1}}}{s_{i}}= \sidem{\sidem{1}{ s_{i}}}{s_{i+1}}=\sidem{\sidem{\sidem{1}{s_{i}}}{ s_{i+1}}}{s_{i}}=(i,i+2)$.
From the claim and Corollary \ref{all-nonempty-cor}, we deduce that   the sets $ \DemA(\theta) ^{-1} $ for $\theta \in \I(S_n)$, which automatically partition $S_n$, each decompose as the union of some positive number of distinct equivalence classes under $\sim_\DemA$.
Since it is known \cite[Theorem 15(a)]{LPRW} that the total number of equivalence classes in $S_n$ under $\sim_\DemA$ is exactly $|\cI(S_n)|$, the theorem follows.
(In fact,  we only need to know  that the number of equivalence classes in $S_n$ under $\sim_\DemA$ is at most $ |\I(S_n)|$, and showing this is the easier half of the proof of \cite[Theorem 13]{LPRW}.)
\end{proof}

If we  omit the relations generating $\sim_{\DemA}$ which are not length-preserving then we obtain the following subrelation, which may be used to characterize the sets of atoms $\cA(x) $ for $x \in \I(S_n)$.

\begin{definition}
Define $\prec_\cA$ as the transitive closure of the relation on integer sequences  with 
$[\ \cdots\ c,a,b\ \cdots\ ] \prec_\cA [\ \cdots\ b,c,a\ \cdots\ ]$ when $a\leq b \leq c$ (interpreting this as in the remark above.) 
\end{definition}

By construction $x \prec_\cA y$ implies $x \sim_\DemA y$. The relation $\prec_\cA$ is a partial order, since it is a transitive subrelation of reverse lexicographic order.
A compactly-drawn interval in this order is shown in Figure \ref{fig1}.

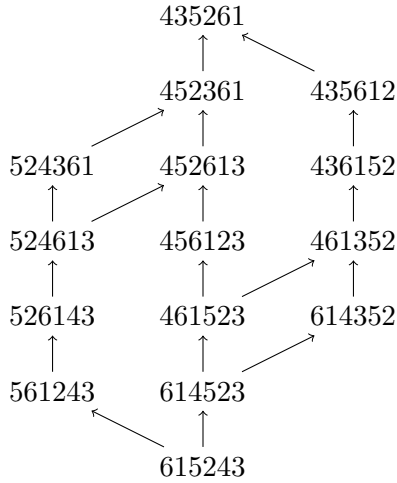
\begin{figure}[h]
\[
\begin{tikzpicture}
  \node (max) at (0,2) {$435261$};
  \node (b) at (0,1) {$452361$};
  \node (c) at (2,1) {$435612$};
  \node (d) at (-2,0) {$524361$};
  \node (e) at (0,0) {$452613$};
  \node (f) at (2,0) {$436152$};
  \node (g) at (-2,-1) {$524613$};
  \node (h) at (0,-1) {$456123$};
  \node (i) at (2,-1) {$461352$};
    \node (j) at (-2,-2) {$526143$};
  \node (k) at (0,-2) {$461523$};
  \node (l) at (2,-2) {$614352$};
  \node (m) at (-2,-3) {$561243$};
  \node (n) at (0,-3) {$614523$};
  \node (min) at (0,-4) {$615243$};
    \draw  [->]
  (min) edge (m)
  (m) edge (j)
  (j) edge (g)
  (g) edge (d)
  (d) edge (b)
  (b) edge (max)
  (c) edge (max)
  (f) edge (c)
  (i) edge (f)
  (l) edge (i)
  (n) edge (l)
  (min) edge (n)
  (g) edge (e)
  (k) edge (i)
  (n) edge (k)
  (k) edge (h)
  (h) edge (e)
  (e) edge (b)
;
\end{tikzpicture}
\]
\caption{Hasse diagram of   $\{ u : \theta \preceq_\cA u \preceq_\cA \tau\}$ for $\theta=[6,1,5,2,4,3]$ and $\tau=[4,3,5,2,6,1]$ }
\label{fig1}
\end{figure}

We need some notation to construct the ``extremal'' atoms of an involution.
Fix   $x \in \I(S_n)$. Write $\Cyc_n(x)=\Cyc(x)$ for the cycle set of $x$, as given at the beginning of  Section \ref{atomSn-sect}.
Assume $\Cyc_n(x)$ has size $k$, and  define $(a_1,b_1),\dots,(a_k,b_k)$ and $(c_1,d_1),\dots,(c_k,d_k)$  such that  
\be\label{param-eq} 
\begin{cases} \{(a_i,b_i) :  i \in [k]\} = \Cyc_n(x) \\ a_1<a_2<\dots<a_k
\end{cases}
\qquand
\begin{cases}
\{ (c_i,d_i) : i \in [k]\} = \Cyc_n(x)
\\
d_1 < d_2 < \dots< d_k.
\end{cases}
\ee
We note a quick lemma before continuing. Recall that a permutation $w \in S_n$ is \emph{321-avoiding} if it never holds for $i<j<k$ in $[n]$ 
 that $w(i) >w(j) > w(k)$.

\begin{lemma}\label{321-lem}
We have $(a_i,b_i) = (c_i,d_i)$ for all $i \in[k]$ if and only if $x$ is 321-avoiding.
\end{lemma}

\begin{proof}
We have $(a_i,b_i) = (c_i,d_i)$ for all $i \in [k]$ precisely when  
 no $(p,q),(p',q') \in \Cyc_n(x)$ exist with $p<p'\leq q'<q$, which is clearly the case if $x$ is 321-avoiding. 
Conversely, if there exist $i<j<k$ in $[n]$ with $x(i) > x(j) > x(k)$,
then either 
$x(i) \leq i$ or $x(j) \leq j$ or $i < j < x(j) < x(i)$, and in each case one easily constructs
pairs $(p,q),(p',q') \in \Cyc_n(x)$ with  $p<p'\leq q'<q$.
\end{proof}

Given any list of numbers $[e_1,e_2,\dots]$, we write $[[e_1,e_2,\dots]]$ for the sublist  formed by omitting all repeated entries
 after their initial occurrence, so that, for example, $[[1,2,1,1,3,3,2,4,2]] = [1,2,3,4]$.
With $(a_i,b_i)$ and $(c_i,d_i)$ as in \eqref{param-eq}, define $\hat0(x)$ and $\hat1(x)$  by
\[
 \hat0(x) = [[b_1,a_1,b_2,a_2,\dots,b_k,a_k]] \qquand \hat1(x) = [[d_1,c_1,d_2,c_2,\dots,d_k,c_k]].
 \]
We interpret these lists as permutations in $S_n$ written in one-line notation.
\begin{example} If  $x=(1,5)(2,4)
$  then
$
\hat0(x) =  [5,1,4,2,3]
 $
 and
 $
\hat1(x) =[3,4,2,5,1].
$
The minimal and maximal elements in  Figure~\ref{fig1} are $\hat0(x)$ and $\hat1(x)$ for $x=(1,6)(2,5)(3,4)$.
\end{example}

We require two lemmas to prove our main theorem about the partial order $\prec_\cA$.

\begin{lemma}\label{orderrev-lem}
Write $w_0$ for the longest element of $S_n$.
\ben
\item[(a)] If $x,y \in S_n$ then $x \preceq_\cA y$ if and only if $w_0yw_0 \preceq_\cA w_0 x w_0$.
\item[(b)] If $x \in \I(S_n)$  then $\hat 0(w_0x w_0) = w_0 \hat 1(x) w_0$ and  $\hat 1(w_0x w_0) = w_0 \hat 0(x) w_0$.
\een
\end{lemma}

\begin{proof} Checking this lemma is a simple exercise, which we leave to the reader.
\end{proof}

\begin{lemma}\label{notmin-lem}
Let $x \in \I(S_n)$ and $u \in S_n$ with $u^{-1} \in \cA(x)$.
If  $u \neq \hat 0(x)$ (respectively, if $u\neq\hat 1(x)$) then $u$ is not minimal (respectively, maximal) in the order $\prec_\cA$.

\end{lemma}

\begin{proof}
These  assertions are equivalent by Lemma \ref{orderrev-lem}, so we only prove one.
Assume $u \neq \hat 0(x)$. 
Suppose $\Cyc_n(x)$ has $k$ elements and let $\{ (a_i,b_i) \}_{i \in[k]}$ be   as in \eqref{param-eq}.
By hypothesis there exists a smallest integer $l \in [k]$
such that $u$ does not begin with the  sequence $[[ b_1,a_1,\dots,b_l,a_l]]$.
 Observe that $u$ necessarily  does then begin with the (possibly empty) sequence $[[ b_1,a_1,\dots,b_{l-1},a_{l-1}]]$.
Define  $
i = u^{-1}(b_l) $ and $ j = u^{-1}(a_l)
$
so that $ u_i = b_l $ and $ u_j = a_l.$
We  claim that 
$ j>2
$
and $u_{j} < u_{j-2} < u_{j-1}$,
 from which  it is immediate that $u$ is not minimal in the order $\prec_\cA$.
 
To prove this claim, 
 let $m\geq 0$ be the length of the initial sequence $[[ b_1,a_1,\dots,b_{l-1},a_{l-1}]]$ in $u$, and note that  $m<i \leq j$ by Corollary \ref{atomSn-cor}(a) and that $a_l = \min \{u_{m+1},u_{m+2},\dots,u_n\}$.
 Since $u$ does not begin with $[[ b_1,a_1,\dots,b_l,a_l]]$, we cannot have $i=j=m+1$ or $(i,j) = (m+1,m+2)$,
 and in view of
Corollary \ref{atomSn-cor}(c)  we cannot have $i=j=m+2$ either.
Hence  we must have $j \geq m+3$,
  so either $u_{j} < u_{j-2} < u_{j-1}$ or  $u_j < u_{j-1} < u_{j-2}$. The first case is  what we want to show, while the second case cannot occur since then we would have $u \sim_{\DemA} us_{j-2} < u$, contradicting via Theorem \ref{equiv1-thm} our assumption that $u^{-1} \in \cA(x)$.
We conclude that part (a) holds, as desired.
\end{proof}

We  now prove the following statement.
As in Section \ref{duality-sect}, let $\cA(x)^{-1} = \{ u^{-1} : u \in \cA(x)\}$.
\begin{theorem}\label{prec1-thm}
Let $x \in \I(S_n)$. Then  
$\cA(x)^{-1}= \left\{ u \in S_n : \hat0(x) \preceq_\cA u \right\} = \left\{ u \in S_n : u \preceq_\cA \hat1(x)\right\}$.
\end{theorem}

\begin{proof}
By construction   
$\hat0(x)$ and $\hat 1(x)$ are respectively minimal and maximal relative to $\prec_\cA$.
In turn, it is straightforward using Corollary \ref{atomSn-cor} to show that $\hat 0(x)$ and $\hat 1(x)$ both belong to $\cA(x)^{-1}$.
Finally, if $u,v,w \in S_n$ are such that $v^{-1} \in \cA(x)$ and $u \preceq_\cA v\preceq_\cA w$, then both $u^{-1}$ and $w^{-1}$  belong to $ \cA(x)$ as well, since these elements belong to $\DemA(x)$ by 
Theorem \ref{equiv1-thm} and have the same length as $v^{-1}$ by Lemma \ref{equiv1-lem}.
Since $\preceq_\cA$ thus preserves membership in $ \cA(x)^{-1}$ and since $\hat 0(x)$ and $\hat 1(x)$ are the unique minimal and maximal elements of this set by Lemma \ref{notmin-lem}, the theorem follows.
\end{proof}

It follows from  Theorem \ref{prec1-thm} that $\cA(x)^{-1}= \left\{ u \in S_n : \hat0(x) \preceq_\cA u\preceq_\cA\hat 1(x)\right\}$.
This fact lets us   classify which involutions in $S_n$ have exactly one atom.

\begin{corollary}\label{prec1-cor} If $x \in \I(S_n)$ then $|\cA(x)| = 1$ if and only if $x$ is 321-avoiding
\end{corollary}
 
\begin{proof}
The theorem implies that $|\cA(x)| = 1$ if and only if $x \in \I(S_n)$ is such that $\hat 0(x) = \hat 1(x)$.
By Lemma \ref{321-lem}, this occurs if and only if $x$ is 321-avoiding.
\end{proof}

Using Corollary \ref{atomSn-cor}, we may show that the poset  $\(\cA(x)^{-1},\prec_\cA\)$ has some additional structure, as  Figure \ref{fig1} suggests. To describe this,
fix $x \in \I(S_n)$ and define 
\[\L_x = \{ a \in [n] : a \leq x(a)\}
\qquand
\R_x = \{b \in [n] : x(b) < b\}
\] so that $[n] = \L_x\sqcup \R_x$, where we write $\sqcup$ to denote disjoint union.
Also let
\[
\cZ_x = \left\{ (a,a') \in \L_x\times \L_x : a>a'\right\} \sqcup\left \{ (b,b') \in \R_x\times \R_x: b<b'\right\}.
\]
Finally, for $u \in S_n$  define 
$\ainv(u;x) = \left\{ (p,q) \in \cZ_x : u^{-1}(p) < u^{-1}(q)\right \}.$
We refer to $\ainv(u;x)$ as the \emph{$\cA$-inversion set} of $u$, relative to $x$.

\begin{example}
If $x=[6,5,4,3,2,1]$ then for $u = [5,6,1,2,4,3]$ and $v = [4,5,6,1,2,3]$  we have $\ainv(u;x) = \{(5,6)\}$ and $\ainv(v;x) = \{ (4,5),(5,6),(4,6)\}$.
 Note from Figure \ref{fig1} that   $u \not \prec_\cA v$.
\end{example}

Although containment of $\cA$-inversion sets does not imply that atoms are  comparable in $\prec_\cA$, the converse implication does hold. Recall that if $(\cP,<)$ is a poset, then  one says that  $y\in \cP$ \emph{covers} $x \in \cP$ when $\{ u\in \cP: x \leq u < y\}= \{x\}$.

\begin{lemma}\label{grade-lem}
Let $x \in \I(S_n)$ and $u,v \in \cA(x)^{-1}$ and suppose  $u\prec_\cA v$. Then 
$\ainv(u;x) \subset \ainv(v;x)$, and if $v$ covers $u$ then $
|\ainv(v;x)| = |\ainv(u;x)|+1.
$
\end{lemma}

\begin{proof}
Suppose $u,v \in \cA(x)^{-1}$ are such that $v$ covers $u$ in $\prec_{\cA}$. By definition, there exists  $i \in [n-2]$ such that $u_j=v_j$ for $j \notin \{i,i+1,i+2\}$
and such that
$[u_i,u_{i+1},u_{i+2} ]= [c,a,b]$ and $[v_i,v_{i+1},v_{i+2}] = [b,c,a]$ for  some $a<b<c$ in $[n]$.
We claim that $a \in \L_x$ and $c \in \R_x$. This claim suffices to prove the lemma since 
\begin{itemize}
\item If $a,b \in \L_x$ and $c\in \R_w$, then $\ainv(v;x) = \ainv(u;x) \sqcup \{ (b,a)\}$.
\item  If $a \in \L_x$ and $b,c \in \R_w$,  then $\ainv(v;x) = \ainv(u;x)\sqcup \{(b,c)\}$.
\end{itemize}
To proceed, we appeal to Corollary \ref{atomSn-cor} and argue by contradiction.
Define $a' = x(a)$ and $c'=x(c)$.
We cannot have  $a \in \R_x$ and $c \in \L_x$, since then $(a',a),(c,c') \in \Cyc(x)$ and $u^{-1}(c) < u^{-1}(a)$, which contradicts 
Corollary
 \ref{atomSn-cor}(c) as $a<c$. 
Suppose instead that both $a,c \in \L_x$. Then $(a,a'),(c,c') \in \Cyc(x)$ and  it follows by  Corollary \ref{atomSn-cor}(a) that 
 \[ u^{-1}(a') < u^{-1}(c)  < u^{-1}(a)
  \qquand
  u^{-1}(c') < u^{-1}(c)
 < u^{-1}(a).\]
In this case, we cannot have $a'<c'$ as this would again contradict Corollary \ref{atomSn-cor}(c), but if $c'<a'$ then  $a <c\leq c' <a'$ which   contradicts  Corollary \ref{atomSn-cor}(b). We conclude that $c \in \R_x$. Suppose finally that both $a,c \in \R_x$. Then
$(a',a),(c',c) \in \Cyc(x)$ and  by  Corollary \ref{atomSn-cor}(a) we have
 \[ u^{-1}(c) < u^{-1}(a)  < u^{-1}(a')
  \qquand
  u^{-1}(c) < u^{-1}(a)
 < u^{-1}(c').\]
As in the previous case, one argues that  $a'<c'$ contradicts Corollary \ref{atomSn-cor}(c) while $c'<a'$ contradicts  Corollary \ref{atomSn-cor}(b), so it must hold that $a \in \L_x$ and $c \in \R_x$,  which completes the proof.
\end{proof}

A poset $(\cP,<)$ is \emph{graded} if there exists a function $\rank : \cP \to \NN$ such that $\rank(y) = \rank(x) +1$ whenever $y$ covers $x$. A poset if \emph{bounded} if it has a  greatest element and a least element.

\begin{proposition}\label{graded-prop} For each $x \in \I(S_n)$, the poset $\(\cA(x)^{-1}, \prec_\cA\)$ is graded and bounded.
\end{proposition}

\begin{proof}
This follows  from Theorem \ref{prec1-thm} and Lemma \ref{grade-lem}, taking $\rank(u) =| \ainv(u;x)|$.
\end{proof}

We suspect that  $\(\cA(x)^{-1}, \prec_\cA\)$ has  even stronger properties than what is given in this proposition; e.g., examples suggest that the poset is always a lattice.

\begin{remark}
The  \emph{inversion set} of  $u \in S_n$ (in one formulation) is  the set $\inv(u) = \{ (p,q) \in [n]\times[n] : p>q\text{ and }u^{-1}(p)<u^{-1}(q)\}$.
It is well-known that a permutation is uniquely determined by its inversion set, and
one can show (as a slightly involved exercise from Corollary  \ref{atomSn-cor}) that if $x \in \cI(S_n)$ is fixed then  each $u \in \cA(x)$ is  uniquely determined likewise by its $\cA$-inversion set relative to $x$. \end{remark}

 \subsection{Fixed-point-free variants}

Fix $n \in \PP$ and, as usual, let $\wfpf{n} = s_1s_3s_5\cdots s_{2n-1} \in S_{2n}$.
Recall that for each $x \in \Ifpf(S_{2n})$  we define
$ \cAfpf(x) = \cA(\wfpf{n},x)$ and $ \DemAfpf(x) = \DemA(\wfpf{n},x).$
There are  analogous  results about these modified sets of atoms.
To describe these, we will need to 
 define some binary relations on  integer sequences of even length.  We introduce some convenient notation for this. Suppose $\sim$ is such a relation.
If $j\in \NN$  and $X_1,X_2,\dots,X_{2j}$ and $Y_1,Y_2,\dots,Y_{2j}$ are any numbers, then we write 
\[ [ \ \dbldots\ X_1,X_2,\dots,X_{2j}\ \dbldots\ ] \sim [ \ \dbldots\ Y_1,Y_2,\dots,Y_{2j}\ \dbldots\ ] \]
to mean that $[ \alpha_1,\dots,\alpha_{2i}, X_1,X_2,\dots,X_{2j}, \beta_1,\dots,\beta_{2k}] \sim [ \alpha_1,\dots,\alpha_{2i}, Y_1,Y_2,\dots,Y_{2j}, \beta_1,\dots,\beta_{2k}]
$
holds 
for all even-length sequences  $[\alpha_1,\dots,\alpha_{2i}]$ and $[\beta_1,\dots,\beta_{2k}]$. Here, the doubled dots in the symbol $\dbldots$ are meant to indicate that the ellipsis masks an even number of elements.

\begin{definition}\label{equiv2-def}
Define $\sim_{\DemAfpf}$ as the transitive and symmetric closure of the   relation on $2n$-element integer sequences
with $ [\ \dbldots\ a,b\ \dbldots\ ] \sim_{\DemAfpf} [\ \dbldots\ b,a\ \dbldots\ ]$ for all $a,b$ and with
\[
[\  \dbldots\ a,d,b,c\ \dbldots\ ] 
\sim_{\DemAfpf} [\ \dbldots\ b,c,a,d\ \dbldots\ ]
\sim_{\DemAfpf} [\ \dbldots\ b,d,a,c\ \dbldots\ ]
\sim_{\DemAfpf} [\ \dbldots\ c,d,a,b\ \dbldots\ ]
\]
for all $a,b,c,d$ with $a\leq b\leq c\leq d$.
\end{definition}

\begin{example}
One checks that
$
{ [1,5,4,6,2,3] }
\sim_{\DemAfpf} [1,5,3,6,2,4] 
\sim_{\DemAfpf} [1,5,2,6,3,4]  
\sim_{\DemAfpf} [1,5,3,4,2,6]
\sim_{\DemAfpf} [3,4,1,5,2,6]
\sim_{\DemAfpf} [3,5,1,4,2,6]
\sim_{\DemAfpf} [4,5,1,3,2,6] .
$
The full  equivalence class under $\sim_{\DemAfpf}$ of these  elements has size 56; one reaches the remaining elements  by applying  equivalences of the form $ [\ \dbldots\ a,b\ \dbldots\ ] \sim_{\DemAfpf} [\ \dbldots\ b,a\ \dbldots\ ]$ to  the elements listed.  
\end{example}

Given $I = [i_1,i_2,\dots,i_k]\subset [n]$ and $w \in S_n$,   set $w(I) = [w_{i_1}, w_{i_2},\dots,w_{i_k}]$.

\begin{lemma} \label{equiv2-lem}
Fix $t,u,v,w \in S_{2n}$ and $i \in [n-1]$. Let $I = [2i-1,2i,2i+1,2i+2]$ and suppose $t_j=u_j = v_j = w_j$ for all $j \notin I$.
The following are then equivalent:
\ben
\item[(a)] 
It holds that
$t(I) = [a,d,b,c] 
$
and
$u(I)= [b,c,a,d]$
and
$v(I)= [b,d,a,c]$
and
$w(I)= [c,d,a,b]$
for some $a,b,c,d \in [2n]$ with $a<b<c<d$.

\item[(b)] There exists $z \in S_{2n}$ with $\ell(t) = \ell(u) = \ell(z)+2$ and $\ell(v) = \ell(z) + 3$ and $\ell(w) = \ell(z)+4$ 
such that 
$ t =zs_{2i+1} s_{2i}$ and $ u = z s_{2i-1} s_{2i} $ and $ v = z s_{2i-1} s_{2i+1} s_{2i} $ and $w= z s_{2i}s_{2i-1} s_{2i+1} s_{2i}   .$

\een
\end{lemma}

\begin{proof}
The proof is similar to that of Lemma \ref{equiv1-lem}; we leave the details to the reader. 
\end{proof}

We have this analogue of Theorem \ref{equiv1-thm}.

\begin{theorem}\label{equiv2-thm}
The sets $\DemAfpf(x)^{-1} = \left\{ u^{-1} : u \in \DemAfpf(x)\right\} $ for $x \in \Ifpf(S_{2n})$ are the distinct  equivalence classes in $S_{2n}$ under the relation $\sim_{\DemAfpf}$.
\end{theorem}

\begin{remark}
Despite the similarities between Theorems \ref{equiv1-thm} and \ref{equiv2-thm}, we do not know of  any  place in the literature 
where the equivalence relation $\sim_{\DemAfpf}$ has been previously studied. The quotient of the submonoid of even-length sequences in $\Mon{\ZZ}$ by this relation may give an interesting analogue of the Chinese monoid.
We were led to connect $\sim_{\DemA}$ to  \cite{CEHKN,DuchampKrob,CM2,LPRW} by observing that the numbers 
$ \( |\DemA(w_{0})|\)_{n=0,1,2,\dots} = \(1,\ 1,\  1,\ 3,\ 7,\ 35,\ 135,\ 945,\ 5193,\ \dots\)   $
with $w_{0}=[n,n-1,\dots,2,1] \in S_n$ coincide with sequence \cite[A212417]{OEIS}. By contrast, there seems to be no sequence currently in \cite{OEIS} which contains 
$ \( |\DemAfpf(w_{0})|\)_{n=0,2,4,\dots} = \(1,\ 2,\ 16,\ 320,\ 12448,\ 809792,\ \dots\)$
or $ ( 2^{-\frac{n}{2}}|\DemAfpf(w_{0})|)_{n=0,2,4,\dots} = \(1,\ 1,\ 4,\ 40,\ 778,\ 25306,\ \dots\)$
  as a  subsequence.
\end{remark}

\begin{proof}
Our proof is similar to that of Theorem \ref{equiv1-thm}, except that we must substitute our reference to \cite{LPRW} with a direct argument.
Let $\wfpf{n} = s_1s_3s_5\cdots s_{2n-1}$.
Note that  if $\mu,\nu \in S_{2n}$ are such that $\nu=\mu s_{2j-1}$ for some $j \in [n]$, 
then $\sidem{\wfpf{n}}{\mu^{-1} }= \sidem{\wfpf{n} }{ \nu^{-1}}$ (with $\dact$ as in \eqref{dact-def})
since $\sidem{\wfpf{n}}{ s_{2j-1}  }= \wfpf{n}$.
In turn, if $t,u,v,w\in S_{2n}$ and $i \in [n-1]$ satisfy the hypotheses of Lemma \ref{equiv2-lem} and have the equivalent properties described by that result, 
then it is straightforward using Lemma \ref{equiv2-lem}(b) to check that 
$\sidem{\wfpf{n}}{  t^{-1}}= \sidem{\wfpf{n}}{ u^{-1}} = \sidem{\wfpf{n}}{ v^{-1}}=\sidem{ \wfpf{n}}{ w^{-1}}$.
%
Since $\sim_{\DemAfpf}$ is generated by equivalences of  the form   $\mu \sim_{\DemAfpf} \nu$ and  $t \sim_{\DemAfpf} u\sim_{\DemAfpf} v \sim_{\DemAfpf} w $,
we deduce that each equivalence class in $S_{2n}$ under $\sim_{\DemAfpf}$ is contained in a set $ \DemAfpf(\theta)^{-1} $ for some $\theta \in \Ifpf(S_{2n})$

Fix $\theta \in \Ifpf(S_{2n})$. It is a simple exercise using Proposition \ref{des-lem} to show that $\DemAfpf(\theta)$ is always nonempty,
and consequently
the first paragraph shows that  $ \DemAfpf(\theta)^{-1} $ is the disjoint union of some positive number of equivalence classes in $S_{2n}$ under $\sim_{\DemAfpf}$. To prove that this set is exactly one equivalence class, it is enough to show that the total number of equivalence classes in $S_{2n}$ under $\sim_{\DemAfpf}$ is at most $|\Ifpf(S_{2n})|$.
We do this by showing that each $u \in S_{2n}$ satisfies $u \sim_{\DemAfpf} w$ for some permutation $w \in S_{2n}$ with $w_{2i-1} < w_{2i}$ for all $i \in [n]$ and $w_2< w_4<w_6<\dots<w_{2n}$; the set of such permutations $w$ is in bijection with $\Ifpf(S_{2n})$ via the map $ w \mapsto (w_1,w_2)(w_3,w_4)\cdots (w_{2n-1},w_{2n}).$
This follows as a simple exercise by induction: one checks   by successively applying the  relations generating $\sim_{\DemAfpf}$ that an arbitrary permutation $u \in S_{2n}$ satisfies $u\sim_{\DemAfpf} u'$ for some $u' \in S_{2n}$ with $u'_{2n}=2n$; by induction the truncation $u''=[u_1',u_2',\dots,u'_{2n-2}]$ satisfies $u'' \sim_{\DemAfpf} w'$ for some sequence $w'=[w'_1,w'_2,\dots,w'_{2n-2}]$ obtained by permuting the entries of $u''$ 
and satisfying both $w'_{2i-1} < w'_{2i}$ for all $i $ and $w'_2< w'_4< \dots<w'_{2n-2}$;
and the concatenation $w=[w'_1,w'_2,\dots,w'_{2n-2},u'_{2n-1},u'_{2n}]$ is then   a permutation of the desired form with $u \sim_{\DemAfpf}w$.
\end{proof}

We may also define a fixed-point-free version of the partial order $\prec_{\cA}$.

\begin{definition}
Define $\prec_{\cAfpf}$ as the transitive closure of the relation on $2n$-element integer sequences  with 
$[\  \dbldots\ a,d,b,c\ \dbldots\ ] \prec_{\cAfpf} [\ \dbldots\ b,c,a,d\ \dbldots\ ]$ when $a\leq b \leq c\leq d$.
\end{definition}

As with $\prec_{\cA}$, it follows that the relation $\prec_{\cAfpf}$ is a partial order since it is a transitive subrelation of lexicographic order.
Figure \ref{fig2} shows an interval in this order; as a visual aid, we write $ab|cd|\cdots$ in place of $[a,b,c,d,\dots]$   in this picture.


\begin{figure}[h]
\[
\begin{tikzpicture}
  \node (max) at (0,2.4) {$15|36|07|48|29$};
  \node (a) at (-2,1.2) {$15|36|07|29|48$};
    \node (b) at (2,1.2) {$15|07|36|48|29$};
  \node (d) at (-2,0) {$15|07|36|29|48$};
  \node (e) at (2,0) {$07|15|36|48|29$};
  \node (g) at (-2,-1.2) {$15|07|29|36|48$};
  \node (h) at (2,-1.2) {$07|15|36|29|48$};
  \node (min) at (0,-2.4) {$07| 15 | 29| 36| 48$};
  \draw 
  [->]
  (min) edge (h)
  (h) edge (e)
  (e) edge (b)
  (b) edge (max)
  (a) edge (max)
  (d) edge (a)
  (g) edge(d)
  (min) edge (g)
  (h) edge (d)
  (d) edge (b)
;
\end{tikzpicture}
\]
\caption{Hasse diagram of an interval in $\prec_{\cAfpf}$ 
}\label{fig2}
\end{figure}

To describe the properties of the order $\prec_{\cAfpf}$ we require some notation similar to that in the previous section.
For $x \in \Ifpf(S_{2n})$,
let $\{(a_i,b_i)\}_{i \in [n]}$ and $\{(c_i,d_i)\}_{i \in [n]}$ be as in \eqref{param-eq}, but replacing $(k,n)$ by $(n,2n)$, and
define $\hat0_\fpf(x)$ and $\hat1_\fpf(x)$  as the permutations given in one-line notation by
\[
 \hat0_\fpf(x) = [a_1,b_1,a_2,b_2,\dots,a_n,b_n] \qquand \hat1_\fpf(x) = [c_1,d_1,c_2,d_2,\dots,c_n,d_n].
 \]
Observe that $\hat0_\fpf(x) = \hat0(x) \cdot \wfpf{n}$ and $\hat1_\fpf(x) = \hat1(x) \cdot \wfpf{n}$ where $\wfpf{n} = s_1 s_3s_5 \cdots s_{2n-1}$.

\begin{example}
For $x 
=(1,8)(2,3)(4,6)(5,7) \in \Ifpf(S_8)$  we have 
$ \hat0_\fpf(x) = [1,8,2,3,4,6,5,7]$ and $ \hat 1_\fpf(x) = [2,3,4,6,5,7,1,8].
$
\end{example}

The following result is similar to Lemma \ref{notmin-lem}.
There is also an obvious  ``fixed-point-free'' version of Lemma \ref{orderrev-lem}, but we omit its statement since we will not require it in what follows.

%

\begin{lemma}\label{notmin-lem2}
Let $x \in \Ifpf(S_{2n})$ and  $u \in S_{2n}$ with $u^{-1} \in \cAfpf(x)$.
 If $u \neq \hat 0_\fpf(x)$ (respectively, 
 if $u \neq \hat 1_\fpf(x)$)
 then $u$ is not minimal (respectively, maximal) in the order $\prec_{\cAfpf}$.
\end{lemma}

\begin{proof}
Assume $u \neq \hat 0_\fpf(x)$ and write $\wfpf{n} = s_1s_3s_5\cdots s_{2n-1}$.
Then  $u\wfpf{n} \neq \hat 0(x)$, so since $(u \wfpf{n})^{-1} = \wfpf{n}   u^{-1} \in \cA(x)$ by definition,
it follows by Lemma \ref{notmin-lem}(a) that $u\wfpf{n}$ has a consecutive subsequence of the form $[b,c,a]$ for some $a<b<c$ in $[2n]$. Since $s_{i} \in \DesR(u\wfpf{n})$ whenever $i$ is odd by Lemma \ref{atomdes-lem} and the exchange principle, this consecutive subsequence must begin at an even index
and so we deduce that $u$ has the form $[\ \dbldots\ b, i, a,c\ \dbldots\ ]$ for some $i \in [2n]$.
If $i<b$ or $i>c$  then one easily constructs  $v \in S_{2n}$ with $u \sim_{\DemAfpf} v < u$, contradicting  via Theorem \ref{equiv2-thm} our assumption that $u^{-1} \in \cAfpf(x)$.
Therefore we must have $a<b<i<c$, so $u$ is  not minimal with respect to $\prec_{\cAfpf}$. This proves the first claim, and the second one follows by a similar argument, now appealing to Lemma \ref{notmin-lem}(b)
\end{proof}

We now prove the main result about the partial order $\prec_{\cAfpf}$ and its relation to  $\cAfpf(x)$.

\begin{theorem}\label{prec2-thm}
Let  $x \in \Ifpf(S_{2n})$. Then 
\[\cAfpf(x)^{-1}= \left\{ u \in S_{2n} : \hat0_\fpf(x) \preceq_{\cAfpf} u\right\}
=
\left\{ u \in S_{2n} : u \preceq_{\cAfpf} \hat1_\fpf(x)\right\}.
\]
\end{theorem}

\begin{proof}
The proof is the same as that of Theorem \ref{prec1-thm}; one just replaces
  $\hat0(x)$ / $\hat 1(x)$ / $\cA(x)$ / $\prec_\cA$ by $\hat0_\fpf(x)$ / $\hat 1_\fpf(x)$ / $\cAfpf(x)$ / $\prec_{\cAfpf}$
and refers to  Corollary \ref{atomSn-fpf-cor}, Theorem \ref{equiv2-thm}, and Lemmas \ref{equiv2-lem} and \ref{notmin-lem2} in place of Corollary \ref{atomSn-cor}, Theorem \ref{equiv1-thm}, and Lemmas \ref{equiv1-lem} and \ref{notmin-lem}.
\end{proof}

The theorem implies  $\cAfpf(x)^{-1}= \left\{ u \in S_{2n} : \hat0_\fpf(x) \preceq_{\cAfpf} u\preceq_{\cAfpf} \hat 1_\fpf(x)\right\}$ for $x \in \Ifpf(S_{2n})$, from which we derive the following analogues of Corollary \ref{prec1-cor} and Proposition \ref{graded-prop}.

\begin{corollary}\label{prec2-cor} If $x \in \Ifpf(S_{2n})$ then $|\cAfpf(x)| = 1$ if and only if $x$ is 321-avoiding.
\end{corollary}

\begin{proof}
As with Corollary \ref{prec1-cor}, the result is clear from Lemma \ref{321-lem} given Theorem \ref{prec2-thm}.
\end{proof}

\begin{proposition}\label{graded-prop2} Let $x \in \Ifpf(S_{2n})$. The poset $\(\cAfpf(x)^{-1},\prec_{\cAfpf}\)$  is isomorphic to a lower interval in the right weak order on $S_{n}$, and consequently is a graded lattice.
\end{proposition}

\begin{proof}
List the elements of $\Cyc_{2n}(x)$ as $(a_1,b_1),(a_2,b_2),\dots,(a_n,b_n)$ such that $a_1<a_2<\dots<a_n$.
Define $\phi : \Cyc_{2n}(x) \to [n]$ by setting $\phi(a_i,b_i) = i$
and for $u  \in \cAfpf(x)^{-1}$
 let
\[
\Phi(u) = [ \phi(u_1,u_2), \phi(u_3,u_4),\dots,\phi(u_{2n-1},u_{2n})] \in S_n.
\]
Since   $u^{-1} \in \cAfpf(x)$  implies  $(u_{2i-1},u_{2u}) \in \Cyc_{2n}(x)$ for all $i \in [n]$
by Corollary \ref{atomSn-fpf-cor}(a), it follows that $\Phi$ is  a well-defined and injective
map
$ \cAfpf(x)^{-1} \to S_n$.

We claim that $\Phi$ is a poset isomorphism 
from $\cAfpf(x)^{-1} $ to a  lower interval in the right weak order $\leq_R$ on $S_n$. 
It is clear that  if $v$ covers $u$ in $\prec_{\cAfpf}$  then $\Phi(u) <_R \Phi(v)$.
Suppose conversely that $v' = \Phi(v )$ for some $v \in\cAfpf(x)^{-1}$  and that $v'$ covers $u' \in S_n$ in $<_R$. To prove our claim, it suffices to check that $u' = \Phi(u)$ for some $u \in   \cAfpf(x)^{-1}$
with $u \prec_{\cAfpf} v$.
To this end, let $i \in [n-1]$ be the unique index such that $v'_i > v'_{i+1}$ and $u'=v's_i$,
and define $a,b,c,d \in [2n]$ such that 
$[v_{2i-1},v_{2i},v_{2i+1},v_{2i+2}] = [b,c,a,d].$
 It is enough to check that $a<b<c<d$. Since $\phi(b,c) = v'_i > v'_{i+1}=\phi(a,d)$, we must have $a<b$, and since both $(a,d)$ and $(b,c)$ belong to $\Cyc_{2n}(x)$  by Corollary \ref{atomSn-fpf-cor}(a), it holds that $a<d$ and $b<c$.
We cannot have  $a<b<d<c$ or $a<d<b<c$ since if either inequality holds then
Lemma \ref{equiv2-lem} implies that there exists $w \in S_{2n}$ with  $v \sim_{\DemAfpf} w $ but $\ell(w) < \ell(v)$, contradicting via Theorem \ref{equiv2-thm} the fact that $v^{-1} \in \cAfpf(x)$.
Hence it must hold that  $a<b<c<d$, as desired.
We  conclude that $\(\cAfpf(x)^{-1},\prec_{\cAfpf}\) $ is a graded lattice by well-known properties of the right weak order on a Coxeter group; see \cite[Corollary 3.2.2]{CCG}. 
\end{proof}

One can check that if $w=[w_1,w_2,\dots,w_n] \in S_n$ is any permutation, then 
the interval bounded above by $w$ in $(S_n,\leq_R)$ is isomorphic via the  map $\Phi$ in the preceding proof  to $\(\cAfpf(y)^{-1},\prec_{\cAfpf}\)$  for the fixed-point-free involution 
$y = (w_1,n+1)(w_2,n+2)\cdots(w_n,2n)$.

\section{Braid relations for involution words}\label{braid-sect}

It is a special case of \emph{Matsumoto's theorem} (see Theorem \ref{matsumoto}) that the set of reduced words $\cR(w)$ for each $w \in S_n$ is an equivalence class under the relation generated by equivalences of the form
$(\dots,s_i,s_{i+1},s_i,\dots) \sim (\dots,s_{i+1},s_i,s_{i+1},\dots)
$
and
$
(\dots,s_i,s_j,\dots)\sim (\dots,s_j,s_i,\dots)
$
for $i \in [n-2]$ and $j \in [n-1]$ with $|i-j| > 1$. Refer to these as the \emph{braid relations} for $S_n$.

As noted in the introduction, there is an appealing version of this fact for involution words in type $A$, first noted in recent work of
 Hu and Zhang  \cite[Theorem 3.1]{HuZhang}.
Their proof of the following theorem derives from an intricate, direct argument, but one can also obtain  this statement as an easy corollary of  our results.

\begin{theorem}[Hu and Zhang \cite{HuZhang}] \label{braid1-thm} Let $x \in \I(S_n)$. Then $\hat\cR(x)$ is an equivalence class under the relation  generated by 
the   equivalences of the form
$(s_i,s_{i+1},\dots) \sim (s_{i+1},s_i,\dots)$  for $i \in [n-2]$
together with
the braid relations for $S_n$. \end{theorem}

\begin{proof}
By Matsumoto's theorem and Lemma \ref{equiv1-lem}, the union of the sets $\cR(w)$ over the permutations $w$ in any  interval in $(S_n,\preceq_\cA)$ is spanned and preserved by the relations $(\dots,s_i,s_{i+1}) \sim (\dots, s_{i+1},s_i)$  for $i \in [n-2]$
together with
the braid relations. The  result now follows from Theorem \ref{prec1-thm}.
\end{proof}

There is  a nice version of this theorem for fixed-point-free involutions, which  does not seem to have been previously observed  in the literature.

\begin{theorem}\label{braid2-thm} Let $x \in \Ifpf(S_{2n})$. Then $\cRfpf(x)$ is an equivalence class under the relation  generated by 
the  equivalences of the form
$
(s_{2i},s_{2i-1},\dots) \sim (s_{2i},s_{2i+1},\dots)
$
for $i \in [n-1]$
together with the braid relations  for $S_{2n}$.
\end{theorem}

\begin{proof}
The result follows from Lemma \ref{equiv2-lem} and Theorem \ref{prec2-thm} exactly as in the previous proof.
\end{proof}

The goal of the rest of this section is to investigate the extent to which these results generalize to arbitrary Coxeter groups. We will find that the situation in type $A$ is quite special: 
in the generic case, one has to include many additional ``prefix-dependent'' relations to define an equivalence which spans and preserves $\hat \cR_*(x)$.

For the duration of this section we fix an arbitrary twisted Coxeter system $(W,S,*)$. Recall that $\Mon{S}$ denotes the set of all finite sequences of elements of $S$.
Given $s,t \in S$, we write $m(s,t)$ for the order of  $st$ in $W$.
We first review the general statement of Matsumoto's theorem. 


\begin{definition}\label{mat-def}
Define $\sim_W$ as the transitive  closure of the relation on $\Mon{S}$ 
with
\be\label{braid-def}
 ( a_1,\dots,a_j, \underbrace{s,t,s,t,\dots}_{m\text{ terms}}, b_1,\dots,b_k )
\sim_W
 (a_1,\dots,a_j, \underbrace{t,s,t,s,\dots}_{m\text{ terms}}, b_1,\dots,b_k )
\ee
whenever  $s,t,a_i,b_i \in S$ and $m=m(s,t)$.
\end{definition}
We   refer to the relations \eqref{braid-def} generating $\sim_W$ in this definition as the \emph{braid relations} for $W$.

\begin{theorem}[See \cite{GP}] \label{matsumoto}
If $w \in W$ then $\cR(w)$ is an equivalence class in $\Mon{S}$ under $\sim_W$.
\end{theorem}

For a proof of this result, see, for example, \cite[Theorem 1.2.2]{GP}. 
The proof in \cite{GP} depends on the following basic lemma (given as \cite[Lemma 1.2.1]{GP}), which will be of use later. Recall that we write $W_J$ for the parabolic subgroup generated by $J\subset S$,
and that $y<w$ implies $\ell(y) < \ell(w)$.

\begin{lemma}[See \cite{GP}]\label{matsu-lem}
Let $w \in W$ and $s,t \in S$. When $m(s,t) < \infty$, let $\Delta$ denote the longest element of the dihedral subgroup $W_{\{ s,t\}} = \langle s,t\rangle$ and define $x,y \in W$  such that $w = x\Delta = \Delta y$.
\ben
\item[(a)] If $ws < w$ and $wt < w$, then $m(s,t) < \infty$ and $\ell(w) = \ell(x) + \ell(\Delta)$.

\item[(b)] If $sw < w$ and $tw < w$, then $m(s,t) < \infty$ and $\ell(w) = \ell(\Delta) + \ell(y)$.
\een
\end{lemma}

To state an analogue of Theorem \ref{matsumoto} for  involution words, we require some new notation.
Fix $s,t \in S$ with $m(s,t) < \infty$ and, as in the lemma, let $\Delta$ be the longest element of $W_{\{s,t\}}$ so that 
\be\label{delta}\Delta = stst\cdots = tsts\cdots\quad\text{where both products have $m(s,t)$ factors.}\ee
 Note that if $\theta$ is a permutation of $J=\{s,t\}$, then $\theta$ extends uniquely to a diagram automorphism of the Coxeter system   $(W_J,J)$, which we denote by the same symbol. In this case $ \Delta =\Delta^{-1} =  \theta(\Delta)$, so the set $\hat \cR_\theta(\Delta)$ as well as the length $\ellhat_\theta(\Delta)$ (which is the length of any word in $\hat \cR_\theta(\Delta)$) is well-defined.

\begin{definition}
Let  $s$, $t$, and $\Delta$ be as above.
Given an arbitrary map $\theta: W \to W$,  define
\[ m(s,t;\theta) \omdef=\begin{cases}
 \ellhat_\theta(\Delta)&\text{if $\theta(\{s,t\}) = \{s,t\}$} \\
\ell(\Delta)&\text{if $\theta(\{s,t\}) \neq \{s,t\}$}.\end{cases}
\]
\end{definition}
Note that  $m(s,t; \theta) \leq m(s,t)$ since one always has $\ellhat_*(x) \leq \ell(x)$. We may express $m(s,t;\theta)$ solely in terms of $m(s,t)$ and the action of $\theta$ on $\{s,t\}$ in the following way.

\begin{proposition}\label{mstar} Let $s,t \in S$ with $m(s,t) < \infty$. 
Then
\[  m(s,t;\theta) = \begin{cases}  \tfrac{1}{2} m(s,t)+\tfrac{1}{2} &\text{if $m(s,t)$ is odd and $\theta(\{s,t\})=\{s,t\}$} \\
  \tfrac{1}{2}m(s,t)+1 &\text{if $m(s,t)$ is even and $\theta(s) = s$ and $\theta(t)=t$} \\
   \tfrac{1}{2} m(s,t)&\text{if $m(s,t)$ is even and $\theta(s)=t$ and $\theta(t)=s$}
   \\
   m(s,t)&\text{if $\theta(\{s,t\})\neq \{s,t\}$}.
    \end{cases}
    \]
\end{proposition}

\begin{proof}
The formula here derives from a straightforward calculation which we leave to the reader.
\end{proof}

Let $\a=(a_1,\dots,a_j) \in \Mon{S}$ and recall from \eqref{delta*} that  $\hat\delta_*(\a) = \idem{\idem{\idem{1} {a_1}}{ \dots}}{ a_j} \in \I_*$.
In the definition which follows, we write  $\theta_\a \in \Aut(W)$ for the automorphism  
$
\theta_\a  : g\mapsto (ugu^{-1})^*$
where
$
u = \hat\delta_*(\a).
$
For involution words, we  have these analogues of Definition \ref{mat-def} and Theorem \ref{matsumoto}:

\begin{definition}\label{d:ibraid}
Define $\sim_{\I_*}$ as the transitive  closure of the relation on $\Mon{S}$ with
\be\label{invbraid}
 ( a_1,\dots,a_j, \underbrace{s,t,s,t,\dots}_{m\text{ terms}}, b_1,\dots,b_k )
\simstar
( a_1,\dots,a_j, \underbrace{t,s,t,s,\dots}_{m\text{ terms}}, b_1,\dots,b_k )
\ee
whenever  $s,t,a_i,b_i \in S$  
and $m=m(s,t;\theta_\a)$ for $\a = (a_1,\dots,a_j)$.
\end{definition}

\begin{theorem}\label{matsumoto-twist}
If $w \in \I_*$ then $\hat \cR_*(w)$ is an equivalence class in $\Mon{S}$ under  $\sim_{\I_*}$.
\end{theorem}

Before proving the theorem, it is helpful to   discuss the definition of $\sim_{\I_*}$ in more concrete detail. 
For this purpose, call the relations \eqref{invbraid} generating $\sim_{\I_*}$ the \emph{involution braid relations} of $W$ relative to $*$.
 We note the following corollary.

\begin{corollary}
Any braid relation between two involution words is an involution braid relation.
\end{corollary}

\begin{proof}
Let $\x$ and $\y$ be the two sides of \eqref{braid-def} and assume both are involution words. The only way that the   braid relation $\x \sim_W \y$ can fail
to be an involution braid relation is if  $m(s,t;\theta_\a) < m(s,t)$ where $\a=(a_1,\dots,a_j)$. But in this case we would have $\x \simstar \z$ where $\z$ is a word with two equal adjacent letters and therefore not an involution word, contradicting Theorem \ref{matsumoto-twist}.
\end{proof}

Thus, when restricted to involution words, the involution braid relations for $W$ consist of the usual braid relations together with certain ``truncated braid relations'' whose size depends on a  prefix $\a$. The involution braid relations which are not braid relations in the ordinary sense include, for example, all  relations of the form 
\be \label{invbraid2} (\underbrace{s,t,s,t,\dots}_{m(s,t;*)\text{ terms}},b_1,\dots,b_k) \simstar (\underbrace{t,s,t,s,\dots}_{m(s,t;*)\text{ terms}},b_1,\dots,b_k)\ee
for $s,t,b_i \in S$;
these are the truncated braid relations induced by the empty prefix $\a = \emptyset$.
\begin{remarks} The situation for Coxeter groups of type $A$ is instructive.
\begin{itemize}
\item When $W=S_n$ and $*=\id$, Theorem \ref{braid1-thm} asserts that the relations \eqref{invbraid2} plus the ordinary braid relations generate $\sim_{\I_*}$.
This is  surprising, as $S_n$ has many other involution braid relations: for example,  
$ (s_{i_1},s_{i_2},\dots,s_{i_l}, s_j) \simstar (s_{i_1},s_{i_2},\dots,s_{i_l}, s_k) $
is  an involution braid relation for $S_n$ relative to $*=\id$ whenever 
$(s_{i_1},s_{i_2},\dots,s_{i_l})$ is an involution word for a permutation with cycles $(j,k)$ and $(j+1,k+1)$ where $|j-k|>1$.

\item When $W=S_n$ and $*$ is the nontrivial diagram automorphism, the involution braid relations of the form \eqref{invbraid2}  are given by
\[\label{invbraid3}
\left\{\ba
(s_i,\dots) &\simstar (s_{n-i},\dots)&&\text{for $i \in[n-1]$ with $n\neq 2i+1$} \\ 
(s_i,s_{i+1},\dots)&\simstar (s_{i+1},s_i,\dots)&&\text{when $n=2i+1$}.
\ea\right.\]
In contrast to the untwisted case, these relations plus the ordinary braid relations do not always span  the sets $\hat\cR_*(x)$, as one can check by considering the longest permutation in $S_4$. 
\end{itemize}
Even when $*=\id$ (e.g., when $(W,S)$ is of type $B$),
the  relations   \eqref{invbraid2} can fail to  span $\hat \cR(x)$.
\end{remarks}

\begin{proof}[Proof of Theorem \ref{matsumoto-twist}]
The proof consists of two steps. First, we will  show that if  $\x \in \hat\cR_*(z)$ for some $z \in \I_*$ and $\y \in \Mon{S}$ is such that $\x \simstar \y$, then $\y \in \hat\cR_*(z)$.
We will then show that if $\x,\y\in \hat\cR_*(z)$, then $\x \simstar \y$.

For the first step, it suffices to take  $\x$ and $\y$ to be the left and right sides of \eqref{invbraid}, respectively, and show that if $\x$ is an  involution word, then $\y$ is also.
Let $u = \hat\delta_*(\a)$ and write $\theta=\theta_\a$ for the map $g\mapsto (ugu^{-1})^*$.
Since $\x$ and $\y$ have the same length, it is enough to check  that both words have the same image under the map $\hat\delta_*:\Mon{S} \to \I_*$.
For this, we may assume without loss of generality that $k=0$ in \eqref{invbraid}.
After this reduction it is straightforward, using the usual braid relations
and the fact that 
if $\theta$ permutes $\{s,t\}$ then $us = \theta(s)^*u$ and $ut = \theta(t)^*u$,
to check that 
indeed
 $\hat\delta_*(\x) = \hat\delta_*(\y) \in \{  \Delta^* u \Delta, \Delta^*  u s\Delta, \Delta^*  u t\Delta, x\Delta\}$ where $\Delta$ is the longest element of $W_{\{ s,t\}}$.

For the second step of the proof, we proceed by induction.
Suppose $\x = (s_1,\dots,s_k)$ and $\y = (t_1,\dots,t_k)$ are two  involution words for some $z \in \I_*$. 
If $ k =0$, then $\x = \y$ so $\x \simstar \y$.
Assume $k>0$ and that $\hat\cR_*(v)$ is an equivalence class under $\simstar$ whenever $v \in \I_*$ has $\ellhat_*(v) < k = \ellhat_*(z)$.
If $s_k = t_k$ then our hypothesis implies $(s_1,\dots,s_{k-1}) \simstar (t_1,\dots,t_{k-1})$ in which case it is clear by definition that $\x \simstar \y$.
Let $s=s_k$ and $t=t_k $ and assume $s\neq t$. Then both $s,t \in \DesR(z)$, so by Lemma \ref{matsu-lem}(a) we have $m(s,t) < \infty$, and if $\Delta$ is the longest element of $W_{\{s,t\}}$ as in \eqref{delta}, then we can write $z = w\Delta$ for $w \in W$ with $\ell(z) = \ell(w) + \ell(\Delta)$. There are four cases to consider according to whether $s^*$ and $t^*$ are left descents of $x$. 
In each case we deduce that $\x \simstar \y$ as follows:
\begin{itemize}
\item Suppose $s^*w < w$ and $t^*w < w$. By Lemma \ref{matsu-lem}(b), we can write $w = \Delta^*v$ for  $v \in W$ with $\ell(w) = \ell(\Delta^*)+\ell(v)$. Thus $z = \Delta^* v \Delta$ and $\ell(z) = \ell(\Delta^*) + \ell(v) + \ell(\Delta)$. Since $\Delta=\Delta^{-1}$, it follows that $v \in \I_*$ and   if $(r_1,\dots,r_j) \in \hat\cR_*(y)$ then both 
\[ (r_1,\dots,r_j,\underbrace{\dots,t,s,t,s}_{m(s,t)\text{ terms}})\qquand 
(r_1,\dots,r_j,\underbrace{\dots,s,t,s,t}_{m(s,t)\text{ terms}})
 \]
 are  involution words for $z$. These words are equivalent under $\simstar$ by definition, as we have $m(s,t) = m(s,t;\theta)$  since otherwise the involution words would not be reduced. On the other hand, the left word is equivalent  to $\x$ and the right to $\y$ by  induction, since their respective last elements coincide.
We conclude by transitivity that $\x \simstar \y$.

\item Suppose $s^*w > w$ and $t^*w > w$. 
Since $zs<z$ and $z^* = z^{-1}$, we have $s^*z < z$. As $s^*w \not < w$, it  follows by the exchange principle that $s^*w \in \{ws,wt\}$. By identical reasoning, we must have $t^*w \in \{ws,wt\}$,
so if $\theta \in \Aut(W)$ is the map $g \mapsto (wgw^{-1})^*$ then $\theta$ permutes the set $\{s,t\}$. Therefore $\theta(\Delta) = \Delta$, so  
since $z^* = w^* \Delta^* = \Delta w^{-1} = z^{-1}$ we deduce that 
$\Delta =\theta(\Delta) = (w\Delta w^{-1})^*= \Delta (w^*w)^{-1}.$
Thus   
$w \in \I_*$. 
Now, if  $(r_1,\dots,r_j) \in \hat\cR_*(w)$, then since $ws = \theta(s)^*w$ and $wt = \theta(t)^*w$,  
 both
\[ (r_1,\dots,r_j,\underbrace{\dots,t,s,t,s}_{m(s,t;\theta)\text{ terms}})\qquand 
(r_1,\dots,r_j,\underbrace{\dots,s,t,s,t}_{m(s,t;\theta)\text{ terms}})
 \]
 are  involution words for $z$. These two words are equivalent under $\simstar$, and the left word is equivalent  to $\x$ and the right to $\y$ by  induction.
We again conclude that $\x \simstar \y$.

\item Suppose $s^*w < w$ and $t^*w > w$. As in the previous case, it follows that $t^*w=wr$ for some $r \in \{s,t\}$, and so  we have  $s^*(t^*w) = (s^*w)r< wr=t^*w$ and $t^*(t^*w)=w<t^*w$. 
By Lemma \ref{matsu-lem}(b), therefore, we can  write $w =(t \Delta)^*v$ for  $v \in W$ with $\ell(w) = \ell((t\Delta)^*)+\ell(v)$. Thus $z = (t\Delta)^* v \Delta$ and $\ell(z) = \ell((t\Delta)^*) + \ell(v) + \ell(\Delta)$. Let $u = s$ if $m(s,t)$ is odd and let $u=t$ if $m(s,t)$ is even. Since $t^*w=wr$ it follows that $u^*v = vr$ and since  $z^*=z^{-1}$ it follows that
$ uv^* = v^{-1} u^*$. From these   identities we deduce that $vr \in \I_*$. This implies in turn that $r=v$, since otherwise the identity
\[ z = (t\Delta)^* (vr) (r\Delta) = (z^*)^{-1} = (\Delta r)^* (vr) (\Delta t) = (\Delta ru)^* v (\Delta t)\]
leads to the contradiction
$\ell(z) = 2m(s,t) - 1 + \ell(v)  \leq \ell(\Delta ru) + \ell(v) + \ell(\Delta t) = 2m(s,t) -3 + \ell(v)$.
Thus   $r^*v = u^*v = vr$,
so    $v \in \I_*$; moreover,  if $(r_1,\dots,r_j) \in \hat\cR_*(v)$ then both
\[ (r_1,\dots,r_j,\underbrace{\dots,t,s,t,s}_{m(s,t)\text{ terms}})\qquand 
(r_1,\dots,r_j,\underbrace{\dots,s,t,s,t}_{m(s,t)\text{ terms}})
 \]
 are  involution words for $z$.
As in the previous cases, these words are equivalent   to each other by definition and to $\x$ and $\y$ respectively by induction, so $\x \simstar \y$.

\item If $s^*w > w$ and $t^*w < w$, then one shows that $\x \simstar \y$ by the same argument as in the previous case, with the roles of $s$ and $t$ interchanged.

\end{itemize}
We conclude from this case analysis that $\sim_{\I_*}$ spans $\hat\cR(z)$, which completes the proof.
\end{proof}

We give one application of  Theorem \ref{matsumoto-twist} to conclude this section.
An element $w \in W$  is \emph{fully commutative} 
if in every reduced word $(s_1,s_2,\dots,s_k) \in \cR(w)$,
there are no indices $0 \leq i \leq k-m$ such that $(s_{i+1},\dots,s_{i+m}) = (s,t,s,t,\dots)$ for some $s,t \in S$ with $m(s,t) =m> 2$.
Thus, an element is fully commutative if and only if any of its reduced words can be obtained from any other by successively switching pairs of adjacent commuting generators.
Stembridge \cite{Stem}   first 
introduced this terminology, though some notable results about  fully commutative elements  appear earlier in the literature.
 Billey, Jockusch, and Stanley  \cite[Theorem 2.1]{BJS}  classified the fully commutative element in the symmetric group: these  are precisely the permutations which are 321-avoiding.
Combining this fact with Corollary \ref{prec1-cor} gives the following:

\begin{corollary} If $x \in \I(S_n)$ then $|\cA(x)| = 1$ if and only if $x$ is fully commutative.
\end{corollary}

This statement does not generalize to arbitrary Coxeter groups: outside of type $A$, it can occur that $|\cA_*(x)| = 1$ without $x \in \I_*$ being fully commutative. One direction of the  corollary does (nearly) hold in general, however, as we show  in the next proposition.


\begin{proposition}\label{fc-prop}
Assume $m(s,s^*) \neq 2$ for all $s \in S$. If $x \in \I_*$ is fully commutative, then there exists a  fully commutative element $w \in W$ such that $\hat \cR_*(x) = \cR(w)$ and $\cA_*(x) = \{w\}$. \end{proposition}

\begin{remark}
Note that one has $m(s,s^*) \neq 2$ for all $s \in S$, for example, whenever $*=\id$ is the identity automorphism.
The result is false without this assumption, for if $s \in S$ is such that $ss^* = s^*s \neq 1$ then
  $ss^*  \in \I_*$ is fully commutative but  $\cA_*(ss^*) = \{s,s^*\}$.
\end{remark}

\begin{proof}
We denote the concatenation of two words $\textbf{u},\textbf{v} \in \Mon{S}$ by $\textbf{u}\textbf{v}$. Fix $x \in \I_*$.
Suppose $|\cA_*(x)| > 1$; we will  show that $x$ is not fully commutative. 
By Theorems \ref{matsumoto} and \ref{matsumoto-twist},
$x$ must have an involution word of the form $\a\b\c$ where $\a,\b,\c\in \Mon{S}$ are such that
\begin{itemize}
\item[(i)] $\a \in \hat \cR_*(u)$ for some twisted involution $u\in \I_*$.
\item[(ii)] $\b$ is an $m$-element sequence of the form $(s,t,s,\dots)$, where  $s,t \in S$ are distinct simple generators such that  $\theta_{\a}(\{s,t\}) =\{s,t\}$ and   $m=m(s,t;\theta_{\a}) < m(s,t) <\infty$. 
\end{itemize}
Let $v = \hat\delta_*(\a\b)$ with $\hat\delta_*$ given as in \eqref{delta*}; note that $\hat\delta_*(\a) = u$, that $\theta_{\a}$ in (ii) is the map $ g \mapsto (ugu^{-1})^*$, and that by definition $v \leq_{T,*} x$.
 It suffices to show that $v$ is not fully commutative,
since the set of fully commutative elements in $W$ is a lower set under both the left and right weak orders \cite[Proposition 2.4]{Stem}, and hence also under $\leq_{T,*}$.

To this end, let $\Delta$  denote   the longest element of  $W_{\{s,t\}}$ and note  that since $\theta_{\a}(\{s,t\}) =\{s,t\}$ and $\a\b \in \hat\cR_*(v)$ by construction, we have
\[v = \idem{u} { \underbrace{\idem{\idem{\idem{s}{t}}{s}}{\cdots}}_{m(s,t;\theta_{\a}) \text{ terms}}} = u\Delta \qquand \ell(v) = \ell(u) + \ell(\Delta).\]
If $m(s,t) = \ell(\Delta)>2$ then $v$ is evidently not fully commutative, so assume $m(s,t) = 2$. 
Since  $0<m(s,t;\theta_\a) < m(s,t)$ by hypothesis, we must  have $m(s,t;\theta_\a) = 1$, so it follows by Proposition \ref{mstar}    
that $\theta_\a$ is the nontrivial permutation of $\{s,t\}$, and  therefore 
$ \ell(v) = \ell(u) + 2 $ and $ v = ust = s^*us = t^*ut.$
But in any reduced word of a fully commutative element, a fixed generator must appear the same number of times, so $v = ust = s^*us$ can only be fully commutative if $s^*=t$. As $s^* \neq t$ since  $m(s,s^*) \neq 2=m(s,t)$,
we deduce that  $v$ is not fully commutative, as required.
\end{proof}


\begin{thebibliography}{99}



\bibitem{BJS} S. C. Billey, W. Jockusch, R. P. Stanley, Some Combinatorial Properties of Schubert Polynomials, \emph{J. Algebr. Combin.} \textbf{2} (1993), 345--374.

\bibitem{CCG} A. Bj\"orner and F. Brenti, \emph{Combinatorics of Coxeter groups}, Graduate Texts in Maths. 231. Springer, New York, 2005

\bibitem{Brion98} M. Brion, The behaviour at infinity of the Bruhat decomposition, \emph{Comment. Math. Helv.} \textbf{73}(1) (1998), 137--174.

\bibitem{CJ} M. B. Can and M. Joyce, Weak Order on Complete Quadrics, \emph{Trans. Amer. Math. Soc.} \textbf{365} (2013), no. 12, 6269--6282.

\bibitem{CJW} M. B. Can, M. Joyce, and B. Wyser, Chains in Weak Order Posets Associated to Involutions, 
\emph{ J. Combin. Theory Ser. A} \textbf{137} (2016), 207--225.

\bibitem{CEHKN} J. Cassaigne, M. Espie, D. Krob, J.-C. Novelli, F. Hivert, The Chinese Monoid, \emph{International
J. Algebra and Comp.} \textbf{11} no. 3 (2001), 301--334.








\bibitem{DuchampKrob} G. Duchamp and D. Krob, Plactic-growth-like monoids, in ``Words, languages and combinatorics II,'' Kyoto, Japan, 25-28 August 1992, M. Ito, H. J\"urgensen, Eds., 124--142, World Scientific, 1994.







\bibitem{FH} W. N. Franzsen and R. B. Howlett, Automorphisms of Coxeter groups of rank three, \emph{Proc. Amer. Math. Soc.} \textbf{129} (2001), 2607--2616.




\bibitem{GP} M. Geck and G. Pfeiffer, \emph{Characters of finite Coxeter groups and Iwahori-Hecke algebras,} Oxford University Press, 2000.



 


 
 
 \bibitem{HMP1} Z. Hamaker, E. Marberg, and B. Pawlowski, Involution words: counting problems and connections to Schubert calculus for symmetric orbit closures, preprint (2015), {\tt arXiv:1508.01823}.
 
 \bibitem{HuZhang} J. Hu and J. Zhang, On involutions in symmetric groups and a conjecture of Lusztig, \emph{Adv. Math.} \textbf{287} (2016) 1--30.
 
 \bibitem{H1} A. Hultman, Fixed points of involutive automorphisms of the Bruhat order,
\emph{Adv. Math.} \textbf{195} (2005), 283--296.

\bibitem{H2} A. Hultman, The combinatorics of twisted involutions in Coxeter groups,
\emph{Trans. Amer. Math. Soc.} \textbf{359} (2007), 2787--2798.

\bibitem{H3} A. Hultman, Twisted identities in Coxeter groups, \emph{J. Algebr. Combin.} \textbf{28} (2008), 313--332.



\bibitem{Hu} J. E. Humphreys, \emph{Reflection groups and Coxeter groups}, Cambridge University Press, Cambridge, 1990.


\bibitem{Incitti1} F. Incitti, The Bruhat Order on the Involutions of the Symmetric Group, \emph{J. Algebr. Combin.} \textbf{20} (2004), 243--261.

\bibitem{Incitti2} F. Incitti, Bruhat order on the involutions of classical Weyl groups, \emph{Adv. Appl. Math.} \textbf{37} (2006), 68--111.



\bibitem{CM2} J. Jaszu\'nska and J. Okni\'nski, Structure of Chinese algebras, \emph{J. Algebra}, \textbf{346} (2011), 31--81.








\bibitem{KM} A. Knutson and E. Miller, Subword complexes in Coxeter groups, \emph{Adv. Math.} \textbf{184}(1) (2004), 161--176.


 






\bibitem{LPRW} S. Linton, J. Propp, T. Roby, and J. West, Equivalence classes of permutations under various relations generated by constrained transpositions,
\emph{J. Integer Seq.} \text{15} (2012), Article 12.9.1.


\bibitem{LV2} G. Lusztig, A bar operator for involutions in a Coxeter group, \emph{Bull. Inst. Math. Acad. Sinica (N.S.)} \textbf{7} (2012), 355--404.


\bibitem{Lu2015} G. Lusztig, An involution based left ideal in the Hecke algebra, preprint (2015), {\tt arXiv: 1507.02263v4}.

\bibitem{LV1} G. Lusztig and D. A. Vogan, Hecke algebras and involutions in Weyl groups, \emph{Bull. Inst. Math. Acad. Sinica (N.S.)} \textbf{7} (2012), 323-354.








\bibitem{EM1} E. Marberg, Positivity conjectures for Kazhdan-Lusztig theory on twisted involutions: the universal case, \emph{Represent. Theory} \textbf{18} (2014), 88--116. 





\bibitem{RV} E. M. Rains and M. J. Vazirani, Deformations of permutation representations of Coxeter groups, \emph{J. Algebr. Comb.} \textbf{37} (2013), 455--502. 



\bibitem{RichSpring} R. W. Richardson and T. A. Springer, The Bruhat order on symmetric varieties, \emph{Geom. Dedicata} \textbf{35} (1990), 389--436. 

\bibitem{RichSpring2} R. W. Richardson and T. A. Springer, Complements to: The Bruhat order on symmetric varieties, \emph{Geom. Dedicata} \textbf{49} (1994), 231--238. 

\bibitem{OEIS} N. J. A. Sloane, editor (2003), The On-Line Encyclopedia of Integer Sequences, published electronically at 
\url{http://www.research.att.com/÷njas/sequences/}.



\bibitem{Springer} T. A. Springer, Some results on algebraic groups with involutions, \emph{Advanced Studies in Pure
Math.} \textbf{6}, 525--543, Kinokuniya/North-Holland, 1985.




\bibitem{Stem} J. R. Stembridge, On the fully commutative elements of Coxeter groups, \emph{J. Algebr. Combin.} \textbf{5} (1996), 353--385.


\bibitem{Wyser} B. J. Wyser. $K$-orbit closures on $G/B$ as universal degeneracy loci for flagged vector bundles with symmetric or skew-symmetric bilinear form, \emph{Transform. Groups} \textbf{18} (2013), 557--594.

\bibitem{WY} B. J. Wyser and A. Yong, Polynomials for symmetric orbit closures in the flag variety, \emph{Transform. Groups}, to appear. 



\end{thebibliography}
\end{document}